\documentclass[11pt,a4paper]{article}

\usepackage[T1]{fontenc}
\usepackage[utf8]{inputenc}
\usepackage{lmodern}
\usepackage{amsmath,amssymb,amsthm,mathtools,mathrsfs}
\usepackage{aliascnt}
\usepackage[a4paper,left=28mm,right=28mm,top=27mm,bottom=27mm]{geometry}
\usepackage{microtype}
\usepackage{xcolor}
\usepackage{cite}
\usepackage[colorlinks=true,linkcolor=blue,citecolor=blue,urlcolor=blue]{hyperref}
\newcommand{\doi}[1]{\href{https://doi.org/#1}{\nolinkurl{doi:#1}}}
\usepackage[capitalize,noabbrev]{cleveref}
\AtBeginDocument{}

\numberwithin{equation}{section}
\allowdisplaybreaks[1]
\newtheorem{theorem}{Theorem}[section]
\newaliascnt{proposition}{theorem}
\newtheorem{proposition}[proposition]{Proposition}
\aliascntresetthe{proposition}
\newaliascnt{lemma}{theorem}
\newtheorem{lemma}[lemma]{Lemma}
\aliascntresetthe{lemma}
\newaliascnt{corollary}{theorem}
\newtheorem{corollary}[corollary]{Corollary}
\aliascntresetthe{corollary}
\theoremstyle{definition}
\newaliascnt{assumption}{theorem}
\newtheorem{assumption}[assumption]{Assumption}
\aliascntresetthe{assumption}
\theoremstyle{definition}
\newaliascnt{remark}{theorem}
\newtheorem{remark}[remark]{Remark}
\aliascntresetthe{remark}
\newaliascnt{definition}{theorem}
\newtheorem{definition}[definition]{Definition}
\aliascntresetthe{definition}

\crefname{equation}{equation}{equations}
\crefname{theorem}{theorem}{theorems}
\crefname{proposition}{proposition}{propositions}
\crefname{lemma}{lemma}{lemmas}
\crefname{corollary}{corollary}{corollaries}
\crefname{assumption}{assumption}{assumptions}
\crefname{definition}{definition}{definitions}
\crefname{section}{section}{sections}

\newcommand{\R}{\mathbb{R}}
\newcommand{\Z}{\mathbb{Z}}
\newcommand{\eps}{\varepsilon}

\newcommand{\diver}{\operatorname{div}}
\newcommand{\esssup}{\operatorname*{ess\,sup}}
\newcommand{\norm}[2]{\left\|#1\right\|_{#2}}
\newcommand{\abs}[1]{\left|#1\right|}
\newcommand{\dual}[2]{\left\langle #1,#2\right\rangle}

\newcommand{\bx}{\boldsymbol{x}}
\newcommand{\by}{\boldsymbol{y}}
\newcommand{\be}{\boldsymbol{e}}
\newcommand{\bn}{\boldsymbol{n}}
\newcommand{\bu}{\boldsymbol{u}}
\newcommand{\bv}{\boldsymbol{v}}
\newcommand{\bw}{\boldsymbol{w}}
\newcommand{\bU}{\boldsymbol{U}}
\newcommand{\bz}{\boldsymbol{z}}
\newcommand{\bphi}{\boldsymbol{\phi}}
\newcommand{\bfv}{\boldsymbol{f}}
\newcommand{\bg}{\boldsymbol{g}}

\newcommand{\bdelta}{\boldsymbol{\delta}}

\newcommand{\bM}{\mathbf{M}}
\newcommand{\bI}{\mathbf{I}}
\newcommand{\bW}{\mathbf{W}}

\newcommand{\bQ}{\boldsymbol{Q}}

\newcommand{\cD}{\mathcal{D}}
\newcommand{\cR}{\mathscr{R}}
\newcommand{\cB}{\mathscr{B}}
\newcommand{\cE}{\mathscr{E}}
\newcommand{\cK}{\mathcal{K}}
\newcommand{\cX}{\mathcal{X}}
\newcommand{\cS}{\mathscr{S}}
\newcommand{\cT}{\mathscr{T}}

\newcommand{\Oe}{\Omega_{\eps}}
\newcommand{\He}{\mathcal{H}_{\eps}}
\newcommand{\ke}{\kappa_{\eps}}
\newcommand{\rhoe}{\rho_{\eps}}
\newcommand{\ue}{\boldsymbol{u}_{\eps}}
\newcommand{\te}{\vartheta_{\eps}}
\newcommand{\pe}{p_{\eps}}
\newcommand{\rhozeroe}{\rho_{\eps}^{0}}
\newcommand{\uzeroe}{\boldsymbol{u}_{\eps}^{0}}

\newcommand{\rhof}{\rho}
\newcommand{\uf}{\boldsymbol{u}}
\newcommand{\tf}{\vartheta}
\newcommand{\pf}{p}

\newcommand{\kf}{\kappa_{\mathrm f}}
\newcommand{\ks}{\kappa_{\mathrm s}}
\newcommand{\kmin}{\kappa_{\mathrm{min}}}
\newcommand{\kmax}{\kappa_{\mathrm{max}}}
\newcommand{\mumin}{\mu_{\mathrm{min}}}
\newcommand{\mumax}{\mu_{\mathrm{max}}}
\newcommand{\rhomin}{\rho_{\mathrm{min}}}
\newcommand{\rhomax}{\rho_{\mathrm{max}}}

\title{Quantitative critical homogenization of a three-dimensional non-homogeneous thermoviscous fluid in perforated domains}
\author{Jiaojiao Pan\footnote{School of Mathematics, Nanjing University, Nanjing 210093, China, panjiaojiao.math@gmail.com} \and Luqi Wang\footnote{School of Mathematics, Nanjing University, Nanjing 210093, China, wangluqi@nju.edu.cn}}
\date{}

\begin{document}
\maketitle
\vspace{-1.5em}
\begin{abstract}
In this paper, we consider a three-dimensional non-homogeneous incompressible fluid in a bounded periodically perforated domain at the critical Stokes-capacity scale. The density obeys a transport equation. The viscosity depends on a quasi-static temperature perturbation, which solves a uniformly elliptic transmission problem with insulated outer boundary conditions and zero spatial mean. The limiting momentum equation contains a Brinkman resistance term generated by the critical perforation regime. For each fixed perforated domain, we construct global finite-energy weak solutions and derive a quantitative relative-energy stability estimate with respect to regular solutions of the homogenized system. As long as a regular effective solution exists, the squared relative error between any microscopic weak solution and the effective solution is bounded by the initial mismatch plus $O(\eps^{2})$. For well-prepared data, the density and uncorrected velocity converge at order $O(\eps)$ in $L^\infty(0,T;L^2)$, the temperature at the same order in both $L^\infty((0,T)\times\Omega)$ and $L^2(0,T;H^1)$, and the corrected velocity in $L^2(0,T;H^1)$. The key estimates are an $O(\eps)$ cell-capacity residual in the dual energy norm, an $O(\eps)$ corrector-gradient bound in $L^{6/5}$, and fixed-domain thermal stability. The critical boundary layers retain order-one viscous energy, represented in the limit by the Brinkman dissipation, so the uncorrected velocity gradient does not converge strongly to the effective gradient whenever the effective velocity is nonzero. Finally, we prove local existence of regular effective solutions for smooth data satisfying a finite set of boundary compatibility conditions, which closes the stability argument.
\end{abstract}



\section{Introduction}\label{sec:introduction}

Let $\Omega\subset\R^{3}$ be a bounded domain.  For $\eps>0$, we consider the following microscopic model
\begin{equation}\label{eq:micro-intro}
\left\{
\begin{aligned}
 \partial_t\rhoe+\diver(\rhoe\ue)&=0,
 &&\text{in }(0,T)\times\Oe,\\
 \diver\ue&=0,
 &&\text{in }(0,T)\times\Oe,\\
 \partial_t(\rhoe\ue)+\diver(\rhoe\ue\otimes\ue)
 -\diver\!\bigl(2\mu(\te)\cD(\ue)\bigr)+\nabla\pe
 &=\rhoe\bfv-\te\bg,
 &&\text{in }(0,T)\times\Oe,\\
 -\diver(\ke\nabla\te)&=\widetilde{\ue}\cdot\bg,
 &&\text{in }(0,T)\times\Omega.
\end{aligned}
\right.
\end{equation}
Here $\Oe$ denotes the perforated fluid domain and $\He$ the union of the holes (or solid obstacles), whose precise periodic geometry is specified in \eqref{eq:geometry}. We write
\(
 (\rhoe,\ue,\te,\pe)
\)
for the density, velocity, temperature perturbation and pressure, respectively, and use the conventions $(\nabla\bv)_{ij}=\partial_jv_i$ and $A:B=\sum_{i,j}A_{ij}B_{ij}$. For a vector field $\bv$, define the symmetric gradient and the zero extension of the microscopic velocity by
\[
 \cD(\bv):=\frac12\bigl(\nabla\bv+(\nabla\bv)^{\textup T}\bigr),
 \qquad
 \widetilde{\ue}:=
 \begin{cases}
  \ue & \text{in }\Oe,\\
  \boldsymbol0 & \text{in }\He.
 \end{cases}
\]
Conductivity is piecewise constant,
\[
 \ke=\kf\mathbf 1_{\Oe}+\ks\mathbf 1_{\He},
 \qquad \kf>0,\quad \ks>0,
\]
while $\mu:\R\longrightarrow(0,\infty)$ is the viscosity, $\bfv$ is the prescribed body force and $\bg$ is a conservative field satisfying the assumptions in \eqref{eq:data-assumptions}. The temperature equation in \eqref{eq:micro-intro} is quasi-static and is supplemented by the insulating boundary condition in \eqref{eq:micro-boundary}. Since $\bg=\nabla F$, incompressibility and the zero normal trace imply the Neumann compatibility condition: 
\[
 \int_\Omega \widetilde{\ue}\cdot\bg\,\mathrm d\bx
 =\int_\Omega \widetilde{\ue}\cdot\nabla F\,\mathrm d\bx
 =-\int_\Omega F\,\diver\widetilde{\ue}\,\mathrm d\bx
 +\int_{\partial\Omega}F\,\widetilde{\ue}\cdot\bn\,\mathrm dS=0.
\]
We regard $\te$ as a temperature perturbation measured relative to a fixed reference temperature. Henceforth, we use the shorter term temperature for this normalized perturbation. The condition $\int_\Omega\te(t)\,\mathrm d\bx=0$ fixes the reference level. Since the constitutive law contains $\mu(\te)$, this normalization is part of the coupled model rather than a harmless choice of additive constant.

Let $a_{\eps}$ denote the size of the holes. Following the classical scaling introduced by Allaire~\cite{AllaireI,AllaireII} and the unified formulation developed by Lu~\cite{Lu2020}, we set
\[
 \sigma_{\eps}^{2}=\frac{\eps^{3}}{a_{\eps}}.
\]
For the geometry considered here, $a_{\eps}=\eps^{3}$ and hence $\sigma_{\eps}=1$. At this scale the total Stokes capacity is of order one. Accordingly, the effective momentum balance contains both the macroscopic viscous stress and a Brinkman resistance term.  The associated exterior Stokes problem is
\[
\left\{
\begin{aligned}
 -\Delta\boldsymbol{V}^{i}+\nabla Q^{i}&=\boldsymbol{0},
 &\qquad \diver\boldsymbol{V}^{i}&=0,
 &&\text{in }\R^{3}\setminus\overline{T_{0}},\\
 \boldsymbol{V}^{i}&=\boldsymbol{0},
 &&&&\text{on }\partial T_{0},\\
 \boldsymbol{V}^{i}(\bx)&\to\be_{i},
 &\qquad Q^i(\bx)&\to0,
 &&\text{as }\abs{\bx}\to\infty.
\end{aligned}
\right.
\]
Here $i=1,2,3$, $T_{0}$ is the reference hole and $\{\be_i\}_{i=1}^{3}$ is the canonical basis of $\R^3$.  Writing the corresponding velocity--pressure pairs as $(\boldsymbol V^i,Q^i)$, define the resistance matrix by
\begin{equation}\label{eq:resistance-matrix}
 (\bM_{0})_{ij}
 :=\int_{\R^{3}\setminus\overline{T_{0}}}
 \nabla\boldsymbol{V}^{i}:\nabla\boldsymbol{V}^{j}\,\mathrm d\bx.
\end{equation}
Symmetry of $\bM_0$ follows directly from the definition in \eqref{eq:resistance-matrix}. Moreover, if $\boldsymbol\xi\in\R^3$ and $\boldsymbol V_\xi=\sum_i\xi_i\boldsymbol V^i$, then
\[
 \boldsymbol\xi\cdot\bM_0\boldsymbol\xi
 =\int_{\R^3\setminus\overline{T_0}}|\nabla\boldsymbol V_\xi|^2\,\mathrm d\bx.
\]
Since $\boldsymbol V_\xi$ has zero trace on $\partial T_0$ and tends to $\boldsymbol\xi$ at infinity, the right-hand side vanishes only when $\boldsymbol\xi=\boldsymbol0$. Hence $\bM_0$ is positive definite.

The closest qualitative critical heat-conducting limit is due to Feireisl, Lu and Sun~\cite{FLS}, who derived a Brinkman law for a three-dimensional non-homogeneous heat-conducting fluid with temperature-dependent viscosity and oscillatory conductivity, under homogeneous Dirichlet data for the temperature. The present model instead imposes a thermally insulated homogeneous Neumann condition on the outer boundary. This changes the thermal functional setting: constants belong to the kernel of the elliptic operator, solvability requires a compatibility condition, and coercivity is recovered on the zero-mean subspace through the Poincar\'e--Wirtinger inequality. In our system this compatibility condition follows from the conservative forcing $\bg=\nabla F$ and incompressibility. For the non-homogeneous incompressible Navier--Stokes system without the present thermal coupling, the critical density-dependent Brinkman regime is treated in \cite{Pan2025}. In the compressible low-Mach setting, Bella and Oschmann~\cite{BO2022} studied a critical perforation problem. Quantitative estimates for the linear Stokes--Brinkman homogenization problem in dilute perforated domains were obtained by Jing, Lu and Prange~\cite{JLP}, while related quantitative and inviscid-limit questions were studied in \cite{HJ2024,Hoefer2023}. In the larger-hole regime $1<\alpha<3$, Darcy-type limits were considered in \cite{BOP}. By contrast, a closely related heat-conducting system with $a_\eps=\eps^\alpha$, $\alpha>3$ is considered in \cite{LPW2026}, in which the total Stokes capacity is $O(\eps^{\alpha-3})$ and no Brinkman term survives in the limit. Thus $\alpha=3$ is the critical threshold in which the total Stokes capacity remains of order one.

The present work provides quantitative estimates for this critical regime. Density transport, temperature-dependent viscosity and a critical Stokes restriction operator are combined within a single relative-energy argument. The restriction estimates give an $O(\eps)$ bound for the corrector gradient in $L^{6/5}$, allowing it to be paired with the $L^6$ relative velocity, and the corresponding weighted cell-capacity residual includes the pressure--coefficient commutator and divergence repair required by the variable viscosity. On the thermal side, we establish fixed-domain zero-mean Neumann stability in both $H^1$ and $L^\infty$. The $L^\infty$ bound controls $\mu(\te)-\mu(\tf)$ in $L^\infty$ and closes the variable-viscosity relative-energy estimate. Together with the renormalized density estimate, these bounds yield an $O(\eps^2)$ squared stability remainder and order-$O(\eps)$ convergence for well-prepared data, while the dissipation limit identifies the order-one boundary-layer contribution. The restriction error itself is sharp at order $O(\eps)$ in $L^6(\Omega)$ for nonzero comparison fields, as shown in \Cref{prop:restriction-sharp-scale}. Relative-energy methods for weak--strong stability of non-homogeneous incompressible flows provide complementary fixed-domain background (see, e.g., \cite{CBSV}).

At the macroscopic level, the limiting fields satisfy
\begin{equation}\label{eq:effective-intro}
\left\{
\begin{aligned}
 \partial_t\rhof+\diver(\rhof\uf)&=0,
 &&\text{in }(0,T)\times\Omega,\\
 \diver\uf&=0,
 &&\text{in }(0,T)\times\Omega,\\
 \partial_t(\rhof\uf)+\diver(\rhof\uf\otimes\uf)
 -\diver\!\bigl(2\mu(\tf)\cD(\uf)\bigr)
 +\mu(\tf)\bM_{0}\uf+\nabla\pf
 &=\rhof\bfv-\tf\bg,
 &&\text{in }(0,T)\times\Omega,\\
 -\kf\Delta\tf&=\uf\cdot\bg,
 &&\text{in }(0,T)\times\Omega.
\end{aligned}
\right.
\end{equation}
Here $(\rhof,\uf,\tf,\pf)$ are the limiting density, velocity, temperature and pressure. On $\partial\Omega$, the effective velocity satisfies the no-slip condition, whereas $\tf$ obeys the homogeneous Neumann condition and the zero-mean reference condition in \eqref{eq:effective-boundary}. Since
\[
 \norm{\ke-\kf}{L^p(\Omega)}
 \le C|\He|^{1/p}=O(\eps^{6/p})
 \qquad\text{for }1\le p<\infty,
\]
the conductivity coefficient converges strongly to $\kf$ in $L^p(\Omega)$ for $1\le p<\infty$. Moreover, as the conductivity defect is supported in $\He$, the quantitative estimate in \Cref{lem:thermal} recovers $\kf$ as the limiting thermal coefficient.

\paragraph{Organization of the paper.} \Cref{sec:setting} states the microscopic and effective problems and the main theorem. \Crefrange{sec:corrector}{sec:thermal} establish the restriction and thermal estimates, and \Crefrange{sec:relative}{sec:dissipation} contain the relative-energy argument and dissipation limit. Appendix~\ref{sec:effective-existence} proves local regular solvability of the effective system, while Appendix~\ref{sec:micro-existence} constructs global microscopic finite-energy weak solutions.

\section{Problem formulation and main result}\label{sec:setting}

Let $Y=(-\tfrac12,\tfrac12)^{3}$ and assume that $\Omega$ is a bounded connected domain of class $C^{2,\beta}$, $\beta\in(0,1)$. Fix a simply connected reference hole $T_{0}\subset\subset B(\boldsymbol{0},1/8)$ of class $C^{2,\beta}$ with connected exterior. Set
\begin{equation}\label{eq:geometry}
 \cK_{\eps}
 =\bigl\{\boldsymbol{k}\in\Z^{3}:\eps(\boldsymbol{k}+\overline Y)\subset\subset\Omega\bigr\},
 \qquad
 T_{\eps,\boldsymbol{k}}
 =\eps\boldsymbol{k}+\eps^{3}T_{0},
 \qquad
 \He=\bigcup_{\boldsymbol{k}\in\cK_{\eps}}\overline{T_{\eps,\boldsymbol{k}}},
 \qquad
 \Oe=\Omega\setminus\He,
\end{equation}
where the hole scale is $a_{\eps}=\eps^{3}$, corresponding to the critical exponent $\alpha=3$ in the convention $a_{\eps}=\eps^{\alpha}$. Since the admissible centers lie on the $\eps$-lattice and each hole has volume $O(\eps^9)$,
\begin{equation}\label{eq:volume-holes}
 \#\cK_{\eps}\le C\eps^{-3},
 \qquad
 \abs{\He}\le C\eps^{6}.
\end{equation}
where we used $\#\cK_{\eps}$ to denote the number of the holes. Assume
\begin{equation}\label{eq:coefficients}
 \mu\in W^{1,\infty}(\R),
 \qquad
 0<\mumin\le\mu(z)\le\mumax<\infty,
 \qquad
 \kf>0,\quad \ks>0.
\end{equation}
Set $\kmin=\min\{\kf,\ks\}$ and $\kmax=\max\{\kf,\ks\}$. For the external fields, assume
\begin{equation}\label{eq:data-assumptions}
 \bfv\in W^{1,\infty}\bigl((0,T)\times\Omega;\R^{3}\bigr),
 \qquad
 F\in W^{2,\infty}(\Omega),
 \qquad
 \bg=\nabla F.
\end{equation}
The potential $F$ is time independent, so $\partial_t\bg=\boldsymbol0$. Fix $0<\rhomin<\rhomax$. For the initial density and velocity, assume
\begin{equation}\label{eq:initial-data-assumptions}
 \rhomin\le\rhozeroe\le\rhomax\quad\text{a.e.\ in }\Oe,
 \qquad
 \uzeroe\in L^{2}(\Oe;\R^{3}),
 \qquad
 \diver\uzeroe=0,
 \qquad
 \uzeroe\cdot\bn=0\quad\text{on }\partial\Oe,
\end{equation}
in the weak normal-trace sense. Extend $\rhozeroe$ to $\He$ by the constant $\rhomin$ and denote the extension by $\widehat{\rhozeroe}$.

Write $\bn$ for the outward unit normal of the domain under consideration, and let $\operatorname{Tr}_{D}$ denote the boundary trace taken from the interior of a domain $D$. With the preceding convention for the normalized temperature perturbation, the microscopic boundary conditions and the zero-mean reference condition are
\begin{equation}\label{eq:micro-boundary}
 \ue=\boldsymbol{0}\quad\text{on }(0,T)\times\partial\Oe,
 \qquad
 \ke\nabla\te\cdot\bn=0\quad\text{on }(0,T)\times\partial\Omega,
 \qquad
 \int_{\Omega}\te(t)\,\mathrm d\bx=0.
\end{equation}
The homogeneous Neumann condition represents thermal insulation at the outer boundary. A single temperature field $\te\in H^1(\Omega)$ enforces equality of the fluid and solid traces on every interface $\partial T_{\eps,\boldsymbol k}$. The variational heat equation also gives continuity of the normal heat flux. Let $\bn_{\mathrm f}$ and $\bn_{\mathrm s}$ denote the outward normals of $\Oe$ and $T_{\eps,\boldsymbol k}$, respectively. Then
\[
 \operatorname{Tr}_{\Oe}\te
 =\operatorname{Tr}_{T_{\eps,\boldsymbol k}}\te,
 \qquad
 \kf\nabla\te|_{\Oe}\cdot\bn_{\mathrm f}
 +\ks\nabla\te|_{T_{\eps,\boldsymbol k}}\cdot\bn_{\mathrm s}=0
 \quad\text{on }\partial T_{\eps,\boldsymbol k},
\]
where the first identity is understood in $H^{1/2}(\partial T_{\eps,\boldsymbol k})$ and the flux identity in $H^{-1/2}(\partial T_{\eps,\boldsymbol k})$.

 \begin{definition}[Finite-energy weak solution]\label{def:finite-energy-weak}
A finite-energy weak solution is a triple $(\rhoe,\ue,\te)$ such that $\rhoe$ is a renormalized density in the sense of the transport theory in \cite[Section~II]{DiPernaLions} and the variable-density framework of \cite{Lions}, $\ue$ is a finite-energy weak velocity, and $\te$ satisfies the variational heat equation together with the zero-mean condition in \eqref{eq:micro-boundary}. More precisely,
\[
\begin{gathered}
 \rhomin\le\rhoe\le\rhomax,\qquad
 \te\in L^{2}(0,T;H^{1}(\Omega)),\\
 \ue\in L^{\infty}(0,T;L^{2}(\Oe;\R^{3}))
 \cap L^{2}(0,T;H_{0}^{1}(\Oe;\R^{3})),
 \qquad
 \diver\ue=0\quad\text{in }\mathcal D'((0,T)\times\Oe).
\end{gathered}
\]
The velocity admits a representative in $C_{\mathrm{weak}}([0,T];L^2(\Oe;\R^3))$ with $\ue(0)=\uzeroe$. Extend the density to the fixed domain by transport with the zero-extended velocity and denote the resulting field by $\widehat{\rhoe}$, with initial value $\widehat{\rhozeroe}$. For $b\in C^{1}([\rhomin,\rhomax])$ and $\psi\in C_{c}^{1}([0,T)\times\Omega)$,
\begin{equation}\label{eq:renormalized-continuity}
 \int_{0}^{T}\!\int_{\Omega}
 b(\widehat{\rhoe})
 \bigl(\partial_t\psi+\widetilde{\ue}\cdot\nabla\psi\bigr)\,\mathrm d\bx\,\mathrm dt
 +\int_{\Omega}b(\widehat{\rhozeroe})\psi(0)\,\mathrm d\bx=0.
\end{equation}
For divergence-free $\bphi\in C_{c}^{1}([0,T)\times\Oe;\R^{3})$,
\begin{equation}\label{eq:weak-momentum}
\begin{aligned}
 &\int_{0}^{T}\!\int_{\Oe}
 \Bigl[
   \rhoe\ue\cdot\partial_t\bphi
  +\rhoe\ue\otimes\ue:\nabla\bphi
  -2\mu(\te)\cD(\ue):\cD(\bphi)
  +(\rhoe\bfv-\te\bg)\cdot\bphi
 \Bigr] \,\mathrm d\bx\,\mathrm dt\\
 &\qquad
 +\int_{\Oe}\rhozeroe\uzeroe\cdot\bphi(0)\,\mathrm d\bx=0.
\end{aligned}
\end{equation}
For a.e. $t\in(0,T)$ and $\varphi\in H^{1}(\Omega)$,
\begin{equation}\label{eq:weak-heat}
 \int_{\Omega}\ke\nabla\te(t)\cdot\nabla\varphi\,\mathrm d\bx
 =\int_{\Omega}\widetilde{\ue}(t)\cdot\bg\,\varphi\,\mathrm d\bx.
\end{equation}
Finally, for a.e. $t\in(0,T)$,
\begin{equation}\label{eq:combined-energy}
\begin{aligned}
 &\frac12\int_{\Oe}\rhoe(t)\abs{\ue(t)}^{2}\,\mathrm d\bx
 +\int_{0}^{t}\!\int_{\Oe}2\mu(\te)\abs{\cD(\ue)}^{2}\,\mathrm d\bx\,\mathrm ds
 +\int_{0}^{t}\!\int_{\Omega}\ke\abs{\nabla\te}^{2}\,\mathrm d\bx\,\mathrm ds\\
 &\qquad\le
 \frac12\int_{\Oe}\rhozeroe\abs{\uzeroe}^{2}\,\mathrm d\bx
 +\int_{0}^{t}\!\int_{\Oe}\rhoe\bfv\cdot\ue\,\mathrm d\bx\,\mathrm ds.
\end{aligned}
\end{equation}
\end{definition}

The Neumann compatibility condition is automatic for any admissible velocity: since $\bg=\nabla F$, $\diver\widetilde{\ue}=0$, and $\widetilde{\ue}$ has zero normal trace, integration by parts gives $\int_\Omega\widetilde{\ue}\cdot\bg\,\mathrm d\bx=0$. Testing \eqref{eq:weak-heat} by $\te(t)$ and using \eqref{eq:combined-energy} then gives, for a.e. $t\in(0,T)$,
\[
 \int_{\Omega}\ke\abs{\nabla\te}^{2}\,\mathrm d\bx
 =\int_{\Oe}\te\bg\cdot\ue\,\mathrm d\bx.
\]
Substituting this identity into \eqref{eq:combined-energy}, we obtain the kinetic-energy inequality
\begin{equation}\label{eq:kinetic-energy}
\begin{aligned}
 &\frac12\int_{\Oe}\rhoe(t)\abs{\ue(t)}^{2}\,\mathrm d\bx
 +\int_{0}^{t}\!\int_{\Oe}2\mu(\te)\abs{\cD(\ue)}^{2}\,\mathrm d\bx\,\mathrm ds\\
 &\qquad\le
 \frac12\int_{\Oe}\rhozeroe\abs{\uzeroe}^{2}\,\mathrm d\bx
 +\int_{0}^{t}\!\int_{\Oe}
 (\rhoe\bfv-\te\bg)\cdot\ue\,\mathrm d\bx\,\mathrm ds.
\end{aligned}
\end{equation}

For the effective problem \eqref{eq:effective-intro}, the boundary conditions and normalizations are
\begin{equation}\label{eq:effective-boundary}
 \uf=\boldsymbol{0},
 \qquad
 \partial_{\bn}\tf=0
 \quad\text{on }(0,T)\times\partial\Omega,
 \qquad
 \int_{\Omega}\tf(t)\,\mathrm d\bx=0,
 \qquad
 \int_{\Omega}\pf(t)\,\mathrm d\bx=0.
\end{equation}
The effective initial data are
\begin{equation}\label{eq:effective-initial}
 \rhof(0,\cdot)=\rho_{0},
 \qquad
 \uf(0,\cdot)=\boldsymbol{u}_{0}
 \quad\text{in }\Omega.
\end{equation}
Under the smoothness and finite initial-jet compatibility assumptions of \Cref{thm:effective-local}, there exists $T_*>0$ such that a regular effective solution exists on $[0,T_*]$. Throughout the comparison argument, $T$ is arbitrary with $0<T\le T_*$.

\begin{assumption}[Regular effective solution on the comparison interval]\label{ass:strong}
The effective system \eqref{eq:effective-intro}, with \eqref{eq:effective-boundary} and initial data \eqref{eq:effective-initial}, admits a solution $(\rhof,\uf,\tf,\pf)$ on $[0,T]$ such that
\[
 \rhomin\le\rhof\le\rhomax,
 \qquad
 \rhof\in W^{1,\infty}((0,T)\times\Omega),
\]
\[
 \uf\in W^{1,\infty}(0,T;W^{2,\infty}(\Omega;\R^{3})) ,
 \qquad
 \tf\in L^{\infty}(0,T;W^{2,\infty}(\Omega)),
\]
and
\[
 \pf\in L^{\infty}(0,T;L^{2}(\Omega)),
 \qquad
 \int_{\Omega}\pf(t)\,\mathrm d\bx=0
 \quad\text{for a.e. }t\in(0,T).
\]
\end{assumption}

Assumption~\ref{ass:strong} places both $\uf$ and $\partial_t\uf$ in the domain of the restriction operator and supplies the coefficient bounds used in the comparison argument. Accordingly, the effective momentum equation is interpreted distributionally, and for solenoidal comparison fields only the distributional gradient $\nabla\pf$ appears. Appendix~\ref{sec:effective-existence} then verifies the assumption locally for smooth data satisfying the finite compatibility hierarchy.

For the quantitative comparison, let $\cR_{\eps}$ denote the solenoidal restriction constructed in \Cref{prop:restriction} and define
\[
 \bU_{\eps}:=\cR_{\eps}\uf,
 \qquad
 \boldsymbol{e}_{\eps}:=\ue-\bU_{\eps}\quad\text{in }\Oe,
 \qquad
 q_{\eps}:=\widehat{\rhoe}-\rhof\quad\text{in }\Omega.
\]
The comparison error is measured by the relative energy
\begin{equation}\label{eq:relative-energy}
 \cE_{\eps}(t)
 :=\frac12\norm{q_{\eps}(t)}{L^{2}(\Omega)}^{2}
 +\frac12\int_{\Oe}\rhoe(t)
 \abs{\boldsymbol{e}_{\eps}(t)}^{2}\,\mathrm d\bx,
 \qquad 0\le t\le T.
\end{equation}
At $t=0$, the initial data give
\[
 \cE_{\eps}(0)
 =\frac12\norm{\widehat{\rhozeroe}-\rho_{0}}{L^{2}(\Omega)}^{2}
 +\frac12\int_{\Oe}\rhozeroe
 \abs{\uzeroe-\cR_{\eps}\boldsymbol{u}_{0}}^{2}\,\mathrm d\bx.
\]

\begin{theorem}[Critical quantitative homogenization and stability]\label{thm:main}
Assume the geometry \eqref{eq:geometry} and the hypotheses \eqref{eq:coefficients}--\eqref{eq:initial-data-assumptions}. Let $(\rhoe,\ue,\te)$ be a finite-energy weak solution of the microscopic system \eqref{eq:micro-intro} in the sense of \Cref{def:finite-energy-weak}. Let $(\rhof,\uf,\tf,\pf)$ be an effective solution satisfying \Cref{ass:strong}. With $\cE_{\eps}$ defined by \eqref{eq:relative-energy}, there exist $\eps_{0}>0$ and $C>0$  independent of $\eps$, such that for $0<\eps<\eps_{0}$,
\begin{equation}\label{eq:main-estimate}
\begin{aligned}
 &\esssup_{0\le t\le T}
 \left[
  \norm{q_{\eps}(t)}{L^{2}(\Omega)}^{2}
  +\int_{\Oe}\rhoe(t)\abs{\boldsymbol{e}_{\eps}(t)}^{2}\,\mathrm d\bx
 \right]\\
 &\quad
 +\int_{0}^{T}\norm{\cD(\boldsymbol{e}_{\eps})}{L^{2}(\Oe)}^{2}\,\mathrm dt
 +\int_{0}^{T}\norm{\te-\tf}{H^{1}(\Omega)}^{2}\,\mathrm dt
 \le C\bigl(\cE_{\eps}(0)+\eps^{2}\bigr),
\end{aligned}
\end{equation}
and
\begin{equation}\label{eq:main-convergences}
\begin{aligned}
 &\norm{\widehat{\rhoe}-\rhof}{L^{\infty}(0,T;L^{2}(\Omega))}
 +\norm{\widetilde{\ue}-\uf}{L^{\infty}(0,T;L^{2}(\Omega))}
 +\norm{\te-\tf}{L^{\infty}((0,T)\times\Omega)}\\
 &\qquad
 +\norm{\te-\tf}{L^{2}(0,T;H^{1}(\Omega))}
 +\norm{\ue-\cR_\eps\uf}{L^{2}(0,T;H^{1}(\Oe))}
 \le C\bigl(\cE_{\eps}(0)^{1/2}+\eps\bigr).
\end{aligned}
\end{equation}
If $\cE_{\eps}(0)=O(\eps^{2})$, all five errors in \eqref{eq:main-convergences} are $O(\eps)$. Moreover,
\[
 \norm{\mu(\te)-\mu(\tf)}{L^\infty((0,T)\times\Omega)}
 \le C\bigl(\cE_\eps(0)^{1/2}+\eps\bigr).
\]
\end{theorem}

\begin{remark}[Well-prepared data, sharpness and constants]\label{rem:main-constants}
The well-prepared class is nonempty. Set
\[
 \rhozeroe=\rho_0|_{\Oe},
 \qquad
 \uzeroe=\cR_\eps\boldsymbol u_0,
\]
and retain the convention that $\rhozeroe$ is extended by $\rhomin$ in $\He$. Then \eqref{eq:initial-data-assumptions} is satisfied for all sufficiently small $\eps$, while
\[
 \cE_\eps(0)
 =\frac12\int_{\He}|\rhomin-\rho_0|^2\,\mathrm d\bx
 \le C|\He|
 \le C\eps^6
\]
by \eqref{eq:volume-holes}. Thus the $O(\eps)$ convergence rate applies to an explicit family of admissible initial data. Moreover, \Cref{prop:restriction-sharp-scale} shows that the $L^6(\Omega)$ zero-extension error of the restriction step cannot, for nonzero comparison fields, be improved from $O(\eps)$ to $o(\eps)$. Whether the rate in \Cref{thm:main} is optimal remains open.

The threshold $\eps_0$ depends only on the geometric separation requirements in \eqref{eq:geometry} and is independent of the norms of the effective solution. The constant $C$ may depend on the fixed geometry, the coefficient and force bounds, $T$, $\rhomin$, $\rhomax$ and the norms of the effective solution in \Cref{ass:strong}, but not on $\eps$.
\end{remark}

The proof of \Cref{thm:main} is given in \Cref{sec:relative}. Before turning to the quantitative estimates, we record that the microscopic comparison class is nonempty.

\begin{proposition}[Global microscopic weak existence]
\label{prop:microscopic-existence}
Assume \eqref{eq:geometry} and \eqref{eq:coefficients}--\eqref{eq:initial-data-assumptions}. For fixed sufficiently small $\eps>0$ and finite $T>0$ for which the data satisfy \eqref{eq:data-assumptions}, there exists a finite-energy weak solution $(\rhoe,\ue,\te)$ of \eqref{eq:micro-intro} in the sense of \Cref{def:finite-energy-weak}, with the initial data specified in \eqref{eq:initial-data-assumptions}. Moreover,
\[
 \rhoe\in C([0,T];L^p(\Oe)),\qquad 1\le p<\infty,
 \qquad
 \rhomin\le\rhoe\le\rhomax.
\]
The velocity also has a weakly continuous $L^2$ representative with initial value $\uzeroe$.
\end{proposition}

The uniform a priori bounds used in the quantitative argument are collected in \Cref{lem:uniform} with constants independent of $\eps$, and the compactness construction for \Cref{prop:microscopic-existence} is carried out at fixed $\eps$ in Appendix~\ref{sec:micro-existence}.

\section{Solenoidal restriction at the critical Stokes-capacity scale}\label{sec:corrector}
Throughout this section, we assume the geometry \eqref{eq:geometry}. 
For any Lipschitz domain $U\subset\R^3$, write
\[
 H_{0,\sigma}^1(U)
 :=\{\bz\in H_0^1(U;\R^3):\diver\bz=0\}.
\]
Set
\[
 \cX(\Omega)
 :=W^{1,\infty}(\Omega;\R^{3})\cap H_{0,\sigma}^{1}(\Omega).
\]
We first record the scaled local Stokes estimate used in the truncated-cell argument. Let $(\mathcal P',\mathcal P)$ be a fixed nested pair of local patches. In the interior case, $\overline{\mathcal P'}\subset\mathcal P$ and in the boundary case, the two patches share a $C^2$ physical boundary portion, while $\overline{\mathcal P'}$ stays a positive distance from the artificial part of $\partial\mathcal P$. If $(\boldsymbol h,\pi)$ solves
\[
 -\Delta\boldsymbol h+\nabla\pi=\boldsymbol0,
 \qquad \diver\boldsymbol h=0
\]
in $\mathcal P$, with $\boldsymbol h=\boldsymbol0$ on the physical boundary portion in the boundary-patch case, the interior and homogeneous-Dirichlet Stokes estimates of \cite[Theorems~IV.4.1 and~IV.5.1]{Galdi}, taken with exponent $6$, followed by Sobolev embedding give
\[
 \norm{\boldsymbol h}{L^\infty(\mathcal P')}
 +\norm{\nabla\boldsymbol h}{L^\infty(\mathcal P')}
 +\inf_{c\in\R}\norm{\pi-c}{L^\infty(\mathcal P')}
 \le C\norm{\boldsymbol h}{L^6(\mathcal P)}.
\]
The constant depends only on the fixed reference patches and their $C^2$ geometry. Under the dilation $\by=r\boldsymbol\eta$, with the pressure rescaled as $\Pi(\boldsymbol\eta)=r\pi(r\boldsymbol\eta)$, this becomes
\begin{equation}\label{eq:scaled-stokes-local}
 \norm{\boldsymbol h}{L^\infty(r\mathcal P')}
 +r\norm{\nabla\boldsymbol h}{L^\infty(r\mathcal P')}
 +r\inf_{c\in\R}\norm{\pi-c}{L^\infty(r\mathcal P')}
 \le Cr^{-1/2}\norm{\boldsymbol h}{L^6(r\mathcal P)}.
\end{equation}
In particular, the constant in \eqref{eq:scaled-stokes-local} is independent of the dilation parameter $r$. The infimum over $c$ reflects the local pressure indeterminacy. The global zero-mean normalization of $Q_R^i$ will fix this additive constant after the dyadic matching below. This explicit scaling will be used both on dyadic interior annuli and on the rescaled outer-boundary patches.

With this notation, the restriction operator will combine truncated exterior Stokes correctors, the associated cell-capacity distribution, and a divergence repair for slowly varying macroscopic fields. The scale $a_\eps=\eps^3$ is precisely the critical Stokes-capacity regime analyzed in \cite{JLP}.

\begin{lemma}[Exterior decay and truncated resistance]\label{lem:truncated-stokes}
Let $(\boldsymbol V^i,Q^i)$, $i=1,2,3$, be the exterior Stokes fields defining $\bM_0$ in \eqref{eq:resistance-matrix}, normalized by $Q^i(\by)\to0$ as $|\by|\to\infty$ and set $\boldsymbol Z^i=\boldsymbol V^i-\be_i$. Then
\[
 |\boldsymbol Z^i(\by)|\le C|\by|^{-1},\qquad
 |\nabla\boldsymbol Z^i(\by)|+|Q^i(\by)|\le C|\by|^{-2},
 \qquad |\by|\ge2.
\]
For $R\ge4$, let $(\boldsymbol V_R^i,Q_R^i)$ be the solution of
\[
\left\{
\begin{aligned}
 -\Delta\boldsymbol V_R^i+\nabla Q_R^i&=\boldsymbol0,
 &\qquad \diver\boldsymbol V_R^i&=0,
 &&\text{in }B_R\setminus\overline{T_0},\\
 \boldsymbol V_R^i&=\boldsymbol0,
 &&&&\text{on }\partial T_0,\\
 \boldsymbol V_R^i&=\be_i,
 &&&&\text{on }\partial B_R,
\end{aligned}
\right.
\]
with $\int_{B_R\setminus\overline{T_0}}Q_R^i\,\mathrm d\by=0$. Define
\[
 (\bM_R)_{ij}=
 \int_{B_R\setminus\overline{T_0}}
 \nabla\boldsymbol V_R^i:\nabla\boldsymbol V_R^j\,\mathrm d\by,
 \qquad
 \boldsymbol t_R^i=(\nabla\boldsymbol V_R^i-Q_R^i\bI)\bn
 \quad\text{on }\partial B_R.
\]
Here $\bn$ denotes the outer normal of $B_R$. Then there exists $C>0$, independent of $R\ge4$, such that the following hold:
\begin{equation}\label{eq:truncated-stokes-estimates}
 \lVert\bM_R-\bM_0\rVert\le CR^{-1},\qquad
 \int_{\partial B_R}\boldsymbol t_R^i\,\mathrm dS=\bM_R\be_i,
 \qquad
 \norm{\boldsymbol t_R^i}{L^2(\partial B_R)}\le CR^{-1}.
\end{equation}
Moreover,
\begin{equation}\label{eq:truncated-pointwise}
 |\boldsymbol V_R^i(\by)-\be_i|\le C|\by|^{-1},\qquad
 |\nabla\boldsymbol V_R^i(\by)|+|Q_R^i(\by)|\le C|\by|^{-2},
\end{equation}
for $2\le|\by|\le R$.  In addition,
\begin{equation}\label{eq:truncated-near-field}
 \norm{\boldsymbol V_R^i-\be_i}{L^6(B_2\setminus T_0)}
 +\norm{\nabla\boldsymbol V_R^i}{L^2(B_2\setminus T_0)}
 +\norm{Q_R^i}{L^{6/5}(B_2\setminus T_0)}\le C.
\end{equation}
\end{lemma}

\begin{proof}
Finite-energy exterior Stokes theory and the large-distance estimates of \cite[Theorem~V.3.2]{Galdi} give the stated decay. The constants in the following estimates are uniform for $R\ge4$. For $\boldsymbol\xi\in\R^3$, define
\begin{align*}
 \boldsymbol V_{\boldsymbol\xi}
 &=\sum_{i=1}^3\xi_i\boldsymbol V^i,
 &\boldsymbol Z_{\boldsymbol\xi}
 &=\boldsymbol V_{\boldsymbol\xi}-\boldsymbol\xi
   =\sum_{i=1}^3\xi_i(\boldsymbol V^i-\be_i),\\
 \boldsymbol Z_{R,\boldsymbol\xi}
 &=\sum_{i=1}^3\xi_i(\boldsymbol V_R^i-\be_i),
 &
 D_R&=B_R\setminus\overline{T_0}.
\end{align*}
Define the affine solenoidal Dirichlet classes
\begin{align*}
 \mathcal A_R(\boldsymbol\xi)
 :={}&\bigl\{\bz\in H^1(D_R;\R^3):
 \diver\bz=0,\
 \operatorname{Tr}\bz=-\boldsymbol\xi\ \text{on }\partial T_0,\
 \operatorname{Tr}\bz=\boldsymbol0\ \text{on }\partial B_R\bigr\},\\
 \mathcal A_\infty(\boldsymbol\xi)
 :={}&\bigl\{\bz\in L^6(\R^3\setminus\overline{T_0};\R^3):
 \nabla\bz\in L^2,
 \ \diver\bz=0,
 \ \operatorname{Tr}\bz=-\boldsymbol\xi\ \text{on }\partial T_0\bigr\}.
\end{align*}
The condition $\bz\in L^6(\R^3\setminus\overline{T_0})$ excludes nonzero constant states at infinity and therefore fixes the finite-energy representative. The fields $\boldsymbol Z_{R,\boldsymbol\xi}$ and $\boldsymbol Z_{\boldsymbol\xi}$ are the unique energy minimizers in these two affine classes. In particular,
\begin{equation*}
 \boldsymbol\xi\cdot\bM_R\boldsymbol\eta
 =\int_{D_R}\nabla\boldsymbol Z_{R,\boldsymbol\xi}:
 \nabla\boldsymbol Z_{R,\boldsymbol\eta}\,\mathrm d\by,
 \qquad
 \boldsymbol\xi\cdot\bM_0\boldsymbol\eta
 =\int_{\R^3\setminus\overline{T_0}}
 \nabla\boldsymbol Z_{\boldsymbol\xi}:
 \nabla\boldsymbol Z_{\boldsymbol\eta}\,\mathrm d\by.
\end{equation*}
Accordingly, $\bM_R$ and $\bM_0$ are the Gram matrices of the corresponding Dirichlet energies.

Choose a radial $\chi_R\in C_c^\infty(B_R)$ such that $\chi_R=1$ on $B_{R/2}$ and $|\nabla\chi_R|\le C/R$, and set $A_R=B_R\setminus\overline{B_{R/2}}$ and $f_R=\nabla\chi_R\cdot\boldsymbol Z_{\boldsymbol\xi}$. Let $\cB_{A_R}$ denote the standard Bogovski\u{\i} operator (see, e.g.,~\cite[Chapter~III]{Galdi}), i.e., the bounded right inverse of the divergence on $A_R$: for $f\in L_0^2(A_R):=\{g\in L^2(A_R):\int_{A_R}g\,\mathrm d\by=0\}$, the field $\cB_{A_R}f\in H_0^1(A_R;\R^3)$ satisfies
\[
 \diver(\cB_{A_R}f)=f,
 \qquad
 \|\nabla\cB_{A_R}f\|_{L^2(A_R)}\le C\|f\|_{L^2(A_R)},
\]
where the constant is uniform under dilation of the fixed annulus. Thus the only compatibility condition needed here is the zero-mean condition $\int_{A_R}f_R\,\mathrm d\by=0$. Since $\diver\boldsymbol Z_{\boldsymbol\xi}=0$,
\begin{align*}
 \int_{A_R}f_R\,\mathrm d\by
 &=\int_{A_R}\diver(\chi_R\boldsymbol Z_{\boldsymbol\xi})\,\mathrm d\by
 =\int_{\partial A_R}\chi_R\boldsymbol Z_{\boldsymbol\xi}\cdot\bn\,\mathrm dS
 =-\int_{\partial B_{R/2}}\boldsymbol Z_{\boldsymbol\xi}\cdot\bn\,\mathrm dS=0.
\end{align*}
For any sphere $\partial B_s$ enclosing $T_0$, the divergence theorem and the zero trace of $\boldsymbol V_{\boldsymbol\xi}$ on $\partial T_0$ give
\[
 \int_{\partial B_s}\boldsymbol V_{\boldsymbol\xi}\cdot\bn\,\mathrm dS=0.
\]
Moreover, $\int_{\partial B_s}\boldsymbol\xi\cdot\bn\,\mathrm dS=0$, so the flux of $\boldsymbol Z_{\boldsymbol\xi}$ also vanishes. We may therefore set $\boldsymbol b_R:=\cB_{A_R}f_R$. By the scale-invariant Bogovski\u{\i} estimate,
\[
 \diver\boldsymbol b_R=f_R,
 \qquad
 \norm{\nabla\boldsymbol b_R}{L^2(A_R)}
 \le C\norm{f_R}{L^2(A_R)}
 \le CR^{-1/2}|\boldsymbol\xi|.
\]
Consequently,
\[
 \boldsymbol Y_R:=\chi_R\boldsymbol Z_{\boldsymbol\xi}-\boldsymbol b_R
 \in\mathcal A_R(\boldsymbol\xi).
\]
The exterior decay then implies
\[
 \int_{A_R}
 \left(|\nabla\boldsymbol Z_{\boldsymbol\xi}|^2
 +R^{-2}|\boldsymbol Z_{\boldsymbol\xi}|^2\right)\,\mathrm d\by
 \le CR^{-1}|\boldsymbol\xi|^2,
\]
and the minimizing property yields
\begin{equation}\label{eq:MR-upper}
 \boldsymbol\xi\cdot\bM_R\boldsymbol\xi
 \le \boldsymbol\xi\cdot\bM_0\boldsymbol\xi
 +CR^{-1}|\boldsymbol\xi|^2.
\end{equation}
Conversely, define
\[
 \widetilde{\boldsymbol Z}_{R,\boldsymbol\xi}(\by)
 :=\begin{cases}
 \boldsymbol Z_{R,\boldsymbol\xi}(\by),&\by\in D_R,\\
 \boldsymbol0,&\by\in\R^3\setminus B_R.
 \end{cases}
\]
Combined with $\operatorname{Tr}_{\partial B_R}\boldsymbol Z_{R,\boldsymbol\xi}=\boldsymbol0$, this extension satisfies
\[
 \widetilde{\boldsymbol Z}_{R,\boldsymbol\xi}\in\mathcal A_\infty(\boldsymbol\xi).
\]
The minimizing property of $\boldsymbol Z_{\boldsymbol\xi}$ therefore gives
\begin{equation}\label{eq:MR-lower}
 \boldsymbol\xi\cdot\bM_0\boldsymbol\xi
 \le\boldsymbol\xi\cdot\bM_R\boldsymbol\xi.
\end{equation}
Since $\bM_R-\bM_0$ is symmetric, \eqref{eq:MR-upper}--\eqref{eq:MR-lower} imply $\|\bM_R-\bM_0\|\le CR^{-1}$. To compare the truncated and exterior correctors, set $D_R=B_R\setminus\overline{T_0}$ and $\boldsymbol Z_R^i:=\boldsymbol V_R^i-\be_i$ in $D_R$, and define
\[
 \boldsymbol H_R^i(\by)=
 \begin{cases}
  \boldsymbol Z_R^i(\by)-\boldsymbol Z^i(\by),&\by\in D_R,\\
  -\boldsymbol Z^i(\by),&\by\in\R^3\setminus\overline{B_R},
 \end{cases}
 \qquad
 \Pi_R^i:=Q_R^i-Q^i\quad\text{in }D_R.
\]
The traces of the two pieces coincide on $\partial B_R$, while $\operatorname{Tr}_{\partial T_0}\boldsymbol H_R^i=\boldsymbol0$. Extend $\boldsymbol H_R^i$ by zero to $T_0$ and retain the same notation. Then
\[
\begin{aligned}
 &\boldsymbol H_R^i\in L^6(\R^3;\R^3),
 \qquad
 \nabla\boldsymbol H_R^i\in L^2(\R^3;\R^{3\times3}),\\
 &\diver\boldsymbol H_R^i=0\quad\text{in }\mathcal D'(\R^3),
 \qquad
 \boldsymbol H_R^i=\boldsymbol0\quad\text{a.e.\ in }T_0.
\end{aligned}
\]
Thus $\boldsymbol H_R^i|_{\R^3\setminus\overline{T_0}}$ is an admissible finite-energy test field in the weak Euler--Lagrange equation for $\boldsymbol Z^i$. Consequently,
\[
 \int_{\R^3\setminus T_0}
 \nabla\boldsymbol Z^i:\nabla\boldsymbol H_R^i\,\mathrm d\by=0,
\]
and therefore
\begin{equation}\label{eq:HR-energy}
 \norm{\nabla\boldsymbol H_R^i}{L^2(\R^3\setminus T_0)}^2
 = (\bM_R-\bM_0)_{ii}
 \le CR^{-1},
 \qquad
 \norm{\boldsymbol H_R^i}{L^6(\R^3\setminus T_0)}
 \le CR^{-1/2}.
\end{equation}

Pointwise control is obtained from the uniform scaled estimate \eqref{eq:scaled-stokes-local}. For the interior region, let $\mathcal Q=B_4\setminus\overline{B_{1/2}}$ and $\mathcal Q'=B_{7/2}\setminus\overline{B_{3/4}}$. Firstly, we suppose that $R\ge16$. Then the interval $2\le r\le R/8$ is nonempty, and $T_0\subset\subset B_{1/8}$ ensures that each dilated pair $\mathcal Q_r=r\mathcal Q$, $\mathcal Q_r'=r\mathcal Q'$ lies entirely in the fluid region. Applying \eqref{eq:scaled-stokes-local} to $(\boldsymbol H_R^i,\Pi_R^i)$ and using \eqref{eq:HR-energy} gives
\begin{equation}\label{eq:HR-local}
 \norm{\boldsymbol H_R^i}{L^\infty(\mathcal Q_r')}
 \le CR^{-1/2}r^{-1/2},
 \qquad
 \norm{\nabla\boldsymbol H_R^i}{L^\infty(\mathcal Q_r')}
 +\inf_c\norm{\Pi_R^i-c}{L^\infty(\mathcal Q_r')}
 \le CR^{-1/2}r^{-3/2}.
\end{equation}
In particular, the same estimates hold on the overlapping annuli $E_r=\{r<|\by|<3r\}\subset\subset\mathcal Q_r'$.

It remains to obtain estimates that are uniform up to the outer boundary. In $B_R\setminus\overline{B_{R/8}}$, the field $\boldsymbol Z_R^i$ solves the homogeneous Stokes system and has homogeneous Dirichlet data on $\partial B_R$, so boundary Stokes regularity can be applied uniformly after rescaling. With $\by=R\boldsymbol\eta$, choose a finite connected cover $\{\mathcal P_\ell\}_{\ell=1}^{N_0}$ of the fixed shell $\{1/4<|\boldsymbol\eta|<1\}$ by interior balls and $C^2$ boundary charts at $\partial B_1$. For each $\mathcal P_\ell$, choose a slightly larger fixed patch $\mathcal P_\ell^+$ so that the pair $(\mathcal P_\ell,\mathcal P_\ell^+)$ satisfies the hypotheses of \eqref{eq:scaled-stokes-local}. On boundary patches, $\boldsymbol Z_R^i=\boldsymbol0$ on the physical boundary portion. The number $N_0$, the chart norms and the overlap geometry are independent of $R$, and the cover stays a fixed positive distance from the artificial sphere $|\boldsymbol\eta|=1/8$. Since
\[
 \norm{\boldsymbol Z_R^i}{L^6(B_R\setminus B_{R/8})}
 \le \norm{\boldsymbol H_R^i}{L^6}
 +\norm{\boldsymbol Z^i}{L^6(B_R\setminus B_{R/8})}
 \le CR^{-1/2},
\]
for each rescaled pair $(R\mathcal P_\ell,R\mathcal P_\ell^+)$, estimate \eqref{eq:scaled-stokes-local} yields a constant $c_{\ell,R}\in\R$ satisfying
\[
 \norm{\boldsymbol Z_R^i}{L^\infty(R\mathcal P_\ell)}
 +R\norm{\nabla\boldsymbol Z_R^i}{L^\infty(R\mathcal P_\ell)}
 +R\norm{Q_R^i-c_{\ell,R}}{L^\infty(R\mathcal P_\ell)}
 \le CR^{-1}.
\]
If two neighboring patches overlap, the two pressure bounds on the overlap imply $|c_{\ell,R}-c_{\ell',R}|\le CR^{-2}$.  Since the adjacency graph of the finite cover is connected, matching the constants along a chain of at most $N_0$ overlaps produces a constant $c_R^{\mathrm{out}}$ for the whole outer shell.  Hence
\begin{equation}\label{eq:outer-stokes-local}
 \norm{\boldsymbol Z_R^i}{L^\infty(B_R\setminus B_{R/4})}
 +R\norm{\nabla\boldsymbol Z_R^i}{L^\infty(B_R\setminus B_{R/4})}
 +R\norm{Q_R^i-c_R^{\mathrm{out}}}{L^\infty(B_R\setminus B_{R/4})}
 \le CR^{-1}.
\end{equation}
In this way, the velocity, gradient and pressure bounds remain uniform up to $\partial B_R$, with one common pressure constant on the connected outer shell. Since the exterior pressure satisfies $|Q^i(\by)|\le C|\by|^{-2}$, subtracting $Q^i$ changes this bound only by the same admissible order on the outer annuli. The subsequent dyadic matching is performed for $\Pi_R^i=Q_R^i-Q^i$, so the pressure constants can be matched globally across the connected outer shell.

We next fix the pressure constants by a dyadic comparison. For $R\ge64$, let $r_j=2^{j+1}$, $j=0,\ldots,J$, choose $J$ so that $R/16\le r_J<R/8$, and set
\[
 E_j=\{\by:r_j<|\by|<3r_j\},
 \qquad
 c_j=\frac{1}{|E_j|}\int_{E_j}\Pi_R^i\,\mathrm d\by.
\]
To connect the dyadic chain to the outer shell without assuming an overlap that may fail for the last dyadic radius, introduce the fixed bridge annulus
\[
 E_*=\{\by:R/8<|\by|<3R/8\},
 \qquad
 c_*=\frac{1}{|E_*|}\int_{E_*}\Pi_R^i\,\mathrm d\by,
\]
and the outer annulus
\[
 E_{\mathrm{out}}=\{\by:R/4<|\by|<3R/4\},
 \qquad
 c_R=\frac{1}{|E_{\mathrm{out}}|}\int_{E_{\mathrm{out}}}\Pi_R^i\,\mathrm d\by.
\]
Successive annuli $E_j$ and $E_{j+1}$ overlap on a set of measure comparable to $r_j^3$. Moreover, since $R/16\le r_J<R/8$, the overlap $E_J\cap E_*$ has measure comparable to $R^3$, while $E_*\cap E_{\mathrm{out}}$ also has measure comparable to $R^3$. Applying \eqref{eq:HR-local} with $r=R/8$ first gives the same bound relative to a suitable additive constant on $E_*$. Replacing that constant by the annular average $c_*$ changes it by at most the oscillation itself and hence
\[
 \norm{\Pi_R^i-c_*}{L^\infty(E_*)}\le CR^{-2}.
\]
Similarly, \eqref{eq:outer-stokes-local} gives $\|Q_R^i-c_R^{\mathrm{out}}\|_{L^\infty(E_{\mathrm{out}})} \le CR^{-2}$, while the exterior decay gives $\|Q^i\|_{L^\infty(E_{\mathrm{out}})}\le CR^{-2}$. Thus we have
\[
 \norm{\Pi_R^i-c_R^{\mathrm{out}}}{L^\infty(E_{\mathrm{out}})}
 \le CR^{-2}.
\]
Since $c_R$ is the average of $\Pi_R^i$ over $E_{\mathrm{out}}$, the same estimate implies $|c_R-c_R^{\mathrm{out}}|\le CR^{-2}$ and hence
\[
 \norm{\Pi_R^i-c_R}{L^\infty(E_{\mathrm{out}})}\le CR^{-2}.
\]
For the dyadic annuli, we have
\[
 \norm{\Pi_R^i-c_j}{L^\infty(E_j)}
 \le CR^{-1/2}r_j^{-3/2}.
\]
Taking averages over successive overlaps yields
\begin{equation}\label{eq:dyadic-pressure-matching}
 |c_{j+1}-c_j|
 \le CR^{-1/2}r_j^{-3/2},
 \qquad
 |c_*-c_J|+|c_R-c_*|\le CR^{-2}.
\end{equation}
Summing \eqref{eq:dyadic-pressure-matching} along the dyadic chain and then through the bridge annulus gives
\[
 |\Pi_R^i(\by)-c_R|
 \le CR^{-1/2}|\by|^{-3/2}
 \qquad (2\le|\by|\le 3R/8).
\]
On the outer region $R/4\le|\by|\le R$, the global outer-shell constant in \eqref{eq:outer-stokes-local}, the exterior pressure decay and $|c_R-c_R^{\mathrm{out}}|\le CR^{-2}$ give $|\Pi_R^i-c_R|\le CR^{-2}$. Since the two regions overlap, this gives
\begin{equation}\label{eq:pressure-oscillation}
 |\Pi_R^i(\by)-c_R|
 \le CR^{-1/2}|\by|^{-3/2}
 \qquad (2\le|\by|\le R).
\end{equation}

The zero-mean normalization of $Q_R^i$ then determines $c_R$. By the exterior decay stated in \Cref{lem:truncated-stokes}, $|Q^i(\by)|\le C|\by|^{-2}$ for $|\by|\ge2$, while local boundary Stokes regularity gives $Q^i\in L^{6/5}(B_3\setminus\overline{T_0})$. Hence
\begin{equation}\label{eq:mean-Pi}
 \left|\frac{1}{|D_R|}\int_{D_R}\Pi_R^i\,\mathrm d\by\right|
 =\left|\frac{1}{|D_R|}\int_{D_R}Q^i\,\mathrm d\by\right|
 \le CR^{-2}.
\end{equation}
On the fixed region $B_3\setminus T_0$, the local pressure estimate determines $\Pi_R^i$ up to an additive constant. The annulus $2<|\by|<3$ lies in the common range of this local estimate and \eqref{eq:pressure-oscillation}. Comparing the two bounds there fixes the additive constant. Hence
\[
 \norm{\Pi_R^i-c_R}{L^{6/5}(B_3\setminus T_0)}\le C.
\]
Combining this fixed-region bound with \eqref{eq:pressure-oscillation} yields
\[
 \left|\frac{1}{|D_R|}\int_{D_R}(\Pi_R^i-c_R)\,\mathrm d\by\right|
 \le CR^{-2}.
\]
Together with \eqref{eq:mean-Pi}, this gives $|c_R|\le CR^{-2}$. Therefore, we have
\begin{equation*}
 |\Pi_R^i(\by)|
 \le C\bigl(R^{-1/2}|\by|^{-3/2}+R^{-2}\bigr)
 \le C|\by|^{-2},
 \qquad 2\le|\by|\le R.
\end{equation*}
For the velocity, the estimates follow from $\boldsymbol Z_R^i=\boldsymbol Z^i+\boldsymbol H_R^i$, \eqref{eq:HR-local}, \eqref{eq:outer-stokes-local} and the exterior decay. This yields \eqref{eq:truncated-pointwise} for $R\ge64$. For $4\le R\le64$, the family of annular domains has uniformly controlled $C^{2,\beta}$ geometry and lies in a compact range of scales. The same estimates follow from fixed-domain interior and boundary Stokes regularity after increasing $C$. On $B_2\setminus T_0$, we use the uniform energy bound, Poincar\'e--Sobolev and the fixed-domain pressure estimate, with its additive constant matched to $2<|\by|<3$, to obtain \eqref{eq:truncated-near-field}.

Integration by parts against $\boldsymbol V_R^j$ gives
\[
 (\bM_R)_{ij}
 =\be_j\cdot\int_{\partial B_R}
 (\nabla\boldsymbol V_R^i-Q_R^i\bI)\bn\,\mathrm dS.
\]
Since $\bM_R$ is symmetric,
\[
 \int_{\partial B_R}\boldsymbol t_R^i\,\mathrm dS=\bM_R\be_i.
\]
Pointwise control on $\partial B_R$ implies $|\boldsymbol t_R^i|\le CR^{-2}$ and therefore $\|\boldsymbol t_R^i\|_{L^2(\partial B_R)}\le CR^{-1}$. The constants in this lemma depend only on the reference hole $T_0$, its fixed regularity class and the fixed annular estimates. In particular, they are independent of $R$.
\end{proof}

With the exterior problem under control, we transfer the estimates to the microscopic cells. Apply \Cref{lem:truncated-stokes} with $R_\eps=(4\eps^2)^{-1}$ and set $B_{\eps,\boldsymbol k}=B(\eps\boldsymbol k,\eps/4)$. In $B_{\eps,\boldsymbol k}\setminus T_{\eps,\boldsymbol k}$ define
\[
 \boldsymbol W_{\eps,\boldsymbol k}^i(\bx)
 =\boldsymbol V_{R_\eps}^i\!\left(\frac{\bx-\eps\boldsymbol k}{\eps^3}\right),
 \qquad
 Q_{\eps,\boldsymbol k}^i(\bx)
 =\eps^{-3}Q_{R_\eps}^i\!\left(\frac{\bx-\eps\boldsymbol k}{\eps^3}\right).
\]
Since the balls $B_{\eps,\boldsymbol k}$ are pairwise disjoint, define the global $i$th velocity corrector by
\[
 \boldsymbol W_\eps^i(\bx)=
 \begin{cases}
  \boldsymbol0,&\bx\in T_{\eps,\boldsymbol k}\ \text{for some }\boldsymbol k\in\cK_\eps,\\
  \boldsymbol W_{\eps,\boldsymbol k}^i(\bx),&\bx\in B_{\eps,\boldsymbol k}\setminus T_{\eps,\boldsymbol k}\ \text{for some }\boldsymbol k\in\cK_\eps,\\
  \be_i,&\bx\in\Omega\setminus\displaystyle\bigcup_{\boldsymbol k\in\cK_\eps}B_{\eps,\boldsymbol k}.
 \end{cases}
\]
The corresponding pressure corrector is
\[
 Q_\eps^i(\bx)=
 \begin{cases}
  Q_{\eps,\boldsymbol k}^i(\bx),&\bx\in B_{\eps,\boldsymbol k}\setminus T_{\eps,\boldsymbol k}\ \text{for some }\boldsymbol k\in\cK_\eps,\\
  0,&\text{otherwise}.
 \end{cases}
\]
Set $\bW_\eps=(\boldsymbol W_\eps^1\ \boldsymbol W_\eps^2\ \boldsymbol W_\eps^3)$ and $\bQ_\eps=(Q_\eps^1,Q_\eps^2,Q_\eps^3)$.

For each $\boldsymbol k\in\cK_\eps$, set
\[
 S_{\eps,\boldsymbol k}=\partial B_{\eps,\boldsymbol k},
 \qquad
 \boldsymbol t_{\eps,\boldsymbol k}^i(\bx)
 =\eps^{-3}
 (\nabla\boldsymbol V_{R_\eps}^i-Q_{R_\eps}^i\bI)
 \!\left(\frac{\bx-\eps\boldsymbol k}{\eps^3}\right)\bn
 \in L^2(S_{\eps,\boldsymbol k};\R^3).
\]
Define the surface distribution on $C_c^\infty(\Oe;\R^3)$ by
\begin{equation}\label{eq:surface-distribution-definition}
 \dual{\cT_\eps^i}{\bz}
 :=\sum_{\boldsymbol k\in\cK_\eps}
 \int_{S_{\eps,\boldsymbol k}}
 \boldsymbol t_{\eps,\boldsymbol k}^i\cdot\bz\,\mathrm dS.
\end{equation}
Cellwise integration by parts identifies this distribution with $-\Delta\boldsymbol W_\eps^i+\nabla Q_\eps^i$ in $\mathcal D'(\Oe;\R^3)$.

\begin{lemma}[Cell-capacity residual]\label{lem:cell-residual}
For $\eps$ small enough,
the correctors $\bW_\eps$ and $\bQ_\eps$ satisfy
\begin{equation}\label{eq:raw-corrector-estimates}
 \norm{\bW_\eps-\bI}{L^6(\Omega)}
 +\norm{\nabla\bW_\eps}{L^{6/5}(\Oe)}
 +\norm{\bQ_\eps}{L^{6/5}(\Oe)}
 \le C\eps,
 \qquad
 \norm{\nabla\bW_\eps}{L^2(\Oe)}\le C.
\end{equation}
For each column, let $\cT_\eps^i$ be the surface distribution defined in \eqref{eq:surface-distribution-definition}. Then $\cT_\eps^i$ extends uniquely to $(H_0^1(\Oe;\R^3))^\ast$ and
\begin{equation*}
 \left|\dual{\cT_\eps^i-\bM_0\be_i}{\bz_\eps}\right|
 \le C\eps\norm{\bz_\eps}{H^1(\Oe)}
 \qquad
 \text{for any }\bz_\eps\in H_0^1(\Oe;\R^3).
\end{equation*}
\end{lemma}

\begin{proof}
Fix a cell and write $r=|\bx-\eps\boldsymbol k|$. Split the fluid part of the cell ball into
\[
 \mathcal N_{\eps,\boldsymbol k}
 =B(\eps\boldsymbol k,2\eps^3)\setminus T_{\eps,\boldsymbol k},
 \qquad
 F_{\eps,\boldsymbol k}
 =B_{\eps,\boldsymbol k}\setminus B(\eps\boldsymbol k,2\eps^3).
\]
For $\bx\in F_{\eps,\boldsymbol k}$, the scaled radius satisfies $2\le r/\eps^3\le R_\eps$, thus \eqref{eq:truncated-pointwise} applies and gives
\[
 |\boldsymbol W_\eps^i-\be_i|\le C\frac{\eps^3}{r},
 \qquad
 |\nabla\boldsymbol W_\eps^i|+|Q_\eps^i|
 \le C\frac{\eps^3}{r^2}.
\]
For $p=6/5$, direct integration in spherical coordinates yields
\[
 \int_{F_{\eps,\boldsymbol k}}
 (|\nabla\boldsymbol W_\eps^i|^p+|Q_\eps^i|^p)\,\mathrm d\bx
 \le C\eps^{21/5},
\]
whereas \eqref{eq:truncated-near-field} and the change of variables $\bx=\eps\boldsymbol k+\eps^3\by$ give
\[
 \int_{\mathcal N_{\eps,\boldsymbol k}}
 (|\nabla\boldsymbol W_\eps^i|^p+|Q_\eps^i|^p)\,\mathrm d\bx
 \le C\eps^{9-3p}=C\eps^{27/5}.
\]
Summing over $O(\eps^{-3})$ cells and taking the $p$th root, we obtain an $O(\eps)$ far-field contribution and an $O(\eps^2)$ near-field contribution. The same decomposition with $p=2$ gives a uniformly bounded total gradient energy. Turning to the velocity,
\[
 \int_{F_{\eps,\boldsymbol k}}|\boldsymbol W_\eps^i-\be_i|^6\,\mathrm d\bx
 \le C\eps^9,
\]
and \eqref{eq:truncated-near-field} gives the same order on $\mathcal N_{\eps,\boldsymbol k}$. The holes contribute $\int_{\He}|\bW_\eps-\bI|^6\,\mathrm d\bx\le C|\He|=O(\eps^6)$. These estimates give \eqref{eq:raw-corrector-estimates}.

To estimate the capacity defect, we begin with the surface functional \eqref{eq:surface-distribution-definition} on $C_c^\infty(\Oe;\R^3)$ and introduce the sphere flux vector
\[
 \boldsymbol\Phi_{\eps,\boldsymbol k}^i
 :=\int_{S_{\eps,\boldsymbol k}}
 \boldsymbol t_{\eps,\boldsymbol k}^i\,\mathrm dS\in\R^3.
\]
Scaling \eqref{eq:truncated-stokes-estimates} gives
\begin{equation}\label{eq:traction-cell-estimates}
 \boldsymbol\Phi_{\eps,\boldsymbol k}^i
 =\eps^3\bM_{R_\eps}\be_i,
 \qquad
 \norm{\boldsymbol t_{\eps,\boldsymbol k}^i}{L^2(S_{\eps,\boldsymbol k})}
 \le C\eps^2.
\end{equation}

Let
\[
 Y_{\eps,\boldsymbol k}=\eps(\boldsymbol k+Y),
 \qquad
 \Omega_\eps^{\mathrm{cell}}
 =\bigcup_{\boldsymbol k\in\cK_\eps}\overline{Y_{\eps,\boldsymbol k}},
 \qquad
 \Gamma_\eps=\Omega\setminus\Omega_\eps^{\mathrm{cell}}.
\]
For $\bz_\eps\in H_0^1(\Oe;\R^3)$, let $\widetilde{\bz}_\eps\in H_0^1(\Omega;\R^3)$ denote the zero extension. Its cell average is denoted by
\begin{equation}\label{eq:cell-average-definition}
 \left\langle\widetilde{\bz}_\eps\right\rangle_{\boldsymbol k}
 :=\eps^{-3}\int_{Y_{\eps,\boldsymbol k}}
 \widetilde{\bz}_\eps\,\mathrm d\bx\in\R^3.
\end{equation}
Cell faces have zero Lebesgue measure, so replacing open cells by their closures leaves all cell integrals unchanged. Every point at distance larger than $\sqrt{3}\eps$ from $\partial\Omega$ belongs to the closure of a cell indexed by $\cK_\eps$. The tubular-neighborhood estimate for the fixed $C^{2,\beta}$ boundary therefore gives
\begin{equation}\label{eq:boundary-strip-definition}
 \Gamma_\eps\subset
 \{\bx\in\Omega:\operatorname{dist}(\bx,\partial\Omega)\le C\eps\},
 \qquad
 |\Gamma_\eps|\le C\eps.
\end{equation}

Taking $\bz_\eps\in C_c^\infty(\Oe;\R^3)$, combining \eqref{eq:surface-distribution-definition}, \eqref{eq:traction-cell-estimates} and \eqref{eq:cell-average-definition}, we obtain the exact decomposition
\begin{align*}
 \dual{\cT_\eps^i-\bM_0\be_i}{\bz_\eps}
 ={}&I_{1,\eps}+I_{2,\eps}+I_{3,\eps},
 \end{align*}
with
 \begin{align*}
 I_{1,\eps}
 :={}&\sum_{\boldsymbol k\in\cK_\eps}
 \int_{S_{\eps,\boldsymbol k}}
 \boldsymbol t_{\eps,\boldsymbol k}^i\cdot
 (\bz_\eps-\left\langle\widetilde{\bz}_\eps\right\rangle_{\boldsymbol k})\,\mathrm dS,\notag\\
 I_{2,\eps}
 :={}&\sum_{\boldsymbol k\in\cK_\eps}
 \eps^3(\bM_{R_\eps}-\bM_0)\be_i\cdot
 \left\langle\widetilde{\bz}_\eps\right\rangle_{\boldsymbol k},\notag\\
 I_{3,\eps}
 :={}&-\int_{\Gamma_\eps}
 \bM_0\be_i\cdot\widetilde{\bz}_\eps\,\mathrm d\bx.\notag
\end{align*}
Throughout this argument, the constant vector $\bM_0\be_i$ is paired with a test field through its zero extension to $\Omega$. Now, we estimate the $ I_{i,\eps}(i=1,2,3)$ separately. 

For $I_{1,\eps}$, the trace--Poincar\'e inequality on the rescaled unit cell gives
\[
 \norm{\bz_\eps-\left\langle\widetilde{\bz}_\eps\right\rangle_{\boldsymbol k}}
 {L^2(S_{\eps,\boldsymbol k})}
 \le C\eps^{1/2}
 \norm{\nabla\widetilde{\bz}_\eps}
 {L^2(Y_{\eps,\boldsymbol k})}.
\]
Hence, by \eqref{eq:traction-cell-estimates},
\begin{align}
 |I_{1,\eps}|
 &\le C\eps^{5/2}
 \sum_{\boldsymbol k\in\cK_\eps}
 \norm{\nabla\widetilde{\bz}_\eps}
 {L^2(Y_{\eps,\boldsymbol k})}\notag\\
 &\le C\eps^{5/2}(\#\cK_\eps)^{1/2}
 \left(\sum_{\boldsymbol k\in\cK_\eps}
 \norm{\nabla\widetilde{\bz}_\eps}
 {L^2(Y_{\eps,\boldsymbol k})}^2\right)^{1/2}\notag\\
 &\le C\eps\norm{\nabla\widetilde{\bz}_\eps}{L^2(\Omega)}.
 \label{eq:I1-bound}
\end{align}
Turning to $I_{2,\eps}$, Jensen's inequality gives
\[
 \eps^3|\left\langle\widetilde{\bz}_\eps\right\rangle_{\boldsymbol k}|^2
 \le\int_{Y_{\eps,\boldsymbol k}}
 |\widetilde{\bz}_\eps|^2\,\mathrm d\bx,
\]
and consequently
\begin{align*}
 \sum_{\boldsymbol k\in\cK_\eps}|\left\langle\widetilde{\bz}_\eps\right\rangle_{\boldsymbol k}|
 &\le(\#\cK_\eps)^{1/2}
 \left(\sum_{\boldsymbol k\in\cK_\eps}
 |\left\langle\widetilde{\bz}_\eps\right\rangle_{\boldsymbol k}|^2\right)^{1/2}\le C\eps^{-3/2}\,\eps^{-3/2}
 \norm{\widetilde{\bz}_\eps}{L^2(\Omega)}
 =C\eps^{-3}\norm{\widetilde{\bz}_\eps}{L^2(\Omega)}.
\end{align*}
Since $R_\eps^{-1}=4\eps^2$ and $\|\bM_{R_\eps}-\bM_0\|\le CR_\eps^{-1}$,
\[
 |I_{2,\eps}|
 \le C\eps^3\eps^2\eps^{-3}
 \norm{\widetilde{\bz}_\eps}{L^2(\Omega)}
 \le C\eps^2\norm{\widetilde{\bz}_\eps}{L^2(\Omega)}.
\]
Finally, for $I_{3,\eps}$, Hardy's inequality on the fixed domain gives
\[
 \left\|\frac{\widetilde{\bz}_\eps}
 {\operatorname{dist}(\cdot,\partial\Omega)}\right\|_{L^2(\Omega)}
 \le C\norm{\nabla\widetilde{\bz}_\eps}{L^2(\Omega)}.
\]
By \eqref{eq:boundary-strip-definition},
\[
 \norm{\widetilde{\bz}_\eps}{L^2(\Gamma_\eps)}
 \le C\eps\norm{\nabla\widetilde{\bz}_\eps}{L^2(\Omega)}.
\]
Therefore, using $|\Gamma_\eps|^{1/2}\le C\eps^{1/2}$, we obtain
\begin{equation}\label{eq:I3-bound}
 |I_{3,\eps}|
 \le C|\Gamma_\eps|^{1/2}
 \norm{\widetilde{\bz}_\eps}{L^2(\Gamma_\eps)}
 \le C\eps^{3/2}
 \norm{\nabla\widetilde{\bz}_\eps}{L^2(\Omega)}
 \le C\eps\norm{\bz_\eps}{H^1(\Oe)},
\end{equation}
for $0<\eps\le1$.

Adding \eqref{eq:I1-bound}--\eqref{eq:I3-bound}, we obtain
\[
 |\dual{\cT_\eps^i-\bM_0\be_i}{\bz_\eps}|
 \le C\eps\norm{\bz_\eps}{H^1(\Oe)}
 \qquad
 \text{for }\bz_\eps\in C_c^\infty(\Oe;\R^3).
\]
Since $C_c^\infty(\Oe;\R^3)$ is dense in $H_0^1(\Oe;\R^3)$, the surface distribution has a unique extension to $(H_0^1(\Oe;\R^3))^\ast$ and the same estimate holds for any test field in that space.
\end{proof}

We now incorporate a slowly varying coefficient and repair the divergence of the raw corrector.

\begin{proposition}[Restriction operator and weighted residual]\label{prop:restriction}There exists $\eps_0>0$ such that, for $0<\eps<\eps_0$, there exists a linear map $\cR_\eps:\cX(\Omega)\longrightarrow H_{0,\sigma}^1(\Oe)$. For $\bv\in\cX(\Omega)$, let
\[
 \widetilde{\cR_\eps\bv}(\bx)
 :=\begin{cases}
 \cR_\eps\bv(\bx),&\bx\in\Oe,\\
 \boldsymbol0,&\bx\in\He
 \end{cases}
\]
denote its zero extension to $\Omega$. Then there hold
\begin{equation}\label{eq:restriction-estimates}
\begin{aligned}
 \norm{\widetilde{\cR_\eps\bv}-\bv}{L^6(\Omega)}
 +\norm{\widetilde{\cR_\eps\bv}-\bv}{L^2(\Omega)}
 &\le C\eps\norm{\bv}{W^{1,\infty}(\Omega)},\\
 \norm{\nabla\cR_\eps\bv}{L^2(\Oe)}
 &\le C\norm{\bv}{W^{1,\infty}(\Omega)},\\
 \norm{\nabla(\widetilde{\cR_\eps\bv}-\bv)}{L^{6/5}(\Omega)}
 &\le C\eps\norm{\bv}{W^{1,\infty}(\Omega)}.
\end{aligned}
\end{equation}
For $a\in W^{1,\infty}(\Omega)$ and $\bz_\eps\in H_{0,\sigma}^1(\Oe)$ set
\begin{equation}\label{eq:weighted-residual-definition}
\begin{aligned}
 \mathfrak R_\eps[a,\bv,\bz_\eps]
 :={}&\int_\Omega2a\,\cD(\bv):\cD(\widetilde{\bz}_\eps)\,\mathrm d\bx
 +\int_\Omega a\,\bM_0\bv\cdot\widetilde{\bz}_\eps\,\mathrm d\bx-\int_{\Oe}2a\,\cD(\cR_\eps\bv):\cD(\bz_\eps)\,\mathrm d\bx.
\end{aligned}
\end{equation}
Then
\begin{equation}\label{eq:weighted-residual}
 |\mathfrak R_\eps[a,\bv,\bz_\eps]|
 \le C\eps\norm{a}{W^{1,\infty}(\Omega)}
 \norm{\bv}{W^{1,\infty}(\Omega)}
 \norm{\bz_\eps}{H^1(\Oe)}.
\end{equation}
If $\bv\in W^{1,\infty}(0,T;\cX(\Omega))$, then
\begin{equation}\label{eq:time-commutation}
 \partial_t(\cR_\eps\bv)=\cR_\eps(\partial_t\bv)
 \quad\text{a.e.\ in }(0,T).
\end{equation}
\end{proposition}

\begin{proof}
Set
\[
 \bU_\eps^0=\sum_{i=1}^3v_i\boldsymbol W_\eps^i=\bW_\eps\bv.
\]
Since each column of $\bW_\eps$ is solenoidal and $\diver\bv=0$, we have
\[
 g_\eps:=\diver\bU_\eps^0
 =(\bW_\eps^{\textup T}-\bI):\nabla\bv,
 \qquad
 \norm{g_\eps}{L^2(\Oe)}
 \le C\eps\norm{\bv}{W^{1,\infty}(\Omega)}.
\]
Moreover, $\bU_\eps^0\in H_0^1(\Oe;\R^3)$, so $\int_{\Oe}g_\eps\,\mathrm d\bx=0$. The divergence repair is obtained from \cite[Theorem~2.3, estimate~(2.21)]{DFL} with $q=2$. The periodic family \eqref{eq:geometry} is a special case of the separated-perforation geometry considered in \cite[Section~2.1]{DFL}, with exponent $\alpha=3$ and fixed shape and separation constants. Fix the linear right inverse supplied by that theorem for $0<\eps<\eps_0$, which depends only on the perforated geometry and is therefore time independent. In the notation of \cite[Theorem~2.3]{DFL},
\[
 \|\cB_\eps^{\mathrm{perf}} f\|_{W_0^{1,q}}
 \le C\bigl(1+\eps^{((3-q)\alpha-3)/q}\bigr)\|f\|_{L^q},
 \qquad 1<q<\infty.
\]
At $(\alpha,q)=(3,2)$ the exponent is zero, so the operator norm is uniform in $\eps$. Thus
\begin{equation}\label{eq:DFL-repair}
 \diver\cB_\eps^{\mathrm{perf}}[g_\eps]=g_\eps,
 \qquad
 \norm{\cB_\eps^{\mathrm{perf}}[g_\eps]}{H_0^1(\Oe)}
 \le C\norm{g_\eps}{L^2(\Oe)}
 \le C\eps\norm{\bv}{W^{1,\infty}(\Omega)}.
\end{equation}
Define
\[
 \cR_\eps\bv=\bU_\eps^0-\cB_\eps^{\mathrm{perf}}[g_\eps].
\]
Then $\cR_\eps\bv\in H_{0,\sigma}^1(\Oe)$. In $\Oe$,
\[
 \cR_\eps\bv-\bv
 =(\bW_\eps-\bI)\bv-\cB_\eps^{\mathrm{perf}}[g_\eps],
\]
so \eqref{eq:raw-corrector-estimates}, \eqref{eq:DFL-repair} and the boundedness of $\Omega$ yield the $L^6$ and $L^2$ bounds in \eqref{eq:restriction-estimates} after zero extension. For the fixed-domain $L^{6/5}$ gradient estimate,
\begin{align*}
 \nabla(\cR_\eps\bv-\bv)
 ={}&\sum_{i=1}^3v_i\nabla\boldsymbol W_\eps^i
 +\sum_{i=1}^3(\boldsymbol W_\eps^i-\be_i)\otimes\nabla v_i
 -\nabla\cB_\eps^{\mathrm{perf}}[g_\eps]
 \qquad\text{in }\Oe.
\end{align*}
For the first term, the $L^{6/5}$ bound is $O(\eps)$ by $\|\nabla\bW_\eps\|_{L^{6/5}}\le C\eps$. The second is $O(\eps)$ in $L^{6/5}$ by $\|\bW_\eps-\bI\|_{L^6}\le C\eps$ together with the boundedness of $\Omega$. The repair term is $O(\eps)$ in $L^2$ and hence also in $L^{6/5}$. Inside the holes, the field $\widetilde{\cR_\eps\bv}-\bv$ equals $-\bv$ and therefore
\[
 \norm{\nabla\bv}{L^{6/5}(\He)}
 \le |\He|^{5/6}\norm{\nabla\bv}{L^\infty}
 \le C\eps^5\norm{\bv}{W^{1,\infty}}.
\]
This yields \eqref{eq:restriction-estimates}.

To prove \eqref{eq:weighted-residual}, fix $\bz_\eps\in C_c^\infty(\Oe;\R^3)$ with $\diver\bz_\eps=0$ and set $\omega_i=av_i.$ \Cref{lem:cell-residual} is valid for any $H_0^1$ test field, hence it applies directly to $\omega_i\bz_\eps\in H_0^1(\Oe;\R^3)$. Define the capacity error
\[
 E_{\mathrm{cap}}
 :=\sum_{i=1}^3\left[
 \dual{\cT_\eps^i}{\omega_i\bz_\eps}
 -\int_\Omega\omega_i\bM_0\be_i\cdot
 \widetilde{\bz}_\eps\,\mathrm d\bx\right].
\]
Then we have
\begin{equation}\label{eq:E-cap-bound}
 |E_{\mathrm{cap}}|
 \le C\eps\norm{a}{W^{1,\infty}}
 \norm{\bv}{W^{1,\infty}}
 \norm{\bz_\eps}{H^1(\Oe)}.
\end{equation}
By the distributional identity for $\cT_\eps^i$ and $\diver\bz_\eps=0$,
\begin{align}
 \dual{\cT_\eps^i}{\omega_i\bz_\eps}
 ={}&\int_{\Oe}\omega_i\nabla\boldsymbol W_\eps^i:
 \nabla\bz_\eps\,\mathrm d\bx+\int_{\Oe}
 (\nabla\boldsymbol W_\eps^i-Q_\eps^i\bI)
 \nabla\omega_i\cdot\bz_\eps\,\mathrm d\bx.
 \label{eq:distribution-product-chain}
\end{align}
Define the pressure/coefficient commutator
\[
 E_{\mathrm{comm}}
 :=\sum_{i=1}^3\int_{\Oe}
 (\nabla\boldsymbol W_\eps^i-Q_\eps^i\bI)
 \nabla\omega_i\cdot\bz_\eps\,\mathrm d\bx.
\]
For spatially varying $a$ and $v_i$, the pressure contribution is nonzero and therefore appears explicitly in the commutator. By \eqref{eq:raw-corrector-estimates} and Sobolev embedding,
\begin{equation}\label{eq:E-comm-bound}
 |E_{\mathrm{comm}}|
 \le C\eps\norm{a}{W^{1,\infty}}
 \norm{\bv}{W^{1,\infty}}
 \norm{\bz_\eps}{H^1(\Oe)}.
\end{equation}
Summing \eqref{eq:distribution-product-chain} over $i=1,2,3$ and using $\sum_{i=1}^3\omega_i\bM_0\be_i=a\bM_0\bv$ yields
\begin{align*}
 &\int_{\Oe}a\sum_{i=1}^3v_i\nabla\boldsymbol W_\eps^i:
 \nabla\bz_\eps\,\mathrm d\bx
 -\int_\Omega a\bM_0\bv\cdot\widetilde{\bz}_\eps\,\mathrm d\bx=E_{\mathrm{cap}}-E_{\mathrm{comm}}.
\end{align*}

Next, we use the identities
\[
 \nabla\bU_\eps^0
 =\sum_{i=1}^3v_i\nabla\boldsymbol W_\eps^i
 +\sum_{i=1}^3\boldsymbol W_\eps^i\otimes\nabla v_i,
 \qquad
 \nabla\bv=\sum_{i=1}^3\be_i\otimes\nabla v_i
\]
to compare the two gradient terms.
Since $\nabla\widetilde{\bz}_\eps=0$ a.e.\ in $\He$,
\begin{align}
 &\int_{\Oe}a\nabla\bU_\eps^0:\nabla\bz_\eps\,\mathrm d\bx
 -\int_\Omega a\nabla\bv:\nabla\widetilde{\bz}_\eps\,\mathrm d\bx
 -\int_\Omega a\bM_0\bv\cdot\widetilde{\bz}_\eps\,\mathrm d\bx
 \notag\\
 &\quad=E_{\mathrm{cap}}-E_{\mathrm{comm}}
 +\underbrace{\int_{\Oe}a\sum_{i=1}^3
 (\boldsymbol W_\eps^i-\be_i)\otimes\nabla v_i:
 \nabla\bz_\eps\,\mathrm d\bx}_{=:E_{\mathrm{slow}}}.
 \label{eq:weighted-chain-two}
\end{align}
The slowly varying replacement is controlled by
\begin{equation}\label{eq:E-slow-bound}
 |E_{\mathrm{slow}}|
 \le C\eps\norm{a}{L^\infty}
 \norm{\bv}{W^{1,\infty}}
 \norm{\bz_\eps}{H^1(\Oe)}.
\end{equation}

Finally, set $\boldsymbol b_\eps=\cB_\eps^{\mathrm{perf}}[g_\eps]$ such that $\cR_\eps\bv=\bU_\eps^0-\boldsymbol b_\eps$, and define
\[
 E_{\mathrm{rep}}
 :=\int_{\Oe}a\nabla\boldsymbol b_\eps:
 \nabla\bz_\eps\,\mathrm d\bx.
\]
By \eqref{eq:DFL-repair},
\begin{equation}\label{eq:E-rep-bound}
 |E_{\mathrm{rep}}|
 \le C\eps\norm{a}{L^\infty}
 \norm{\bv}{W^{1,\infty}}
 \norm{\bz_\eps}{H^1(\Oe)}.
\end{equation}
The four error terms $(E_{\mathrm{cap}},E_{\mathrm{comm}},E_{\mathrm{slow}},E_{\mathrm{rep}})$ are estimated in \eqref{eq:E-cap-bound}, \eqref{eq:E-comm-bound}, \eqref{eq:E-slow-bound} and \eqref{eq:E-rep-bound}, respectively. Rearranging \eqref{eq:weighted-chain-two} gives
\begin{align}
 \mathfrak R_\eps^{\mathrm{full}}[a,\bv,\bz_\eps]
 :={}&\int_\Omega a\nabla\bv:\nabla\widetilde{\bz}_\eps\,\mathrm d\bx
 +\int_\Omega a\bM_0\bv\cdot\widetilde{\bz}_\eps\,\mathrm d\bx
 -\int_{\Oe}a\nabla(\cR_\eps\bv):\nabla\bz_\eps\,\mathrm d\bx
 \notag\\
 ={}&-E_{\mathrm{cap}}+E_{\mathrm{comm}}-E_{\mathrm{slow}}+E_{\mathrm{rep}}.
 \label{eq:full-gradient-residual-chain}
\end{align}
Adding the four error estimates gives
\begin{equation}\label{eq:full-gradient-residual-bound}
 |\mathfrak R_\eps^{\mathrm{full}}[a,\bv,\bz_\eps]|
 \le C\eps\norm{a}{W^{1,\infty}}
 \norm{\bv}{W^{1,\infty}}
 \norm{\bz_\eps}{H^1(\Oe)}.
\end{equation}
By density, \eqref{eq:full-gradient-residual-chain}--\eqref{eq:full-gradient-residual-bound} extend from smooth solenoidal fields to every $\bz_\eps\in H_{0,\sigma}^1(\Oe)$.

To pass from the full-gradient identity to the symmetric-gradient form without integrating by parts across the internal hole boundaries, let $U$ be a Lipschitz domain, $a\in W^{1,\infty}(U)$, and let $\bu,\bz\in H_{0,\sigma}^1(U)$. Then
\begin{align}
 2\int_Ua\,\cD(\bu):\cD(\bz)\,\mathrm d\bx
 &=\int_Ua\nabla\bu:\nabla\bz\,\mathrm d\bx
 +\sum_{i,j=1}^3\int_Ua(\partial_j u_i)(\partial_i z_j)\,\mathrm d\bx\notag\\
 &=\int_Ua\nabla\bu:\nabla\bz\,\mathrm d\bx
 -\sum_{i,j=1}^3\int_Uu_i(\partial_j a)(\partial_i z_j)\,\mathrm d\bx,
 \label{eq:sym-full-identity}
\end{align}
where the boundary term vanishes since $\bu$ has zero trace and the remaining second-derivative term vanishes due to $\diver\bz=0$.

Apply \eqref{eq:sym-full-identity} first on $U=\Oe$ to $(\cR_\eps\bv,\bz_\eps)$ and then on $U=\Omega$ to $(\bv,\widetilde{\bz}_\eps)$. Performing the integrations separately keeps every boundary calculation inside its natural domain and avoids traces of $\cR_\eps\bv-\bv$ on the hole boundaries. Substituting the two identities into \eqref{eq:weighted-residual-definition} gives
\begin{align}
 \mathfrak R_\eps[a,\bv,\bz_\eps]
 ={}&\mathfrak R_\eps^{\mathrm{full}}[a,\bv,\bz_\eps]
 +\sum_{i,j=1}^3\int_{\Oe}
 \bigl((\cR_\eps\bv)_i-v_i\bigr)
 (\partial_j a)(\partial_i z_{\eps,j})\,\mathrm d\bx.
 \label{eq:symmetric-residual-chain}
\end{align}
For the coefficient commutator in \eqref{eq:symmetric-residual-chain}, one has
\begin{align*}
 \left|\sum_{i,j=1}^3\int_{\Oe}
 \bigl((\cR_\eps\bv)_i-v_i\bigr)
 (\partial_j a)(\partial_i z_{\eps,j})\,\mathrm d\bx\right|
 &\le
 C\norm{\widetilde{\cR_\eps\bv}-\bv}{L^2(\Omega)}
 \norm{\nabla a}{L^\infty(\Omega)}
 \norm{\nabla\bz_\eps}{L^2(\Oe)}\\
 &\le
 C\eps\norm{a}{W^{1,\infty}}
 \norm{\bv}{W^{1,\infty}}
 \norm{\bz_\eps}{H^1(\Oe)}.
\end{align*}
We combine this estimate with \eqref{eq:full-gradient-residual-bound} to obtain \eqref{eq:weighted-residual}.

Since both $\bW_\eps$ and the chosen right inverse $\cB_\eps^{\mathrm{perf}}$ are independent of time, the construction commutes with Bochner differentiation. Indeed,
\[
 \cR_\eps\bv
 =\bW_\eps\bv
 -\cB_\eps^{\mathrm{perf}}\!\left[(\bW_\eps^{\textup T}-\bI):\nabla\bv\right]
\]
is a bounded linear expression in $\bv$. Hence Bochner differentiation gives, for $\bv\in W^{1,\infty}(0,T;\cX(\Omega))$,
\[
 \partial_t(\cR_\eps\bv)
 =\bW_\eps\partial_t\bv
 -\cB_\eps^{\mathrm{perf}}\!\left[(\bW_\eps^{\textup T}-\bI):\nabla\partial_t\bv\right]
 =\cR_\eps(\partial_t\bv),
\]
for a.e. $t\in(0,T)$, which is \eqref{eq:time-commutation}.
\end{proof}

For $\bz_\eps\in H_0^1(\Oe;\R^3)$, let $\widetilde{\bz}_\eps$ denote the zero extension to $\Omega$. Then
\[
 \widetilde{\bz}_\eps\in H_0^1(\Omega;\R^3),
 \qquad
 \|\widetilde{\bz}_\eps\|_{H^1(\Omega)}=\|\bz_\eps\|_{H^1(\Oe)},
 \qquad
 \|\cD(\widetilde{\bz}_\eps)\|_{L^2(\Omega)}=\|\cD(\bz_\eps)\|_{L^2(\Oe)}.
\]
The fixed-domain Korn inequality therefore gives
\begin{equation}\label{eq:korn}
 \norm{\bz_\eps}{H^1(\Oe)}
 \le C\norm{\cD(\bz_\eps)}{L^2(\Oe)},
\end{equation}
with $C$ uniform in $\eps$.

The zero extension through the holes yields a matching lower bound for the $L^6(\Omega)$ restriction error. Thus the restriction estimate itself has the sharp scale $\Theta(\eps)$ for nonzero comparison fields, although this does not provide a lower bound for the full relative-energy remainder in \Cref{thm:main}.

\begin{proposition}[Sharp scale of the zero-extension restriction error]
\label{prop:restriction-sharp-scale}
Let $\bv\in\cX(\Omega)$ and let $\bv_\eps\in L^6(\Omega;\R^3)$ vanish a.e.\ in $\He$. Then
\begin{equation}\label{eq:restriction-lower-scale}
 \liminf_{\eps\to0}\eps^{-1}
       \|\bv_\eps-\bv\|_{L^6(\Omega)}
 \ge |T_0|^{1/6}\|\bv\|_{L^6(\Omega)}.
\end{equation}
In particular, for each nonzero $\bv\in\cX(\Omega)$ there exist constants $c_{\bv},C_{\bv}>0$ such that, for all sufficiently small $\eps$,
\begin{equation}\label{eq:restriction-sharp-two-sided}
 c_{\bv}\eps\le\|\widetilde{\cR_\eps\bv}-\bv\|_{L^6(\Omega)}
                  \le C_{\bv}\eps.
\end{equation}
\end{proposition}

\begin{proof}
Since $\bv_\eps=0$ in $\He$,
\[
 \|\bv_\eps-\bv\|_{L^6(\Omega)}^6
 \ge\int_{\He}|\bv|^6\,\mathrm d\bx.
\]
The cell geometry and the Lipschitz continuity of $|\bv|^6$ give
\[
\begin{aligned}
 \eps^{-6}\int_{\He}|\bv|^6\,\mathrm d\bx
 &=\eps^3\sum_{\boldsymbol k\in\cK_\eps}
       \int_{T_0}|\bv(\eps\boldsymbol k+\eps^3\by)|^6
                                                   \,\mathrm d\by\\
 &=|T_0|\eps^3\sum_{\boldsymbol k\in\cK_\eps}
                    |\bv(\eps\boldsymbol k)|^6+O(\eps^3)
 \to |T_0|\int_\Omega|\bv|^6\,\mathrm d\bx.
\end{aligned}
\]
For the last limit, the full-cell Riemann sums converge to the integral, while the omitted boundary strip has measure $O(\eps)$ by \eqref{eq:boundary-strip-definition}. Taking sixth roots establishes \eqref{eq:restriction-lower-scale}. Applying this estimate with $\bv_\eps=\widetilde{\cR_\eps\bv}$ and then using \eqref{eq:restriction-estimates}, we obtain \eqref{eq:restriction-sharp-two-sided}.
\end{proof}

\section{Uniform a priori bounds and thermal stability}\label{sec:thermal}

With the geometric restriction estimates in hand, we turn to the fixed-domain thermal problem. The thermal operator satisfies
\[
 \kmin|\boldsymbol\xi|^2\le \ke(\bx)|\boldsymbol\xi|^2
 \le \kmax|\boldsymbol\xi|^2
 \qquad\text{a.e.\ in }\Omega,
\]
so all elliptic constants in the following estimates are uniform in $\eps$.

\begin{lemma}[Uniform bounds]\label{lem:uniform}
Let $(\rhoe,\ue,\te)$ be a weak solution under the hypotheses of \Cref{thm:main}. There exists $C>0$, uniform in $\eps$ and independent of the particular weak solution, such that
\begin{equation}\label{eq:uniform-bounds}
 \esssup_{0\le t\le T}\norm{\widetilde{\ue}(t)}{L^{2}(\Omega)}^{2}
 +\int_{0}^{T}\norm{\nabla\widetilde{\ue}}{L^{2}(\Omega)}^{2}\,\mathrm dt
 +\esssup_{0\le t\le T}\norm{\te(t)}{H^{1}(\Omega)}^{2}
 \le C\left(1+\norm{\uzeroe}{L^2(\Oe)}^2\right),
\end{equation}
and
\[
 \rhomin\le\widehat{\rhoe}(t,\bx)\le\rhomax
 \quad\text{for a.e. }(t,\bx)\in(0,T)\times\Omega.
\]
\end{lemma}

\begin{proof}
Renormalized transport preserves the density bounds. For the velocity and the time-integrated dissipation, we start from the combined energy inequality \eqref{eq:combined-energy}. Since $\rhoe\le\rhomax$ and $\bfv\in L^\infty$, the fixed-domain Poincar\'e--Korn inequality and Young's inequality give, for $\eta>0$,
\[
 \left|\int_{\Oe}\rhoe\bfv\cdot\ue\,\mathrm d\bx\right|
 \le C\norm{\widetilde{\ue}}{L^2(\Omega)}
 \le \eta\norm{\nabla\widetilde{\ue}}{L^2(\Omega)}^2+C_\eta.
\]
Choosing $\eta$ small enough and using $\rhoe\ge\rhomin$ and $\mu(\te)\ge\mumin$ in \eqref{eq:combined-energy} yields
\[
 \esssup_{0\le t\le T}\norm{\widetilde{\ue}(t)}{L^2(\Omega)}^2
 +\int_0^T\norm{\nabla\widetilde{\ue}}{L^2(\Omega)}^2\,\mathrm dt
 +\int_0^T\norm{\nabla\te}{L^2(\Omega)}^2\,\mathrm dt
 \le C\left(1+\norm{\uzeroe}{L^2(\Oe)}^2\right).
\]
To recover the pointwise-in-time thermal bound, test \eqref{eq:weak-heat} by $\te(t)$ for a.e. $t\in(0,T)$. Using the zero-mean condition and the fixed-domain Poincar\'e--Wirtinger inequality, we obtain
\[
 \kmin\norm{\nabla\te(t)}{L^2(\Omega)}^2
 \le C\norm{\widetilde{\ue}(t)}{L^2(\Omega)}
       \norm{\nabla\te(t)}{L^2(\Omega)}.
\]
If $\nabla\te(t)=0$, the desired estimate is immediate. Otherwise, division by $\|\nabla\te(t)\|_{L^2(\Omega)}$ and another application of Poincar\'e--Wirtinger give
\[
 \norm{\te(t)}{H^1(\Omega)}
 \le C\norm{\widetilde{\ue}(t)}{L^2(\Omega)}
 \quad\text{for a.e. }t\in(0,T).
\]
Taking the essential supremum proves \eqref{eq:uniform-bounds}.
\end{proof}

To compare the microscopic and effective temperatures on the fixed domain, set
\[
 \delta\vartheta_{\eps}=\te-\tf,
 \qquad
 \delta\bu_{\eps}=\widetilde{\ue}-\uf.
\]
Writing $-\kf\Delta\tf=-\diver(\kf\nabla\tf)=\uf\cdot\bg$ and subtracting it from the microscopic heat equation gives
\begin{equation}\label{eq:thermal-difference}
 \left\{
 \begin{aligned}
 -\diver(\ke\nabla\delta\vartheta_{\eps})
 &=\delta\bu_{\eps}\cdot\bg
   +\diver\!\bigl((\ke-\kf)\nabla\tf\bigr),
 &&\text{in }\Omega,\\
 \ke\nabla\delta\vartheta_{\eps}\cdot\bn&=0,
 &&\text{on }\partial\Omega,\\
 \int_{\Omega}\delta\vartheta_{\eps}\,\mathrm d\bx&=0,
 \end{aligned}
 \right.
\end{equation}
Since $\ke-\kf$ is supported in $\He\subset\subset\Omega$, the Neumann compatibility condition follows. By incompressibility and the zero normal trace,
\[
 \int_{\Omega}\delta\bu_{\eps}\cdot\nabla F\,\mathrm d\bx=0.
\]

Coercivity and Poincar\'e--Wirtinger provide the $H^1$ control for \eqref{eq:thermal-difference}. Pointwise stability requires the following zero-mean Neumann estimate.

\begin{lemma}[Zero-mean Neumann $L^\infty$ estimate]\label{lem:neumann-linf-general}
Let $A\in L^\infty(\Omega;\R^{3\times3})$ satisfy, for some constants $0<\lambda\le\Lambda<\infty$,
\[
 A(\bx)\boldsymbol\xi\cdot\boldsymbol\xi\ge\lambda|\boldsymbol\xi|^2,
 \qquad |A(\bx)|\le\Lambda
 \quad\text{for a.e. }\bx\in\Omega\text{ and all }\boldsymbol\xi\in\R^3.
\]
Suppose $w\in H^1(\Omega)$, $\int_\Omega w\,\mathrm d\bx=0$, satisfies
\[
 \int_\Omega A\nabla w\cdot\nabla\varphi\,\mathrm d\bx
 =\int_\Omega h\varphi\,\mathrm d\bx
 -\int_\Omega\boldsymbol H\cdot\nabla\varphi\,\mathrm d\bx
 \qquad\text{for any }\varphi\in H^1(\Omega),
\]
with $h\in L^2(\Omega)$ and $\boldsymbol H\in L^6(\Omega;\R^3)$. Then
\begin{equation}\label{eq:neumann-Linf}
 \|w\|_{L^\infty(\Omega)}
 \le C_{\Omega,\lambda,\Lambda}
 \left(\|h\|_{L^2(\Omega)}+\|\boldsymbol H\|_{L^6(\Omega)}+\|w\|_{L^2(\Omega)}\right).
\end{equation}
The constant depends only on $\Omega$, $\lambda$ and $\Lambda$.
\end{lemma}

\begin{proof}
The proof adapts Stampacchia's iteration \cite{Stampacchia} to the zero-mean Neumann setting. We begin with the uniform Poincar\'e--Sobolev inequality needed for the truncation argument. For $\gamma>0$ there exists $C=C(\Omega,\gamma)$ such that, for measurable $E\subset\Omega$ with $|E|\ge\gamma|\Omega|$ and $v\in H^1(\Omega)$ with $v=0$ a.e.\ on $E$,
\[
 \|v\|_{L^6(\Omega)}\le C\|\nabla v\|_{L^2(\Omega)}.
\]
Indeed, otherwise one could find $E_n$ and $v_n$ with $|E_n|\ge\gamma|\Omega|$, $v_n=0$ on $E_n$, $\|v_n\|_6=1$ and $\|\nabla v_n\|_2\to0$. Since $H^1(\Omega)\hookrightarrow\!\hookrightarrow L^2(\Omega)$, a subsequence converges strongly in $L^2$. The vanishing gradients identify the limit with a constant $c$. Since $|c|^2\gamma|\Omega|\le\int_{E_n}|v_n-c|^2\,\mathrm d\bx\to0$, one has $c=0$. Thus $v_n\to0$ in $L^2(\Omega)$, which together with $\|\nabla v_n\|_{L^2}\to0$ and the three-dimensional Sobolev embedding, gives $\|v_n\|_{L^6}\le C\|v_n\|_{H^1}\to0$, contradicting $\|v_n\|_{L^6}=1$.

If $w=0$ in $L^2(\Omega)$, the conclusion is immediate. Otherwise set $k_0=2|\Omega|^{-1/2}\|w\|_{L^2(\Omega)}$, then for $A_k=\{|w|>k\}$, $k\ge k_0$, Chebyshev's inequality gives $|A_k|\le|\Omega|/4$. The truncation $\phi_k=(|w|-k)_+\operatorname{sgn}w$ belongs to $H^1(\Omega)$ by Lipschitz composition. Since it vanishes on a measurable set of measure at least $3|\Omega|/4$, we can use the preceding zero-set inequality to obtain
\[
 \|\phi_k\|_{L^6(\Omega)}\le C\|\nabla\phi_k\|_{L^2(\Omega)}.
\]
Testing the weak equation with $\phi_k$ and using H\"older's inequality yields
\[
 \lambda\|\nabla\phi_k\|_2^2
 \le C\|h\|_2|A_k|^{1/3}\|\phi_k\|_6
 +C\|\boldsymbol H\|_6|A_k|^{1/3}\|\nabla\phi_k\|_2.
\]
With $D=\|h\|_2+\|\boldsymbol H\|_6$, we have
\[
 \|\phi_k\|_6\le CD|A_k|^{1/3}.
\]
Hence, for $\ell>k\ge k_0$,
\[
 (\ell-k)|A_\ell|^{1/6}\le CD|A_k|^{1/3},
 \qquad
 |A_\ell|\le \frac{CD^6}{(\ell-k)^6}|A_k|^2.
\]
If $D=0$, testing with $w$ gives $\nabla w=0$, and the zero mean implies $w=0$. For $D>0$, take $k_n=k_0+K(1-2^{-n})$ and $Y_n=|A_{k_n}|/|\Omega|$. Since $k_{n+1}-k_n=K2^{-(n+1)}$, the preceding level-set estimate yields
\[
 Y_{n+1}\le C_*2^{6(n+1)}\left(\frac{D}{K}\right)^6Y_n^2,
\]
where $C_*$ depends only on $\Omega$, $\lambda$ and $\Lambda$. Choose $K=C_0D$ with $C_0$ so large that $C_*C_0^{-6}Y_0\le2^{-12}$. By $Y_0\le1/4$, an induction gives $Y_n\le Y_0 2^{-6n}$ for all $n$. Hence $Y_n\to0$ and $|\{|w|>k_0+K\}|=0$. Therefore $\|w\|_\infty\le k_0+K$, which is \eqref{eq:neumann-Linf}.
\end{proof}
For $A=\ke\bI$, one may take $\lambda=\kmin$ and $\Lambda=\kmax$, so \eqref{eq:neumann-Linf} applies uniformly in $\eps$. The energy estimate and \Cref{lem:neumann-linf-general} give the two temperature bounds needed in the relative-energy argument.

\begin{lemma}[Thermal stability]\label{lem:thermal}
Under the hypotheses of \Cref{thm:main}, there exists $C>0$, uniform in $\eps$, such that for a.e. $t\in(0,T)$,
\begin{equation}\label{eq:thermal-H1}
 \norm{\te(t)-\tf(t)}{H^{1}(\Omega)}
 \le C\left(
 \norm{\widetilde{\ue}(t)-\uf(t)}{L^{2}(\Omega)}+\eps^{3}
 \right),
\end{equation}
\begin{equation}\label{eq:thermal-Linf}
 \norm{\te(t)-\tf(t)}{L^{\infty}(\Omega)}
 \le C\left(
 \norm{\widetilde{\ue}(t)-\uf(t)}{L^{2}(\Omega)}+\eps
 \right),
\end{equation}
and
\begin{equation}\label{eq:mu-difference}
 \norm{\mu(\te(t))-\mu(\tf(t))}{L^{\infty}(\Omega)}
 \le C\left(
 \norm{\widetilde{\ue}(t)-\uf(t)}{L^{2}(\Omega)}+\eps
 \right).
\end{equation}
\end{lemma}

\begin{proof}
Testing \eqref{eq:thermal-difference} by $\delta\vartheta_{\eps}$ gives
\[
 \kmin\norm{\nabla\delta\vartheta_{\eps}}{L^{2}}^{2}
 \le
 C\norm{\delta\bu_{\eps}}{L^{2}}
 \norm{\delta\vartheta_{\eps}}{L^{2}}
 +\abs{\ks-\kf}\,
 \norm{\mathbf 1_{\He}\nabla\tf}{L^{2}}
 \norm{\nabla\delta\vartheta_{\eps}}{L^{2}}.
\]
Since both temperatures satisfy the same zero-mean reference condition, we have $\int_\Omega\delta\vartheta_\eps\,\mathrm d\bx=0$. The Poincar\'e--Wirtinger inequality on $\Omega$ therefore applies with a constant uniform in $\eps$. Moreover, \eqref{eq:volume-holes} gives
\[
 \norm{\mathbf 1_{\He}\nabla\tf}{L^{2}}
 \le \abs{\He}^{1/2}\norm{\nabla\tf}{L^{\infty}}
 \le C\eps^{3}.
\]
Thus \eqref{eq:thermal-H1} holds.

For the viscosity term, \Cref{lem:neumann-linf-general} is applied to the difference equation. Take
\[
 A=\ke\bI,
 \qquad h=\delta\bu_\eps\cdot\bg,
 \qquad \boldsymbol H=(\ke-\kf)\nabla\tf,
 \qquad w=\delta\vartheta_\eps.
\]
By \Cref{lem:uniform} and \Cref{ass:strong}, the data in \eqref{eq:thermal-difference} satisfy $\delta\bu_\eps\in L^\infty(0,T;L^2(\Omega))$. Moreover,
\[
 \norm{\boldsymbol H}{L^6(\Omega)}
 \le |\ks-\kf|\,|\He|^{1/6}
 \norm{\nabla\tf}{L^\infty(\Omega)}
 \le C\eps,
\]
and $\norm{h}{L^2(\Omega)}\le C\norm{\delta\bu_\eps}{L^2(\Omega)}$. At every time for which \eqref{eq:thermal-H1} holds, the estimate \eqref{eq:neumann-Linf} applies with $\|w\|_{L^2}\le C(\|\delta\bu_\eps\|_{L^2}+\eps^3)$. Combining these estimates gives \eqref{eq:thermal-Linf}. The Lipschitz continuity of $\mu$ then yields
\[
 \|\mu(\te)-\mu(\tf)\|_{L^\infty(\Omega)}
 \le \|\mu'\|_{L^\infty(\R)}
 \|\te-\tf\|_{L^\infty(\Omega)}.
\]
Together with \eqref{eq:thermal-Linf}, this inequality proves \eqref{eq:mu-difference}.
\end{proof}

The estimates in \Crefrange{lem:uniform}{lem:thermal} provide the fixed-domain thermal control required in the relative-energy argument. In addition, \eqref{eq:restriction-estimates} gives
\begin{equation}\label{eq:velocity-comparison}
 \norm{\widetilde{\ue}-\uf}{L^{2}(\Omega)}
 \le \norm{\boldsymbol{e}_{\eps}}{L^{2}(\Oe)}
 +C\eps\norm{\uf}{W^{2,\infty}(\Omega)}.
\end{equation}
Combining \eqref{eq:mu-difference}--\eqref{eq:velocity-comparison},
\begin{equation}\label{eq:mu-error-e}
 \norm{\mu(\te)-\mu(\tf)}{L^{\infty}(\Omega)}
 \le C\bigl(\norm{\boldsymbol{e}_{\eps}}{L^{2}(\Oe)}+\eps\bigr)
 \qquad\text{for a.e. }t\in(0,T).
\end{equation}
Thus \eqref{eq:mu-error-e} reduces the viscosity mismatch to the relative velocity.

\section{Relative energy and proof of the main theorem}\label{sec:relative}

With the thermal mismatch controlled, we turn to the transported densities on the fixed domain. Since both velocities are divergence-free, their difference satisfies the following stability estimate.

\begin{lemma}[Density stability]\label{lem:density}
Under the hypotheses of \Cref{thm:main}, for a.e. $\tau\in(0,T)$,
\begin{equation}\label{eq:density-ineq}
 \norm{q_{\eps}(\tau)}{L^{2}(\Omega)}^{2}
 \le \norm{q_{\eps}(0)}{L^{2}(\Omega)}^{2}
 +C\int_{0}^{\tau}
 \left(
 \norm{q_{\eps}}{L^{2}(\Omega)}^{2}
 +\norm{\boldsymbol{e}_{\eps}}{L^{2}(\Oe)}^{2}
 +\eps^{2}
 \right)\,\mathrm dt.
\end{equation}
\end{lemma}

\begin{proof}
Both densities satisfy transport equations on the fixed domain. Subtracting them and using $\diver\widetilde{\ue}=\diver\uf=0$ gives
\begin{equation*}
 \partial_tq_{\eps}+\widetilde{\ue}\cdot\nabla q_{\eps}
 =f_{\eps},
 \qquad
 f_{\eps}:=-(\widetilde{\ue}-\uf)\cdot\nabla\rhof.
\end{equation*}
The bounds in \Cref{lem:uniform} and \Cref{ass:strong} imply
\[
 q_{\eps}\in L^\infty((0,T)\times\Omega),
 \qquad
 f_{\eps}\in L^1(0,T;L^2(\Omega)),
 \qquad
 \widetilde{\ue}\in L^1(0,T;W_0^{1,2}(\Omega;\R^3)).
\]
Moreover, $q_{\eps}$ takes values in the bounded interval $[-(\rhomax-\rhomin),\,\rhomax-\rhomin]$. To justify passage across the boundary, write $d(\bx)=\operatorname{dist}(\bx,\partial\Omega)$ and choose smooth cutoffs $\chi_h\in C_c^\infty(\Omega)$ such that
\[
 \chi_h=0\ \text{on }\{d<h\},\qquad
 \chi_h=1\ \text{on }\{d>2h\},\qquad
 |\nabla\chi_h|\le C h^{-1}.
\]
For any $\varphi\in C_c^\infty([0,T)\times\R^3)$, the restriction of $\chi_h\varphi$ to $[0,T)\times\Omega$ is an admissible test function. Hardy's inequality and the boundary-layer volume bound give
\begin{align*}
 \left|\int_0^T\!\int_\Omega q_\eps\widetilde{\ue}\cdot\nabla\chi_h\,\varphi \,\mathrm d\bx\,\mathrm dt\right|
 &\le C\|q_\eps\|_{L^\infty}\|\varphi\|_{L^\infty}
 |\{d<2h\}|^{1/2}\left\|\frac{\widetilde{\ue}}{d}\right\|_{L^1(0,T;L^2(\Omega))}\\
 &\le C h^{1/2}\|q_\eps\|_{L^\infty}\|\varphi\|_{L^\infty}
 \|\nabla\widetilde{\ue}\|_{L^1(0,T;L^2(\Omega))}\to0.
\end{align*}
Let $\overline q_\eps$, $\overline f_\eps$ and $\overline{\boldsymbol u}_\eps$ denote the zero extensions of $q_\eps$, $f_\eps$ and $\widetilde{\ue}$ from $\Omega$ to $\R^3$. Testing the difference equation with $\chi_h\varphi$, including the initial term, and passing $h\to0^+$ yields
\[
 \partial_t\overline q_\eps
 +\diver\!\bigl(\overline q_\eps\,\overline{\boldsymbol u}_\eps\bigr)
 =\overline f_\eps
 \quad\text{in }\mathcal D'((0,T)\times\R^3),
 \qquad
 \overline q_\eps|_{t=0}=\overline q_\eps(0).
\]
Moreover,
\[
 \overline{\boldsymbol u}_\eps\in L^1(0,T;W^{1,2}(\R^3;\R^3)),
 \qquad
 \diver\overline{\boldsymbol u}_\eps=0\quad\text{in }\mathcal D'((0,T)\times\R^3),
\]
so the preceding equation is equivalently
\[
 \partial_t\overline q_\eps
 +\overline{\boldsymbol u}_\eps\cdot\nabla\overline q_\eps
 =\overline f_\eps
 \quad\text{in }\mathcal D'((0,T)\times\R^3).
\]
The DiPerna--Lions commutator lemma \cite[Lemma~II.1]{DiPernaLions} therefore applies with a smooth function agreeing with $\beta(s)=s^2/2$ on the range of $q_\eps$ and satisfying $\beta(0)=0$. Restricting the resulting identity to $\Omega$ gives
\begin{equation}\label{eq:density-renormalized}
 \partial_t\!\left(\frac{|q_{\eps}|^2}{2}\right)
 +\diver\!\left(
   \frac{|q_{\eps}|^2}{2}\widetilde{\ue}
 \right)
 =q_{\eps}f_{\eps}
 =-q_{\eps}(\widetilde{\ue}-\uf)\cdot\nabla\rhof.
\end{equation}
The zero trace of $\widetilde{\ue}$ makes the boundary flux vanish. Integrating \eqref{eq:density-renormalized} over $(0,\tau)\times\Omega$ therefore gives
\[
 \frac12\|q_{\eps}(\tau)\|_{L^2(\Omega)}^2
 -\frac12\|q_{\eps}(0)\|_{L^2(\Omega)}^2
 =
 -\int_0^\tau\!\int_\Omega
 q_{\eps}(\widetilde{\ue}-\uf)\cdot\nabla\rhof
 \,\mathrm d\bx\,\mathrm dt.
\]

Set $M_\rho:=\|\nabla\rhof\|_{L^\infty((0,T)\times\Omega)}$. Then
\begin{align*}
 \left|
 \int_0^\tau\!\int_\Omega
 q_{\eps}(\widetilde{\ue}-\uf)\cdot\nabla\rhof
 \,\mathrm d\bx\,\mathrm dt
 \right|
 &\le
 M_\rho\int_0^\tau
 \|q_{\eps}\|_{L^2(\Omega)}
 \|\widetilde{\ue}-\uf\|_{L^2(\Omega)}
 \,\mathrm dt\\
 &\le
 C\int_0^\tau
 \|q_{\eps}\|_{L^2(\Omega)}
 \bigl(
   \|\boldsymbol e_{\eps}\|_{L^2(\Oe)}+\eps
 \bigr)\,\mathrm dt,
\end{align*}
where the last step uses \eqref{eq:velocity-comparison} and the $W^{2,\infty}$ bound for $\uf$ from \Cref{ass:strong}. Applying Young's inequality, we obtain
\[
 \left|
 \int_0^\tau\!\int_\Omega
 q_{\eps}(\widetilde{\ue}-\uf)\cdot\nabla\rhof
 \,\mathrm d\bx\,\mathrm dt
 \right|
 \le
 C\int_0^\tau
 \left(
  \|q_{\eps}\|_{L^2(\Omega)}^2
  +\|\boldsymbol e_{\eps}\|_{L^2(\Oe)}^2
  +\eps^2
 \right)\,\mathrm dt.
\]
Combining this estimate with the preceding identity proves \eqref{eq:density-ineq}.
\end{proof}

The kinetic error is treated with a relative-energy inequality for a general solenoidal comparison field.

\begin{lemma}[Relative kinetic-energy inequality]\label{lem:relative-general}
Let $(\rhoe,\ue,\te)$ be a weak solution under the hypotheses of \Cref{thm:main}. Let $\bU$ be solenoidal and satisfy
\[
 \bU\in W^{1,\infty}(0,T;H_{0}^{1}(\Oe;\R^{3})).
\]
Then, for a.e. $\tau\in(0,T)$, there holds
\begin{equation}\label{eq:general-relative}
\begin{aligned}
 &\frac12\int_{\Oe}\rhoe(\tau)\abs{\ue(\tau)-\bU(\tau)}^{2}\,\mathrm d\bx
 +\int_{0}^{\tau}\!\int_{\Oe}
 2\mu(\te)\abs{\cD(\ue-\bU)}^{2}\,\mathrm d\bx\,\mathrm dt\\
 &\quad\le
 \frac12\int_{\Oe}\rhozeroe\abs{\uzeroe-\bU(0)}^{2}\,\mathrm d\bx-\int_{0}^{\tau}\!\int_{\Oe}
 \rhoe\bigl(\partial_t\bU+\ue\cdot\nabla\bU\bigr)
 \cdot(\ue-\bU)\,\mathrm d\bx\,\mathrm dt\\
 &\qquad
 +\int_{0}^{\tau}\!\int_{\Oe}
 (\rhoe\bfv-\te\bg)\cdot(\ue-\bU)\,\mathrm d\bx\,\mathrm dt -\int_{0}^{\tau}\!\int_{\Oe}
 2\mu(\te)\cD(\bU):\cD(\ue-\bU)\,\mathrm d\bx\,\mathrm dt.
\end{aligned}
\end{equation}
\end{lemma}

\begin{proof}
Fix $\eps>0$. The density result \cite[Theorem~III.4.2]{Galdi} in $H_{0,\sigma}^1(\Oe)$, followed by time mollification, gives smooth solenoidal approximations $\bU_n$ such that
\[
 \bU_n\to\bU\quad\text{in }C([0,T];H_0^1(\Oe;\R^3)),
 \qquad
 \partial_t\bU_n\to\partial_t\bU
 \quad\text{in }L^1(0,T;H_0^1(\Oe;\R^3)).
\]
Choosing $\boldsymbol\Psi_n=\bU_n-\bU$, we have
\begin{align*}
 \left|\int_0^T\!\int_{\Oe}
 \rhoe\ue\otimes\ue:\nabla\boldsymbol\Psi_n\,\mathrm d\bx\,\mathrm dt\right|
 &\le\rhomax\|\ue\|_{L^2L^6}^2
              \|\nabla\boldsymbol\Psi_n\|_{L^\infty L^{3/2}}\to0,\\
 \left|\int_0^T\!\int_{\Oe}
 \rhoe\ue\cdot\partial_t\boldsymbol\Psi_n\,\mathrm d\bx\,\mathrm dt\right|
 &\le\rhomax\|\ue\|_{L^\infty L^2}
              \|\partial_t\boldsymbol\Psi_n\|_{L^1L^2}\to0.
\end{align*}
For the viscous, forcing and initial terms, we obtain
\[
 \left|\int_0^T\!\int_{\Oe}
 2\mu(\te)\cD(\ue):\cD(\boldsymbol\Psi_n)\,\mathrm d\bx\,\mathrm dt\right|
 \le\! 2\mumax\|\cD(\ue)\|_{L^2(0,T;L^2(\Oe))}
       \|\cD(\boldsymbol\Psi_n)\|_{L^2(0,T;L^2(\Oe))}\!\to0,
\]
\[
 \left|\int_0^T\!\int_{\Oe}
 (\rhoe\bfv-\te\bg)\cdot\boldsymbol\Psi_n\,\mathrm d\bx\,\mathrm dt\right|
\le \|\rhoe\bfv-\te\bg\|_{L^1(0,T;L^2(\Oe))}
       \|\boldsymbol\Psi_n\|_{L^\infty(0,T;L^2(\Oe))}\to0,
\]
and
\[
 \left|\int_{\Oe}\rhozeroe\uzeroe\cdot\boldsymbol\Psi_n(0)\,\mathrm d\bx\right|
 \le \rhomax\|\uzeroe\|_{L^2(\Oe)}\|\boldsymbol\Psi_n(0)\|_{L^2(\Oe)}\to0.
\]
To retain the initial-data term, we apply \eqref{eq:weak-momentum} with time cutoffs that equal one at $t=0$ and converge to $\mathbf1_{[0,\tau]}$. Hence, for a.e. $\tau\in(0,T)$,
\[
\begin{aligned}
 &\int_{\Oe}\rhoe(\tau)\ue(\tau)\cdot\bU(\tau)\,\mathrm d\bx
 -\int_{\Oe}\rhozeroe\uzeroe\cdot\bU(0)\,\mathrm d\bx\\
 &\quad=\int_0^\tau\!\int_{\Oe}
 \bigl[\rhoe\ue\cdot\partial_t\bU
 +\rhoe\ue\otimes\ue:\nabla\bU\bigr]\,\mathrm d\bx\,\mathrm dt\\
 &\qquad+\int_0^\tau\!\int_{\Oe}
 \bigl[-2\mu(\te)\cD(\ue):\cD(\bU)
 +(\rhoe\bfv-\te\bg)\cdot\bU\bigr]\,\mathrm d\bx\,\mathrm dt.
\end{aligned}
\]

The zero extension of $\bU$ belongs to $H_0^1(\Omega)$. Moreover, $|\bU|^2/2\in L^\infty(0,T;W_0^{1,3/2}(\Omega))$ and $\widetilde{\ue}\in L^2(0,T;L^6(\Omega))$. Since $\bU,\partial_t\bU\in L^\infty(0,T;H_0^1(\Oe))\hookrightarrow L^\infty(0,T;L^6(\Oe))$, one also has $\partial_t(|\bU|^2/2)=\bU\cdot\partial_t\bU\in L^\infty(0,T;L^3(\Oe))$. By density, the renormalized continuity identity therefore admits the test function $|\bU|^2/2$, yielding
\[
\begin{aligned}
 \frac12\int_{\Oe}\rhoe(\tau)|\bU(\tau)|^2\,\mathrm d\bx
 =\frac12\int_{\Oe}\rhozeroe|\bU(0)|^2\,\mathrm d\bx
 +\int_0^\tau\!\int_{\Oe}
 \left[\rhoe\bU\cdot\partial_t\bU
 +\rhoe\ue\cdot\nabla\!\left(\frac{|\bU|^2}{2}\right)\right]
 \,\mathrm d\bx\,\mathrm dt.
\end{aligned}
\]
Adding this identity to \eqref{eq:kinetic-energy} and subtracting the momentum identity gives the kinetic part of \eqref{eq:general-relative}. Using the identity
\[
 2\mu|\cD(\ue)|^2-2\mu\cD(\ue):\cD(\bU)
 =2\mu|\cD(\ue-\bU)|^2
  +2\mu\cD(\bU):\cD(\ue-\bU),
\]
we reorganize the remaining viscous terms and obtain \eqref{eq:general-relative}. The same time-cutoff construction retains the initial terms in both testing identities. Together with \eqref{eq:kinetic-energy}, this completes the endpoint argument.
\end{proof}

By \Cref{prop:restriction} and \Cref{ass:strong}, the field $\bU_\eps=\cR_\eps\uf$ is solenoidal and belongs to $W^{1,\infty}(0,T;H_0^1(\Oe;\R^3))$. Accordingly, \Cref{lem:relative-general} applies with $\bU=\bU_\eps$. Set
\[
 \bdelta_{\eps}=\bU_{\eps}-\uf,
 \qquad
 \boldsymbol{A}=\partial_t\uf+\uf\cdot\nabla\uf.
\]
Then
\begin{equation}\label{eq:acceleration-decomp}
 \partial_t\bU_{\eps}+\bU_{\eps}\cdot\nabla\bU_{\eps}
 =\boldsymbol{A}+\boldsymbol{G}_{\eps},
\end{equation}
where
\[
 \boldsymbol{G}_{\eps}
 =\partial_t\bdelta_{\eps}
 +\bdelta_{\eps}\cdot\nabla\uf
 +\uf\cdot\nabla\bdelta_{\eps}
 +\bdelta_{\eps}\cdot\nabla\bdelta_{\eps}.
\]
The restriction estimates control each component of $\boldsymbol G_\eps$.

\begin{lemma}[Acceleration corrector]\label{lem:acceleration}
Under the hypotheses of \Cref{thm:main}, for $\zeta>0$ there exists $C_{\zeta}>0$ such that, for a.e. $t\in(0,T)$,
\begin{equation}\label{eq:acceleration-bound}
 \left|
 \int_{\Oe}\rhoe\boldsymbol{G}_{\eps}\cdot\boldsymbol{e}_{\eps}\,\mathrm d\bx
 \right|
 \le
 \zeta\norm{\cD(\boldsymbol{e}_{\eps})}{L^{2}(\Oe)}^{2}
 +C_{\zeta}\left(
 \norm{\boldsymbol{e}_{\eps}}{L^{2}(\Oe)}^{2}+\eps^{2}
 \right).
\end{equation}
In addition,
\begin{equation}\label{eq:quadratic-convection}
 \left|
 \int_{\Oe}\rhoe
 (\boldsymbol{e}_{\eps}\cdot\nabla\bU_{\eps})
 \cdot\boldsymbol{e}_{\eps}\,\mathrm d\bx
 \right|
 \le
 \zeta\norm{\cD(\boldsymbol{e}_{\eps})}{L^{2}(\Oe)}^{2}
 +C_{\zeta}\norm{\boldsymbol{e}_{\eps}}{L^{2}(\Oe)}^{2}.
\end{equation}
\end{lemma}

\begin{proof}
By \Cref{prop:restriction} and \Cref{ass:strong}, both $\uf(t)$ and $\partial_t\uf(t)$ belong to $\cX(\Omega)$ for a.e. $t\in(0,T)$: differentiation of the no-slip trace and of $\diver\uf=0$ gives the corresponding trace and divergence conditions for $\partial_t\uf$. Since $\partial_t\bdelta_\eps=\cR_\eps(\partial_t\uf)-\partial_t\uf$ by \eqref{eq:time-commutation}, \eqref{eq:restriction-estimates} applied to $\uf$ and $\partial_t\uf$ gives
\[
 \norm{\bdelta_{\eps}}{L^{6}(\Oe)}
 +\norm{\partial_t\bdelta_{\eps}}{L^{2}(\Oe)}
 +\norm{\nabla\bdelta_{\eps}}{L^{6/5}(\Oe)}
 \le C\eps,
 \qquad
 \norm{\nabla\bdelta_{\eps}}{L^{2}(\Oe)}\le C.
\]
Write
\[
 \boldsymbol G_\eps
 =\boldsymbol G_{1,\eps}+\boldsymbol G_{2,\eps}
 +\boldsymbol G_{3,\eps}+\boldsymbol G_{4,\eps},
\]
where
\[
 \boldsymbol G_{1,\eps}=\partial_t\bdelta_\eps,
 \quad
 \boldsymbol G_{2,\eps}=\bdelta_\eps\cdot\nabla\uf,
 \quad
 \boldsymbol G_{3,\eps}=\uf\cdot\nabla\bdelta_\eps,
 \quad
 \boldsymbol G_{4,\eps}=\bdelta_\eps\cdot\nabla\bdelta_\eps.
\]
The first two terms are of lower order and satisfy
\begin{align*}
 \left|\int_{\Oe}\rhoe\boldsymbol G_{1,\eps}\cdot\boldsymbol e_\eps\,\mathrm d\bx\right|
 \le C\eps\norm{\boldsymbol e_\eps}{L^2(\Oe)},
\quad \left|\int_{\Oe}\rhoe\boldsymbol G_{2,\eps}\cdot\boldsymbol e_\eps\,\mathrm d\bx\right|
 \le C\eps\norm{\boldsymbol e_\eps}{L^2(\Oe)}.
\end{align*}
For the third term, H\"older's inequality with exponents $(\infty,6/5,6)$ together with \eqref{eq:korn} gives
\[
 \left|\int_{\Oe}\rhoe\boldsymbol G_{3,\eps}\cdot\boldsymbol e_\eps\,\mathrm d\bx\right|
 \le C\eps\norm{\boldsymbol e_\eps}{L^6(\Oe)}
 \le C\eps\norm{\cD(\boldsymbol e_\eps)}{L^2(\Oe)}.
\]
For the quadratic term $\boldsymbol G_{4,\eps}$, we first note that
\[
 \norm{\boldsymbol G_{4,\eps}}{L^{3/2}(\Oe)}
 \le \norm{\bdelta_\eps}{L^6(\Oe)}
 \norm{\nabla\bdelta_\eps}{L^2(\Oe)}
 \le C\eps.
\]
Applying the three-dimensional Gagliardo--Nirenberg inequality to the zero extension and then using \eqref{eq:korn}, we obtain
\[
 \norm{\boldsymbol e_\eps}{L^3(\Oe)}
 \le C\norm{\boldsymbol e_\eps}{L^2(\Oe)}^{1/2}
 \norm{\cD(\boldsymbol e_\eps)}{L^2(\Oe)}^{1/2}.
\]
Hence
\[
 \left|\int_{\Oe}\rhoe\boldsymbol G_{4,\eps}\cdot\boldsymbol e_\eps\,\mathrm d\bx\right|
 \le C\eps
 \norm{\boldsymbol e_\eps}{L^2(\Oe)}^{1/2}
 \norm{\cD(\boldsymbol e_\eps)}{L^2(\Oe)}^{1/2}.
\]
Writing $A=\|\boldsymbol e_\eps\|_{L^2(\Oe)}$ and $B=\|\cD(\boldsymbol e_\eps)\|_{L^2(\Oe)}$, we apply Young's inequality with exponents $4$ and $4/3$ to obtain
\[
 C\eps A^{1/2}B^{1/2}
 \le\zeta B^2+C_\zeta\eps^{4/3}A^{2/3}
 \le\zeta B^2+C_\zeta(\eps^2+A^2).
\]
For the second inequality, apply Young's inequality to $x=A^{2/3}$ and $y=\eps^{4/3}$ with conjugate exponents $3$ and $3/2$, we have
\[
 \eps^{4/3}A^{2/3}=xy
 \le \eta x^3+C_\eta y^{3/2}
 =\eta A^2+C_\eta\eps^2.
\]
Hence the quadratic corrector contributes $\zeta B^2+C_\zeta(A^2+\eps^2)$. We assign $\zeta/2$ to each dissipative contribution from $\boldsymbol G_{3,\eps}$ and $\boldsymbol G_{4,\eps}$, while the two lower-order terms are bounded by $C(A^2+\eps^2)$. Combining the four estimates proves \eqref{eq:acceleration-bound}.

For \eqref{eq:quadratic-convection}, we use the uniform $H^{1}$ bound for $\bU_{\eps}$ and the three-dimensional Gagliardo--Nirenberg inequality to obtain
\[
 \int_{\Oe}\abs{\boldsymbol{e}_{\eps}}^{2}\abs{\nabla\bU_{\eps}}\,\mathrm d\bx
 \le C\norm{\boldsymbol{e}_{\eps}}{L^{4}(\Oe)}^{2}
 \le C\norm{\boldsymbol{e}_{\eps}}{L^{2}(\Oe)}^{1/2}
       \norm{\nabla\boldsymbol{e}_{\eps}}{L^{2}(\Oe)}^{3/2}.
\]
With $A=\|\boldsymbol e_\eps\|_{L^2(\Oe)}$ and $B=\|\cD(\boldsymbol e_\eps)\|_{L^2(\Oe)}$, we combine Korn's inequality with Young's inequality, using exponents $4/3$ and $4$, to obtain
\[
 C A^{1/2}B^{3/2}\le \zeta B^2+C_\zeta A^2.
\]
This proves \eqref{eq:quadratic-convection} and completes the argument.

\end{proof}

After substitution of the effective momentum equation, the geometric cross term is exactly the weighted residual controlled by \Cref{prop:restriction}.

\begin{proposition}[Closed relative kinetic-energy inequality]\label{prop:kinetic-closed}
Under the hypotheses of \Cref{thm:main}, for $0<\zeta<\mumin/2$ there exists $C_{\zeta}>0$ such that, for a.e. $\tau\in(0,T)$,
\begin{equation}\label{eq:closed-kinetic}
\begin{aligned}
 &\frac12\int_{\Oe}\rhoe(\tau)\abs{\boldsymbol{e}_{\eps}(\tau)}^{2}\,\mathrm d\bx
 +(2\mumin-4\zeta)
 \int_{0}^{\tau}\norm{\cD(\boldsymbol{e}_{\eps})}{L^{2}(\Oe)}^{2}\,\mathrm dt\\
 &\qquad\le
 \frac12\int_{\Oe}\rhozeroe
 \abs{\uzeroe-\cR_{\eps}\boldsymbol{u}_{0}}^{2}\,\mathrm d\bx
 +C_{\zeta}\int_{0}^{\tau}
 \left(
 \norm{q_{\eps}}{L^{2}(\Omega)}^{2}
 +\norm{\boldsymbol{e}_{\eps}}{L^{2}(\Oe)}^{2}
 \right)\,\mathrm dt
 +C_{\zeta}\eps^{2}.
\end{aligned}
\end{equation}
\end{proposition}

\begin{proof}
Insert $\bU_{\eps}$ into \eqref{eq:general-relative}. The decomposition
\[
 \ue\cdot\nabla\bU_{\eps}
 =\bU_{\eps}\cdot\nabla\bU_{\eps}
 +\boldsymbol{e}_{\eps}\cdot\nabla\bU_{\eps}
\]
separates the comparison convection from the error contribution.
The decomposition \eqref{eq:acceleration-decomp} together with \Cref{lem:acceleration} controls the acceleration and quadratic convection terms. After absorbing the corresponding dissipative contributions into the left-hand side, it remains to estimate
\[
 \int_{\Oe}
 \bigl[-\rhoe\boldsymbol{A}+\rhoe\bfv-\te\bg\bigr]
 \cdot\boldsymbol{e}_{\eps}\,\mathrm d\bx
 -\int_{\Oe}2\mu(\te)\cD(\bU_{\eps}):
 \cD(\boldsymbol{e}_{\eps})\,\mathrm d\bx.
\]
Decompose the first integrand into
\begin{align*}
 -\rhoe\boldsymbol{A}+\rhoe\bfv-\te\bg
 ={}&-\rhof\boldsymbol{A}+\rhof\bfv-\tf\bg+q_{\eps}(\bfv-\boldsymbol{A})
 -(\te-\tf)\bg.
\end{align*}
Combined with $\diver\uf=0$ and the effective continuity equation, \eqref{eq:effective-intro} gives
\[
 -\rhof\boldsymbol{A}+\rhof\bfv-\tf\bg
 =-\diver\!\bigl(2\mu(\tf)\cD(\uf)\bigr)
 +\mu(\tf)\bM_{0}\uf+\nabla\pf.
\]
Since $\widetilde{\boldsymbol{e}}_{\eps}\in H_{0}^{1}(\Omega;\R^{3})$ is solenoidal, integration by parts yields
\begin{align*}
 &\int_{\Oe}
 \bigl[-\rhof\boldsymbol{A}+\rhof\bfv-\tf\bg\bigr]
 \cdot\boldsymbol{e}_{\eps}\,\mathrm d\bx=
 \int_{\Omega}2\mu(\tf)\cD(\uf):
 \cD(\widetilde{\boldsymbol{e}}_{\eps})\,\mathrm d\bx
 +\int_{\Omega}\mu(\tf)\bM_{0}\uf\cdot
 \widetilde{\boldsymbol{e}}_{\eps}\,\mathrm d\bx.
\end{align*}
Then we add and subtract $\mu(\tf)$ in the microscopic cross term and define
\[
\begin{aligned}
 \mathfrak R_{\mathrm{geo}}
 :={}&\int_\Omega2\mu(\tf)\cD(\uf):\cD(\widetilde{\boldsymbol e}_\eps)\,\mathrm d\bx
 +\int_\Omega\mu(\tf)\bM_0\uf\cdot\widetilde{\boldsymbol e}_\eps\,\mathrm d\bx-\int_{\Oe}2\mu(\tf)\cD(\bU_\eps):\cD(\boldsymbol e_\eps)\,\mathrm d\bx.
\end{aligned}
\]
These three contributions match \eqref{eq:weighted-residual-definition} term by term. The first is the fixed-domain viscous term, the second is the Brinkman term, and the third is the perforated-domain viscous term with $\cR_\eps\uf=\bU_\eps$. Therefore, we have $\mathfrak R_{\mathrm{geo}}=\mathfrak R_\eps[\mu(\tf),\uf,\boldsymbol e_\eps]$. Using $\mu(\tf)\in W^{1,\infty}(\Omega)$, we find
\[
 \abs{\mathfrak R_{\mathrm{geo}}}
 \le C\eps
 \norm{\boldsymbol{e}_{\eps}}{H^{1}(\Oe)}
 \le \zeta\norm{\cD(\boldsymbol{e}_{\eps})}{L^{2}(\Oe)}^{2}
 +C_{\zeta}\eps^{2}.
\]
For the viscosity mismatch, write
\[
 \mathfrak R_{\mu}
 =-2\int_{\Oe}
 \bigl(\mu(\te)-\mu(\tf)\bigr)
 \cD(\bU_{\eps}):\cD(\boldsymbol{e}_{\eps})\,\mathrm d\bx.
\]
By \eqref{eq:mu-error-e} and the uniform $H^{1}$ bound for $\bU_{\eps}$,
\[
 \abs{\mathfrak R_{\mu}}
 \le \zeta\norm{\cD(\boldsymbol{e}_{\eps})}{L^{2}(\Oe)}^{2}
 +C_{\zeta}\left(
 \norm{\boldsymbol{e}_{\eps}}{L^{2}(\Oe)}^{2}+\eps^{2}
 \right).
\]

The density remainder satisfies
\[
 \left|
 \int_{\Oe}q_{\eps}(\bfv-\boldsymbol{A})
 \cdot\boldsymbol{e}_{\eps}\,\mathrm d\bx
 \right|
 \le C\left(
 \norm{q_{\eps}}{L^{2}(\Omega)}^{2}
 +\norm{\boldsymbol{e}_{\eps}}{L^{2}(\Oe)}^{2}
 \right).
\]
For the thermal force, \Cref{lem:thermal} and \eqref{eq:velocity-comparison} show
\[
 \left|
 \int_{\Oe}(\te-\tf)\bg\cdot\boldsymbol{e}_{\eps}\,\mathrm d\bx
 \right|
 \le C\left(
 \norm{\boldsymbol{e}_{\eps}}{L^{2}(\Oe)}^{2}+\eps^{2}
 \right).
\]

On the left-hand side of \eqref{eq:general-relative}, coercivity gives
\[
 \int_{\Oe}2\mu(\te)\abs{\cD(\boldsymbol{e}_{\eps})}^{2}\,\mathrm d\bx
 \ge 2\mumin\norm{\cD(\boldsymbol{e}_{\eps})}{L^{2}(\Oe)}^{2}.
\]
Then we apply Young's inequality with the same parameter $\zeta$ to the acceleration error, the quadratic convection term, the geometric residual, and the viscosity mismatch. 
Absorbing the four resulting dissipative remainders into the left-hand side leaves the coefficient $2\mumin-4\zeta$. The density and thermal-force remainders contain no gradient contribution and therefore enter only the quadratic error integral. Collecting these estimates, we prove \eqref{eq:closed-kinetic}.
\end{proof}

Gronwall's inequality now closes the density and kinetic estimates.

\begin{proof}[Proof of \Cref{thm:main}]
Recall the relative energy \eqref{eq:relative-energy}. Choose $\zeta=\mumin/8$ in \Cref{prop:kinetic-closed}. The uniform bounds from \Cref{lem:uniform} ensure that all terms in the following calculation are integrable. Multiplying \eqref{eq:density-ineq} by $1/2$ and adding the result to \eqref{eq:closed-kinetic}, one obtains
\begin{equation}\label{eq:gronwall-inequality}
 \cE_{\eps}(\tau)
 +c_{0}\int_{0}^{\tau}
 \norm{\cD(\boldsymbol{e}_{\eps})}{L^{2}(\Oe)}^{2}\,\mathrm dt
 \le
 \cE_{\eps}(0)
 +C\int_{0}^{\tau}\cE_{\eps}(t)\,\mathrm dt
 +C\eps^{2},
\end{equation}
for a.e. $\tau\in(0,T)$, with $c_{0}=3\mumin/2>0$ independent of $\eps$. To make the endpoint argument explicit, set
\[
 G_\eps(t)=\cE_\eps(0)+C\eps^2+C\int_0^t\cE_\eps(s)\,\mathrm ds.
\]
Then \eqref{eq:gronwall-inequality} implies $\cE_\eps(t)\le G_\eps(t)$, while $G_\eps'(t)=C\cE_\eps(t)\le C G_\eps(t)$ for a.e. $t\in(0,T)$. We apply Gronwall's inequality to the absolutely continuous function $G_\eps$. This gives a uniform bound for $G_\eps$ and hence for the essential supremum of $\cE_\eps$. Returning to \eqref{eq:gronwall-inequality} also controls the dissipation and yields
\begin{equation}\label{eq:relative-final}
 \esssup_{0\le t\le T}\cE_{\eps}(t)
 +\int_{0}^{T}
 \norm{\cD(\boldsymbol{e}_{\eps})}{L^{2}(\Oe)}^{2}\,\mathrm dt
 \le C\bigl(\cE_{\eps}(0)+\eps^{2}\bigr).
\end{equation}
Estimate \eqref{eq:relative-final} first gives the density component of \eqref{eq:main-estimate}. Since $\rhoe\ge\rhomin$, the same estimate also controls $\boldsymbol e_\eps$ in $L^{\infty}(0,T;L^{2})$. Korn's inequality \eqref{eq:korn} then yields the corrected $L^{2}(0,T;H^{1})$ bound.

Next, \eqref{eq:velocity-comparison} gives
\[
 \norm{\widetilde{\ue}-\uf}{L^{\infty}(0,T;L^{2}(\Omega))}
 \le C\bigl(\cE_{\eps}(0)^{1/2}+\eps\bigr).
\]
The pointwise estimate \eqref{eq:thermal-Linf}, together with \eqref{eq:velocity-comparison} and \eqref{eq:relative-final}, gives
\[
 \norm{\te-\tf}{L^\infty((0,T)\times\Omega)}
 \le C\bigl(\cE_\eps(0)^{1/2}+\eps\bigr).
\]
Finally, integrating \eqref{eq:thermal-H1} in time and using \eqref{eq:relative-final} yields
\[
 \int_{0}^{T}\norm{\te-\tf}{H^{1}(\Omega)}^{2}\,\mathrm dt
 \le C\bigl(\cE_{\eps}(0)+\eps^{2}\bigr).
\]
Thus estimates \eqref{eq:main-estimate}--\eqref{eq:main-convergences} follow. Since $\mu\in W^{1,\infty}(\R)$, we also have
\[
 \norm{\mu(\te)-\mu(\tf)}{L^\infty((0,T)\times\Omega)}
 \le C\bigl(\cE_\eps(0)^{1/2}+\eps\bigr).
\]

\end{proof}

\section{Dissipation limit}\label{sec:dissipation}

The same corrector identifies the limiting viscous dissipation, including the Brinkman contribution generated by the critical boundary layers.

\begin{corollary}[Viscous dissipation]\label{cor:dissipation}
Under the assumptions of \Cref{thm:main}, if $\cE_{\eps}(0)\to0$ as $\eps\to0$, then
\begin{equation}\label{eq:dissipation-limit}
\begin{aligned}
 \int_{0}^{T}\!\int_{\Oe}
 2\mu(\te)\abs{\cD(\ue)}^{2}\,\mathrm d\bx\,\mathrm dt
 \to{}&
 \int_{0}^{T}\!\int_{\Omega}
 2\mu(\tf)\abs{\cD(\uf)}^{2}\,\mathrm d\bx\,\mathrm dt+\int_{0}^{T}\!\int_{\Omega}
 \mu(\tf)\bM_{0}\uf\cdot\uf\,\mathrm d\bx\,\mathrm dt.
\end{aligned}
\end{equation}
Moreover, the uncorrected gradient error satisfies
\begin{equation}\label{eq:uncorrected-dissipation-defect}
 \int_0^T\!\int_\Omega
 2\mu(\tf)|\cD(\widetilde{\ue}-\uf)|^2\,\mathrm d\bx\,\mathrm dt
 \to\int_0^T\!\int_\Omega
 \mu(\tf)\bM_0\uf\cdot\uf\,\mathrm d\bx\,\mathrm dt.
\end{equation}
If $\uf\not\equiv\boldsymbol0$ in $(0,T)\times\Omega$, then the limit in \eqref{eq:uncorrected-dissipation-defect} is strictly positive.
\end{corollary}

\begin{proof}
By \eqref{eq:relative-final}, we have
\[
 \|\cD(\ue-\bU_{\eps})\|_{L^{2}((0,T)\times\Oe)}\to0.
\]
Using \Cref{lem:thermal}, \eqref{eq:velocity-comparison} and the $L^\infty(0,T;L^2)$ velocity bound from \eqref{eq:relative-final}, we also have convergence in $L^\infty((0,T)\times\Omega)$ in the essential-supremum sense:
\[
 \norm{\mu(\te)-\mu(\tf)}{L^{\infty}((0,T)\times\Omega)}\to0.
\]
Uniformly in $\eps$, the family $\bU_\eps$ is bounded in $L^\infty(0,T;H^1(\Oe))$. Therefore the cross term in $\cD(\ue)=\cD(\bU_\eps)+\cD(\ue-\bU_\eps)$ satisfies
\[
 \left|\int_0^T\!\int_{\Oe}
 4\mu(\tf)\cD(\bU_\eps):\cD(\ue-\bU_\eps)\,\mathrm d\bx\,\mathrm dt\right|
 \le C\norm{\cD(\ue-\bU_\eps)}{L^2((0,T)\times\Oe)}\to0,
\]
and the quadratic error is $o(1)$ as well. In addition,
\[
 \norm{\cD(\ue)}{L^2((0,T)\times\Oe)}
 \le \norm{\cD(\ue-\bU_\eps)}{L^2}
      +\norm{\cD(\bU_\eps)}{L^2}\le C,
\]
and thus
\[
 \left|\int_0^T\!\int_{\Oe}
 2(\mu(\te)-\mu(\tf))|\cD(\ue)|^2\,\mathrm d\bx\,\mathrm dt\right|
 \le C\norm{\mu(\te)-\mu(\tf)}{L^\infty((0,T)\times\Omega)}=o(1).
\]
Consequently, the microscopic dissipation differs by $o(1)$ from
\[
 \int_0^T\!\int_{\Oe}2\mu(\tf)|\cD(\bU_\eps)|^2\,\mathrm d\bx\,\mathrm dt.
\]
Let $\widetilde{\bU}_\eps$ denote the zero extension of $\bU_\eps$ to $\Omega$. Then
\[
 \widetilde{\bU}_\eps\in H_0^1(\Omega;\R^3),
 \qquad
 \nabla\widetilde{\bU}_\eps=\boldsymbol0\quad\text{a.e.\ in }\He,
\]
and hence
\[
 \int_\Omega2\mu(\tf)|\cD(\widetilde{\bU}_\eps)|^2\,\mathrm d\bx
 =\int_{\Oe}2\mu(\tf)|\cD(\bU_\eps)|^2\,\mathrm d\bx.
\]

To identify the limiting dissipation, set $a=\mu(\tf)$, $\bv=\uf$, and $\bz_\eps=\bU_\eps$. The bounds in \Cref{ass:strong} then allow us to apply \Cref{prop:restriction} at each fixed time. Equations \eqref{eq:weighted-residual-definition}--\eqref{eq:weighted-residual} then yield, uniformly in time,
\begin{align*}
 \int_{\Oe}2\mu(\tf)|\cD(\bU_\eps)|^2\,\mathrm d\bx
 ={}&\int_\Omega2\mu(\tf)\cD(\uf):\cD(\widetilde{\bU}_\eps)\,\mathrm d\bx+\int_\Omega\mu(\tf)\bM_0\uf\cdot\widetilde{\bU}_\eps\,\mathrm d\bx+O(\eps).
\end{align*}
Convergence of the Brinkman term follows from the $L^2$ estimate in \eqref{eq:restriction-estimates}. For the bulk term, since $\diver(2\mu(\tf)\cD(\uf))\in L^\infty(\Omega;\R^3)$ under \Cref{ass:strong}, integration by parts on the fixed domain gives
\[
 \left|\int_\Omega2\mu(\tf)\cD(\uf):\cD(\widetilde{\bU}_\eps-\uf)\,\mathrm d\bx\right|
 \le C\norm{\widetilde{\bU}_\eps-\uf}{L^2(\Omega)}\le C\eps.
\]
Integrating in time proves \eqref{eq:dissipation-limit}. 

Using \eqref{eq:main-convergences} and integrating by parts on the fixed domain,
\[
\begin{aligned}
 &\int_0^T\!\int_\Omega 2\mu(\tf)\cD(\uf):
       \cD(\widetilde{\ue}-\uf)\,\mathrm d\bx\,\mathrm dt=-\int_0^T\!\int_\Omega
       \diver(2\mu(\tf)\cD(\uf))\cdot
                 (\widetilde{\ue}-\uf)\,\mathrm d\bx\,\mathrm dt
       \to0.
\end{aligned}
\]
The coefficient replacement established in the proof of \eqref{eq:dissipation-limit} gives the explicit convergence
\[
 \int_0^T\!\int_\Omega2\mu(\tf)
       |\cD(\widetilde{\ue})|^2\,\mathrm d\bx\,\mathrm dt
 \to
 \int_0^T\!\int_\Omega2\mu(\tf)|\cD(\uf)|^2\,\mathrm d\bx\,\mathrm dt
 +\int_0^T\!\int_\Omega\mu(\tf)\bM_0\uf\cdot\uf\,\mathrm d\bx\,\mathrm dt.
\]
Expanding the square in the left-hand side of \eqref{eq:uncorrected-dissipation-defect} then gives the Brinkman term. Its positivity for nonzero $\uf$ follows from $\mu\ge\mumin$ and the positive definiteness of $\bM_0$.
\end{proof}

\appendix
\section{Local regular solvability with compatible initial jets}
\label{sec:effective-existence}

To verify that the comparison class in \Cref{ass:strong} is nonempty for a smooth class of data, this appendix establishes local regular solvability for the effective system. The smoothness assumptions imposed here on $\mu$, the data and the boundary are used only for this auxiliary construction. The main homogenization theorem retains the lower coefficient regularity stated in \eqref{eq:coefficients}. The argument proceeds from a high-order stationary Stokes estimate to a time-dependent linear hierarchy and then to a fixed-point construction, with the initial time derivatives constrained by a finite compatibility hierarchy. In \Cref{thm:effective-local}, the force is prescribed independently of the solution. For comparison, lower-regularity strong-solution theories were developed in \cite{ChoeKim,ChoKim}.

In this appendix, all Sobolev spaces are over $\Omega$, with values in $\R^3$ for vector fields. By an initial time jet of order $N$ we mean the finite collection $\{\partial_t^j\bv(0)\}_{j=0}^N$ (and, when needed, the analogous time derivatives of the pressure and coefficients). Write
\[
 H_\sigma^k=H^k(\Omega;\R^3)\cap H_{0,\sigma}^1(\Omega),
 \qquad k\ge1,
\]
and abbreviate $\operatorname{Tr}_{\Omega}$ by $\operatorname{Tr}$. For a scalar coefficient $b$, define
\[
 \mathcal L_b\bv=-\diver(2b\cD(\bv))+b\bM_0\bv.
\]
The operator $\mathcal L_b$ is linear in the coefficient $b$. We therefore use the same notation when $b$ is a time derivative of the viscosity coefficient, for which no sign condition is imposed.

For any $\bv\in H^1(\Omega;\R^3)$ with $\diver\bv=0$ and $\bv\cdot\bn=0$, the Neumann operator is defined by
\begin{equation}\label{eq:effective-neumann-map}
 \left\{
 \begin{aligned}
 -\kf\Delta\Theta[\bv]&=\bv\cdot\bg,
 &&\text{in }\Omega,\\
 \partial_{\bn}\Theta[\bv]&=0,
 &&\text{on }\partial\Omega,\\
 \int_\Omega\Theta[\bv]\,\mathrm d\bx&=0.
 \end{aligned}
\right.
\end{equation}
Compatibility for \eqref{eq:effective-neumann-map} follows from $\bg=\nabla F$, since integration by parts gives $\int_\Omega\bv\cdot\bg\,\mathrm d\bx=0$. Hence Lax--Milgram on the zero-mean subspace of $H^1(\Omega)$ defines $\Theta$ uniquely, with the mean condition fixing the same thermal reference level as in \eqref{eq:effective-boundary}. For smooth $F$ and $\partial\Omega$, elliptic regularity then yields
\begin{equation}\label{eq:appendix-theta-bound}
 \|\Theta[\bv]\|_{H^{k+2}}\le C_F\|\bv\|_{H^k},
 \qquad k\ge0,
\end{equation}
whenever the right-hand side is finite. 
Fix $\overline T>0$ and an integer $N\ge3$. Put $s=2N+1$ and $s_j=s-2j$ for $0\le j\le N$. For $0<\tau\le\overline T$, define the velocity space $\mathbb V_{N,\tau}$ by
\begin{equation}\label{eq:appendix-velocity-space}
\begin{gathered}
 \partial_t^j\bv\in C([0,\tau];H_\sigma^{s_j}) \cap L^2(0,\tau;H^{s_j+1}) \ (0\le j\le N), \quad \partial_t^{N+1}\bv\in L^2(0,\tau;L^2),\\[0.25em]
 \begin{aligned}
 \|\bv\|_{\mathbb V_{N,\tau}}
 :={}&\sum_{j=0}^N\Bigl(
 \|\partial_t^j\bv\|_{C H^{s_j}}
 +\|\partial_t^j\bv\|_{L^2H^{s_j+1}}\Bigr)+\|\partial_t^{N+1}\bv\|_{L^2L^2}.
 \end{aligned}
\end{gathered}
\end{equation}
Here $C H^k=C([0,\tau];H^k)$ and all time norms in this appendix are on $(0,\tau)$ unless another interval is displayed. For a force $\boldsymbol h$, write $\boldsymbol h_j=\partial_t^j\boldsymbol h$. Its data space $\mathbb F_{N,\tau}$ is specified by
\begin{equation}\label{eq:appendix-force-space}
\begin{gathered}
 \boldsymbol h_j\in C H^{s_j-2}\cap L^2H^{s_j-1}\  (0\le j\le N-1),\quad \boldsymbol h_N\in L^2L^2,\\[0.25em]
 \begin{aligned}
 \|\boldsymbol h\|_{\mathbb F_{N,\tau}}
 :={}&\sum_{j=0}^{N-1}\Bigl(
 \|\boldsymbol h_j\|_{C H^{s_j-2}}
 +\|\boldsymbol h_j\|_{L^2H^{s_j-1}}\Bigr)+\|\boldsymbol h_N\|_{L^2L^2}.
 \end{aligned}
\end{gathered}
\end{equation}
In particular, $s_N=1$ and $s_{N-1}=3$. Consequently, all time derivatives can first be estimated at the $H^1$ energy level before the spatial regularity is recovered.

For coefficients $r,a$, set $r_j=\partial_t^jr$ and $a_j=\partial_t^ja$. Assume
\begin{equation}\label{eq:linear-coefficient-class}
\begin{gathered}
 r\in C H^{s+1},
 \quad
 r_j\in C H^{s_j+2}\quad(1\le j\le N),\\
 a_j\in C H^{s_j+2}\quad(0\le j\le N),\quad  0<r_*\le r\le r^*, \quad 0<a_*\le a\le a^*.
\end{gathered}
\end{equation}
Time derivatives in these definitions are distributional. Define
\[
 \mathfrak K_N:=\|r\|_{C H^{s+1}}
 +\sum_{j=1}^N\|r_j\|_{C H^{s_j+2}}
 +\sum_{j=0}^N\|a_j\|_{C H^{s_j+2}}.
\]
The lower bounds and $\mathfrak K_N$ control all coefficient norms needed in the linear estimates, including $\|r_1\|_{L^\infty}$ and $\|a_j\|_{L^\infty W^{1,\infty}}$. The initial jets are determined recursively. Set $\boldsymbol U_0=\boldsymbol u_0$. For $0\le j<N$, given $\boldsymbol U_0,\ldots,\boldsymbol U_j$, define
\[
 \boldsymbol K_j:=\boldsymbol h_j(0)
 -\sum_{\ell=0}^j\binom j\ell
       \mathcal L_{a_\ell(0)}\boldsymbol U_{j-\ell}-\sum_{\ell=1}^j\binom j\ell
       r_\ell(0)\boldsymbol U_{j+1-\ell},
\]
with
\begin{equation}\label{eq:linear-jet-recursion}
\begin{gathered}
 \left\{
 \begin{aligned}
 r(0)\boldsymbol U_{j+1}+\nabla P_j&=\boldsymbol K_j
     &&\text{in }\Omega,\\
 \diver\boldsymbol U_{j+1}&=0
     &&\text{in }\Omega,\\
 \boldsymbol U_{j+1}\cdot\bn&=0
     &&\text{on }\partial\Omega,\\
 \int_\Omega P_j\,\mathrm d\bx&=0,
 \end{aligned}
 \right.
 \qquad 0\le j<N.
\end{gathered}
\end{equation}
Equivalently, $P_j$ is the zero-mean solution of the weighted Neumann problem
\[
 \left\{
 \begin{aligned}
 \diver\bigl(r(0)^{-1}\nabla P_j\bigr)
 &=\diver\bigl(r(0)^{-1}\boldsymbol K_j\bigr)
       &&\text{in }\Omega,\\
 r(0)^{-1}\nabla P_j\cdot\bn
 &=r(0)^{-1}\boldsymbol K_j\cdot\bn
       &&\text{on }\partial\Omega.
 \end{aligned}
 \right.
\]
Its compatibility follows from the divergence theorem. This recursion fixes the normal component of $\boldsymbol U_{j+1}$. The full no-slip conditions
\begin{equation}\label{eq:linear-jet-compatibility}
 \boldsymbol U_j\in H_\sigma^{s_j},\qquad 0\le j\le N,
\end{equation}
are additional requirements on the data. At $j=0$, \eqref{eq:linear-jet-recursion} is the usual first compatibility relation.

The spatial bootstrap uses the following elliptic estimate, stated for a spatially variable viscosity and arbitrary integer order.

\begin{lemma}[High-order stationary Stokes estimate with scalar variable viscosity]
\label{lem:stationary-variable-stokes}
Let $m\ge0$ be an integer and let $\partial\Omega$ be of class $C^{m+3}$. Assume
\[
 a\in H^{m+3}(\Omega),\qquad
 0<a_*\le a(\bx)\le a^*<\infty\quad\text{a.e.\ in }\Omega.
\]
If $\boldsymbol f\in H^m(\Omega;\R^3)$ and
\[
 \mathcal L_a\bv+\nabla\pi=\boldsymbol f,\qquad
 \diver\bv=0\quad\text{in }\Omega,\qquad
 \bv=\boldsymbol0\quad\text{on }\partial\Omega,\qquad
 \int_\Omega\pi\,\mathrm d\bx=0,
\]
in the weak sense, with $\bv\in H_{0,\sigma}^1(\Omega)$ and $\pi\in L^2(\Omega)$, then
\begin{equation}\label{eq:stationary-variable-stokes}
 \|\bv\|_{H^{m+2}}+\|\pi\|_{H^{m+1}}
 \le C\bigl(\|\boldsymbol f\|_{H^m}+\|\bv\|_{H^1}\bigr).
\end{equation}
The constant depends only on $m$, $\Omega$, $\bM_0$, $a_*$, $a^*$ and an upper bound for $\|a\|_{H^{m+3}}$.
\end{lemma}

\begin{proof}
We reduce the variable-viscosity system to the classical constant-viscosity Dirichlet Stokes estimate and then argue by induction on $m$. Since $\diver\bv=0$, we have
\[
 -\diver(2a\cD(\bv))
 =-a\Delta\bv-2\cD(\bv)\nabla a.
\]
Set $\varpi=a^{-1}\pi$. Dividing the momentum equation by $a$ and using
$a^{-1}\nabla\pi=\nabla \varpi+\varpi\,a^{-1}\nabla a$ gives
\begin{equation}\label{eq:stationary-reduced-stokes}
 -\Delta\bv+\nabla \varpi
 =\boldsymbol F
 :=a^{-1}\boldsymbol f
   +2a^{-1}\cD(\bv)\nabla a
   -\bM_0\bv
   -\varpi\,a^{-1}\nabla a,
 \qquad \diver\bv=0,\qquad \bv|_{\partial\Omega}=0.
\end{equation}
Although $\varpi$ need not have zero mean,
$\|\varpi\|_{L^2}\le a_*^{-1}\|\pi\|_{L^2}$. Hence its mean is controlled, so an estimate for $\varpi-\overline\varpi$ also yields an estimate for $\varpi$.

We first establish the case $m=0$. The map
$\bphi\longmapsto\int_\Omega 2a\cD(\bv):\cD(\bphi)
 +a\bM_0\bv\cdot\bphi\,\mathrm d\bx$
is bounded on $H_0^1(\Omega;\R^3)$. Hence the inf--sup estimate applied to original weak equation gives
\[
 \|\pi\|_{L^2}
 \le C\bigl(\|\boldsymbol f\|_{H^{-1}}+\|\bv\|_{H^1}\bigr),
\]
Since $a\in H^3(\Omega)\hookrightarrow W^{1,\infty}(\Omega)$, the right-hand side of \eqref{eq:stationary-reduced-stokes} belongs to $L^2$ and
\[
 \|\boldsymbol F\|_{L^2}
 \le C\bigl(\|\boldsymbol f\|_{L^2}
             +\|\bv\|_{H^1}+\|\pi\|_{L^2}\bigr).
\]
The classical constant-coefficient Dirichlet Stokes estimate
\cite[Theorem~IV.6.1]{Galdi}, applied to \eqref{eq:stationary-reduced-stokes}, therefore yields
$\bv\in H^2$ and $\varpi-\overline \varpi\in H^1$, where
$\overline \varpi=|\Omega|^{-1}\int_\Omega \varpi\,\mathrm d\bx$.
The $L^2$ bound for $\varpi$ controls $\overline \varpi$, hence
$\|\varpi\|_{H^1}\le C(\|\boldsymbol f\|_{L^2}+\|\bv\|_{H^1})$.
Multiplication by $a$ then gives the same bound for
$\pi=a\varpi$ in $H^1$. This proves \eqref{eq:stationary-variable-stokes} for $m=0$.

Assume now that the result is known at order $m-1$ and let
$\boldsymbol f\in H^m$. Then
\[
 \bv\in H^{m+1},\qquad \pi\in H^m,\qquad \varpi=a^{-1}\pi\in H^m.
\]
Since $a\in H^{m+3}$ and $a$ is bounded away from zero, the Sobolev composition theorem gives $a^{-1}\in H^{m+3}$ with a norm controlled by the stated quantities. The product structure in \eqref{eq:stationary-reduced-stokes} is closed at order $m$. Indeed, in three dimensions multiplication by an $H^{m+2}$ function is bounded on $H^m$ for $m\ge1$. Since
\[
 a^{-1}\in H^{m+3},\qquad
 a^{-1}\nabla a\in H^{m+2},\qquad
 \cD(\bv)\in H^m,\qquad \varpi\in H^m,
\]
we have, with constants depending only on the stated coefficient bounds,
\[
\begin{aligned}
 \|a^{-1}\boldsymbol f\|_{H^m}
 &\le C\|\boldsymbol f\|_{H^m},\\
 \|a^{-1}\cD(\bv)\nabla a\|_{H^m}
 &\le C\|\bv\|_{H^{m+1}},\\
 \|\varpi\,a^{-1}\nabla a\|_{H^m}
 &\le C\|\varpi\|_{H^m}.
\end{aligned}
\]
Consequently,
\[
 \|\boldsymbol F\|_{H^m}
 \le C\bigl(
   \|\boldsymbol f\|_{H^m}
   +\|\bv\|_{H^{m+1}}
   +\|\varpi\|_{H^m}\bigr).
\]
The constant-coefficient Stokes estimate \cite[Theorem~IV.6.1]{Galdi} at order $m$ yields
\[
 \|\bv\|_{H^{m+2}}
 +\|\varpi-\overline \varpi\|_{H^{m+1}}
 \le C\|\boldsymbol F\|_{H^m}.
\]
As in the base case, $\overline \varpi$ is controlled by $\|\varpi\|_{L^2}$, already bounded at the preceding level. Using the induction hypothesis in the right-hand side and $\pi=a\varpi$ proves \eqref{eq:stationary-variable-stokes}. The induction is finite, so the dependence of the constant is exactly as stated.
\end{proof}

\begin{proposition}[High-order linear Stokes estimate with compatible jets]
\label{prop:linear-effective-stokes}
Assume that $\partial\Omega$ is of class $C^{s+3}$, $\boldsymbol h\in\mathbb F_{N,\tau}$ and \eqref{eq:linear-coefficient-class} holds. Suppose that the recursion \eqref{eq:linear-jet-recursion} satisfies \eqref{eq:linear-jet-compatibility}. Then the problem
\[
 \left\{
 \begin{aligned}
 r\partial_t\boldsymbol u+\mathcal L_a\boldsymbol u+\nabla p
 &=\boldsymbol h,
 &&\text{in }(0,\tau)\times\Omega,\\
 \diver\boldsymbol u&=0,
 &&\text{in }(0,\tau)\times\Omega,\\
 \boldsymbol u&=\boldsymbol0,
 &&\text{on }(0,\tau)\times\partial\Omega,\\
 \boldsymbol u(0)&=\boldsymbol U_0,
 \qquad \int_\Omega p(t)\,\mathrm d\bx=0,
 \end{aligned}
\right.
\]
has a unique solution with $\boldsymbol u\in\mathbb V_{N,\tau}$ and $\partial_t^j\boldsymbol u(0)=\boldsymbol U_j$ for $0\le j\le N$. Writing $p_j=\partial_t^jp$, one has
\begin{equation}\label{eq:linear-pressure-hierarchy}
\begin{gathered}
 p_j\in C H^{s_j-1}\cap L^2H^{s_j}
 \quad(0\le j\le N-1),
 \qquad
 p_N\in L^2H^1,\\[0.35em]
 \begin{aligned}
 &\|\boldsymbol u\|_{\mathbb V_{N,\tau}}
 +\sum_{j=0}^{N-1}\Bigl(
 \|p_j\|_{C H^{s_j-1}}+\|p_j\|_{L^2H^{s_j}}\Bigr)
 +\|p_N\|_{L^2H^1}\le C\left(
 \sum_{j=0}^N\|\boldsymbol U_j\|_{H^{s_j}}
 +\|\boldsymbol h\|_{\mathbb F_{N,\tau}}\right).
 \end{aligned}
\end{gathered}
\end{equation}
Here $C$ depends on $N$, $\Omega$, $\overline T$, $\bM_0$, the ellipticity bounds and an upper bound for $\mathfrak K_N$. It is uniform for $0<\tau\le\overline T$.

For zero initial velocity and an arbitrary force $\boldsymbol H\in L^2(0,\tau;H^{-1})$, the energy solution satisfies
\begin{equation}\label{eq:linear-energy-difference}
 \|\boldsymbol z\|_{\mathbb E_\tau}
 :=\|\boldsymbol z\|_{L^\infty L^2}
   +\|\boldsymbol z\|_{L^2H^1}
 \le C\|\boldsymbol H\|_{L^2H^{-1}}.
\end{equation}
If $\boldsymbol H\in\mathbb F_{N,\tau}$ and its initial jets satisfy \eqref{eq:linear-jet-recursion} with $\boldsymbol U_0=\cdots=\boldsymbol U_N=0$, then
\begin{equation}\label{eq:linear-zero-high}
 \|\boldsymbol z\|_{\mathbb V_{N,\tau}}
 \le C\|\boldsymbol H\|_{\mathbb F_{N,\tau}}.
\end{equation}
\end{proposition}

\begin{proof}
We begin with
\[
 r\partial_t\bw+\mathcal L_a\bw+c\bw+\nabla\pi=\boldsymbol G,
 \qquad \bw(0)=\bw_0\in H_{0,\sigma}^1,
\]
with homogeneous Dirichlet data and a bounded scalar coefficient $c$. The following energy calculation is first performed for the finite-dimensional solenoidal Galerkin system. There $\partial_t\bw^M$ belongs to the trial space and is therefore an admissible test function. The resulting bounds are uniform in $M$ and pass to the weak limit. Testing the Galerkin equation by $\partial_t\bw^M$ and omitting the approximation index in the displayed a priori estimate gives
\[
 \begin{aligned}
 \int_\Omega r|\partial_t\bw|^2\,\mathrm d\bx
 &+\frac{\mathrm d}{\mathrm dt}\left[
       \int_\Omega a|\cD(\bw)|^2\,\mathrm d\bx
       +\frac12\int_\Omega a\bM_0\bw\cdot\bw\,\mathrm d\bx\right]\\
 &=\int_\Omega(\boldsymbol G-c\bw)\cdot\partial_t\bw\,\mathrm d\bx
   +\int_\Omega a_1|\cD(\bw)|^2\,\mathrm d\bx
   +\frac12\int_\Omega a_1\bM_0\bw\cdot\bw\,\mathrm d\bx.
 \end{aligned}
\]
The pressure term vanishes since $\partial_t\bw$ is solenoidal and has zero trace. We use the positivity of $r$ and $a$, the positive definiteness of $\bM_0$, and Korn's and Young's inequalities to obtain
\[
 \frac{\mathrm d}{\mathrm dt}\mathscr E_{\bw}(t)
 +c_1\|\partial_t\bw(t)\|_{L^2}^2
 \le C\mathscr E_{\bw}(t)+C\|\boldsymbol G(t)\|_{L^2}^2,
\]
where $\mathscr E_{\bw}$ is uniformly equivalent to $\|\bw\|_{H^1}^2$. Gronwall's lemma gives the $L^\infty H^1$ bound for $\bw$ and the $L^2L^2$ bound for $\partial_t\bw$. For a.e. $t\in(0,\tau)$, \Cref{lem:stationary-variable-stokes} with $m=0$ applies to the stationary equation with right-hand side $\boldsymbol G-r\partial_t\bw-c\bw$. Since the coefficient bounds in \eqref{eq:linear-coefficient-class} are stronger than those required by the lemma,
\begin{equation}\label{eq:linear-base-energy}
 \|\bw\|_{C H^1}+\|\bw\|_{L^2H^2}
 +\|\partial_t\bw\|_{L^2L^2}+\|\pi\|_{L^2H^1}
 \le C\bigl(\|\bw_0\|_{H^1}+\|\boldsymbol G\|_{L^2L^2}\bigr).
\end{equation}
Passing to the Galerkin limit gives existence at this level. We then recover the pressure from the solenoidal weak formulation by the de Rham characterization (see, e.g.~\cite[Theorem~III.5.3]{Galdi}). Moreover, $L^2H^2\cap H^1L^2$ and the corresponding trace theorem give continuity into $H^1$. For the constant-coefficient nonstationary Stokes estimate at this level, see \cite[Theorem~1.1, p.~333]{SolonnikovStokes}. The variable-viscosity spatial regularity used here is supplied by \Cref{lem:stationary-variable-stokes}. Testing by $\bw$ itself also gives
\[
 \|\bw\|_{\mathbb E_\tau}
 \le C\bigl(\|\bw_0\|_{L^2}
                 +\|\boldsymbol G\|_{L^2H^{-1}}\bigr).
\]
This estimate allows a bounded zeroth-order coefficient $c$ of either sign.

To construct the time-derivative hierarchy, introduce auxiliary fields $\bw_j$ corresponding to $\partial_t^j\boldsymbol u$. The differentiated equation has the form
\begin{equation}\label{eq:time-differentiated-stokes}
\begin{aligned}
 &r\partial_t\bw_j+\mathcal L_a\bw_j+jr_1\bw_j+\nabla p_j=\boldsymbol G_j,\\
 &\boldsymbol G_j:={\boldsymbol h_j}-\sum_{\ell=2}^j\binom j\ell r_\ell\bw_{j+1-\ell}-\sum_{\ell=1}^j\binom j\ell
       \mathcal L_{a_\ell}\bw_{j-\ell},
 \qquad 0\le j\le N,
\end{aligned}
\end{equation}
where empty sums are zero. All terms in $\boldsymbol G_j$, except $\boldsymbol h_j$, involve strictly lower time indices. Since $r_\ell\in L^\infty$ and $a_\ell\in L^\infty W^{1,\infty}$,
\[
 \|\boldsymbol G_j\|_{L^2L^2}
 \le \|\boldsymbol h_j\|_{L^2L^2}
     +C\sum_{k<j}\bigl(\|\bw_k\|_{L^2H^2}
                         +\|\bw_k\|_{L^2L^2}\bigr).
\]
Equation \eqref{eq:time-differentiated-stokes} is solved successively for $j=0,\ldots,N$ with initial values $\boldsymbol U_j$. Since the coefficient $jr_1$ is bounded uniformly for the finite range of $j$, \eqref{eq:linear-base-energy} applies at each level and yields
\begin{equation}\label{eq:all-time-base-bounds}
\begin{aligned}
 \sum_{j=0}^N\Bigl(
 \|\bw_j\|_{C H^1}+\|\bw_j\|_{L^2H^2}
 +\|\partial_t\bw_j\|_{L^2L^2}+\|p_j\|_{L^2H^1}\Bigr)
 \le C\left(
 \sum_{j=0}^N\|\boldsymbol U_j\|_{H^1}
 +\sum_{j=0}^N\|\boldsymbol h_j\|_{L^2L^2}\right).
\end{aligned}
\end{equation}

To identify the auxiliary fields with the time derivatives, work in the solenoidal weak formulation. Assume inductively that $\bw_k=\partial_t^k\boldsymbol u$ for $k<j$, and set
\[
 \boldsymbol V_{j-1}(t)=\boldsymbol U_{j-1}
             +\int_0^t\bw_j(\sigma)\,\mathrm d\sigma,
 \qquad
 \boldsymbol Z_j=\boldsymbol V_{j-1}-\bw_{j-1}.
\]
For $\bphi\in H_{0,\sigma}^1(\Omega)$, define the level-$(j-1)$ residual of $\boldsymbol V_{j-1}$ by
\[
 \begin{aligned}
 \langle\mathfrak R_j(t),\bphi\rangle
 :={}&\int_\Omega r\,\partial_t\boldsymbol V_{j-1}\cdot\bphi
       +2a\cD(\boldsymbol V_{j-1}):\cD(\bphi)
       +a\bM_0\boldsymbol V_{j-1}\cdot\bphi\\
 &\quad +(j-1)r_1\boldsymbol V_{j-1}\cdot\bphi
       -\boldsymbol G_{j-1}\cdot\bphi\,\mathrm d\bx.
 \end{aligned}
\]
The compatibility recursion \eqref{eq:linear-jet-recursion} implies
$\mathfrak R_j(0)=0$ in the dual of $H_{0,\sigma}^1(\Omega)$: at $t=0$ the full residual is the gradient $-\nabla P_{j-1}$ and therefore vanishes on solenoidal test fields. The cancellation
\[
 \partial_t\boldsymbol G_{j-1}
 =\boldsymbol G_j+\mathcal L_{a_1}\bw_{j-1}
   +(j-1)r_2\bw_{j-1}
\]
and the level-$j$ equation in \eqref{eq:time-differentiated-stokes} give, in the same dual space,
\[
 \partial_t\mathfrak R_j
 =\bigl(\mathcal L_{a_1}+(j-1)r_2\bigr)\boldsymbol Z_j.
\]
Hence
\[
 \mathfrak R_j(t)
 =\int_0^t\bigl(\mathcal L_{a_1(\sigma)}
    +(j-1)r_2(\sigma)\bigr)\boldsymbol Z_j(\sigma)
                    \,\mathrm d\sigma.
\]
Subtracting the solenoidal weak equation for $\bw_{j-1}$ shows that $\boldsymbol Z_j(0)=0$ and
\[
 r\partial_t\boldsymbol Z_j+\mathcal L_a\boldsymbol Z_j
 +(j-1)r_1\boldsymbol Z_j
 =\int_0^t\bigl(\mathcal L_{a_1(\sigma)}
    +(j-1)r_2(\sigma)\bigr)\boldsymbol Z_j(\sigma)
                    \,\mathrm d\sigma
\]
in the dual of $H_{0,\sigma}^1(\Omega)$. The operator
$\mathcal L_{a_1}+(j-1)r_2$ maps $H_0^1(\Omega)$ continuously into $H^{-1}(\Omega)$ under \eqref{eq:linear-coefficient-class}. Therefore, on an interval of length $\tau_0$, the pressure-free energy estimate corresponding to \eqref{eq:linear-energy-difference} gives
\[
 \|\boldsymbol Z_j\|_{\mathbb E_{\tau_0}}
 \le C\left\|\int_0^t
       \|\boldsymbol Z_j(\sigma)\|_{H^1}\,\mathrm d\sigma
       \right\|_{L^2(0,\tau_0)}
 \le C\tau_0\|\boldsymbol Z_j\|_{L^2(0,\tau_0;H^1)}.
\]
Choosing $\tau_0$ so that $C\tau_0<1$ yields $\boldsymbol Z_j=0$ on the first interval, and repetition on consecutive intervals gives $\boldsymbol Z_j=0$ on $[0,\tau]$. Thus $\bw_j=\partial_t\bw_{j-1}$ distributionally. Differentiating the level-$(j-1)$ equation in distributions and comparing it with the level-$j$ equation then gives
$\nabla(p_j-\partial_t p_{j-1})=0$. Since both pressure representatives have zero spatial mean, $p_j=\partial_t p_{j-1}$ in distributions. Induction identifies all derivatives up to order $N$ and recovers the specified values $\partial_t^j\boldsymbol u(0)=\boldsymbol U_j$.

Spatial regularity is recovered by an elliptic bootstrap. At any fixed time, write
\begin{equation}\label{eq:elliptic-jet-system}
 \mathcal L_a\bw_j+\nabla p_j
 =\boldsymbol h_j-r\bw_{j+1}
 -\sum_{\ell=1}^j\binom j\ell
       \bigl(r_\ell\bw_{j+1-\ell}
              +\mathcal L_{a_\ell}\bw_{j-\ell}\bigr),
 \quad j<N.
\end{equation}
For $0\le k\le s-1$, \Cref{lem:stationary-variable-stokes} applies to \eqref{eq:elliptic-jet-system} since $a(t)\in H^{s+2}\subset H^{k+3}$ and $\partial\Omega$ is of class $C^{s+3}$. Hence
\[
 \|\bv\|_{H^{k+2}}+\|\pi\|_{H^{k+1}}
 \le C\bigl(\|\mathcal L_a\bv+\nabla\pi\|_{H^k}
                  +\|\bv\|_{H^1}\bigr),
 \qquad \bv\in H_{0,\sigma}^1,\quad
 \int_\Omega\pi\,\mathrm d\bx=0,
\]
with a constant uniform in time. For the differentiated coefficient terms on the right-hand side, the Sobolev product estimates give
$\|\mathcal L_b\bv\|_{H^k}\le C\|b\|_{H^{k+2}}\|\bv\|_{H^{k+2}}$ at every order needed in the bootstrap. The coefficient class \eqref{eq:linear-coefficient-class} supplies the required coefficient norms.

Starting from \eqref{eq:all-time-base-bounds}, consider the triangular induction
\begin{equation}\label{eq:triangular-spatial-induction}
 \ell=1,\ldots,N,
 \qquad j=0,1,\ldots,N-\ell,
 \qquad
 \bw_j\in C H^{2\ell+1}\cap L^2H^{2\ell+2}.
\end{equation}
Fix $\ell$ and assume that the induction has been established for all spatial levels $1,\ldots,\ell-1$ and, at level $\ell$, for all time indices $k<j$. From \eqref{eq:all-time-base-bounds} when $\ell=1$, and from level $\ell-1$ when $\ell\ge2$,
\[
 \bw_{j+1},\;r_1\bw_j
 \in C H^{2\ell-1}\cap L^2H^{2\ell}.
\]
Every other coefficient term in \eqref{eq:elliptic-jet-system} contains a velocity with a strictly smaller time index, so the induction in $j$, together with the Sobolev product estimates and \eqref{eq:linear-coefficient-class}, gives the same regularity for those terms.

For the force, $j\le N-\ell$ implies
\[
 s_j=2N+1-2j\ge2\ell+1,
 \qquad
 s_j-2\ge2\ell-1,
 \qquad
 s_j-1\ge2\ell.
\]
Hence the complete right-hand side of \eqref{eq:elliptic-jet-system} belongs to
\[
 C H^{2\ell-1}\cap L^2H^{2\ell}.
\]
Applying \Cref{lem:stationary-variable-stokes} with $m=2\ell-1$ and with $m=2\ell$ at each fixed time gives $\bw_j\in L^\infty(0,\tau;H^{2\ell+1})\cap L^2(0,\tau;H^{2\ell+2})$ and the corresponding pressure bounds $p_j\in L^2(0,\tau;H^{2\ell+1})$.
The upgrade to continuity in time is obtained by interpolation. By the induction hypothesis at level $\ell-1$ we have $\bw_j\in L^2(0,\tau;H^{2\ell+2})$ and $\partial_t\bw_j=\bw_{j+1}\in L^2(0,\tau;H^{2\ell})$, and $[H^{2\ell+2}(\Omega),H^{2\ell}(\Omega)]_{1/2}=H^{2\ell+1}(\Omega)$, so the Lions--Magenes time-continuity theorem gives $\bw_j\in C([0,\tau];H^{2\ell+1})$. The same argument applied to $p_j\in L^2(0,\tau;H^{2\ell+1})$ with $\partial_tp_j=p_{j+1}\in L^2(0,\tau;H^{2\ell-1})$ gives $p_j\in C([0,\tau];H^{2\ell})$.
 Taking $\ell=N-j$ gives $2\ell+1=s_j$, so \eqref{eq:triangular-spatial-induction} and \eqref{eq:all-time-base-bounds} imply the full hierarchy \eqref{eq:linear-pressure-hierarchy}.

For the zero-data estimate, test the undifferentiated equation by $\boldsymbol z$ and differentiate the weighted kinetic energy:
\[
 \frac12\frac{\mathrm d}{\mathrm dt}
       \int_\Omega r|\boldsymbol z|^2\,\mathrm d\bx
 +2\int_\Omega a|\cD(\boldsymbol z)|^2\,\mathrm d\bx
 +\int_\Omega a\bM_0\boldsymbol z\cdot\boldsymbol z\,\mathrm d\bx
 =\langle\boldsymbol H,\boldsymbol z\rangle
   +\frac12\int_\Omega r_1|\boldsymbol z|^2\,\mathrm d\bx.
\]
Using the bounds on $r_1$ together with Korn's and Young's inequalities, we apply Gronwall's lemma to obtain \eqref{eq:linear-energy-difference} and uniqueness. For compatible zero jets, \eqref{eq:linear-zero-high} is the zero-data specialization of \eqref{eq:linear-pressure-hierarchy}. This completes the proof of the proposition.
\end{proof}

The nonlinear compatibility conditions are encoded through the initial time jets. Given smooth initial fields $\rho_0,\boldsymbol u_0$ and a smooth force $\bfv$, write
\[
 R_0=\rho_0,\qquad \boldsymbol U_0=\boldsymbol u_0,
 \qquad \boldsymbol F_j=\partial_t^j\bfv(0).
\]
For $j=0,\ldots,N-1$, use the fields already determined to set
\begin{equation}\label{eq:nonlinear-density-jets}
\begin{aligned}
 R_{j+1}=-\sum_{\ell=0}^j\binom j\ell \boldsymbol U_\ell\cdot\nabla R_{j-\ell},\quad  \Theta_j=\Theta[\boldsymbol U_j],\quad A_j
 =\left.\frac{\mathrm d^j}{\mathrm dt^j}\mu\!\left(\sum_{\ell=0}^j\frac{t^\ell}{\ell!}\Theta_\ell\right) \right|_{t=0}, \quad 0\le j<N.
\end{aligned}
\end{equation}
Define
\[
 \begin{aligned}
  \boldsymbol K_j:=&\sum_{\ell=0}^j\binom j\ell R_\ell\boldsymbol F_{j-\ell}-\Theta_j\bg-\sum_{\ell+k+n=j}\frac{j!}{\ell!k!n!} R_\ell(\boldsymbol U_k\cdot\nabla)\boldsymbol U_n\\
 &-\sum_{\ell=0}^j\binom j\ell \mathcal L_{A_\ell}\boldsymbol U_{j-\ell}-\sum_{\ell=1}^j\binom j\ell R_\ell\boldsymbol U_{j+1-\ell},
 \end{aligned}
 \]
with
\begin{equation}\label{eq:nonlinear-momentum-jets}
\begin{gathered}
 \left\{
 \begin{aligned}
 \rho_0\boldsymbol U_{j+1}+\nabla P_j&=\boldsymbol K_j
     &&\text{in }\Omega,\\
 \diver\boldsymbol U_{j+1}&=0
     &&\text{in }\Omega,\\
 \boldsymbol U_{j+1}\cdot\bn&=0
     &&\text{on }\partial\Omega,\\
 \int_\Omega P_j\,\mathrm d\bx&=0,
 \end{aligned}
 \right.
 \qquad 0\le j<N.
\end{gathered}
\end{equation}
The weighted Neumann construction following \eqref{eq:linear-jet-recursion} determines $P_j$ and $\boldsymbol U_{j+1}$. After obtaining $\boldsymbol U_N$, also set $\Theta_N=\Theta[\boldsymbol U_N]$ and define $A_N$ by the same composition formula. Smooth data give smooth jets. Their full boundary traces are then imposed through the compatibility condition \eqref{eq:nonlinear-jet-compatibility}. These recursively generated jets define the finite compatibility class used in the nonlinear fixed-point construction.

\begin{theorem}[Local regular solvability for compatible data]
\label{thm:effective-local}
Let $m\ge5$ and $\overline T>0$. Choose $N=\lceil(m+1)/2\rceil$. With the initial jets defined by \eqref{eq:nonlinear-density-jets}--\eqref{eq:nonlinear-momentum-jets},
\[
 s=2N+1\ge m+2,\qquad s_1=s-2\ge m,\qquad s_N=1.
\]
Assume that $\partial\Omega$, $F$, $\rho_0$, $\boldsymbol u_0$ and $\bfv$ are smooth up to their respective boundaries, with
\[
 \bfv\in C^\infty([0,\overline T]\times\overline\Omega;\R^3),
 \qquad \rhomin\le\rho_0\le\rhomax,
 \qquad \diver\boldsymbol u_0=0,
 \qquad \operatorname{Tr}\boldsymbol u_0=\boldsymbol0.
\]
Let $\mu\in C^\infty(\R)$ satisfy \eqref{eq:coefficients}. Form the initial jets by \eqref{eq:nonlinear-density-jets}--\eqref{eq:nonlinear-momentum-jets} and assume the finite hierarchy
\begin{equation}\label{eq:nonlinear-jet-compatibility}
 \operatorname{Tr}\boldsymbol U_j=\boldsymbol0,
 \qquad 1\le j\le N.
\end{equation}
Then there exists $T_*\in(0,\overline T]$ and a unique effective solution of \eqref{eq:effective-intro} with the conditions \eqref{eq:effective-boundary} and initial data \eqref{eq:effective-initial} such that, with $s=2N+1$,
\begin{equation}\label{eq:effective-high-regularity}
\begin{aligned}
 \rhof&\in C([0,T_*];H^{s+1})\cap C^1([0,T_*];H^s),
 &\qquad \uf&\in\mathbb V_{N,T_*},\\
 \tf&\in C([0,T_*];H^{s+2}),
 &\pf&\in C([0,T_*];H^{s-1}).
\end{aligned}
\end{equation}
The density remains between $\rhomin$ and $\rhomax$, and the solution is unique in the displayed class. Moreover,
\[
 \uf\in L^\infty(0,T_*;H^{m+2})
        \cap W^{1,\infty}(0,T_*;H^m),
 \qquad \pf\in L^\infty(0,T_*;H^{m+1}).
\]
\end{theorem}

\begin{proof}
A high-regularity invariant set is combined with contraction in the lower energy norm \eqref{eq:linear-energy-difference}. Since the initial jets $\boldsymbol U_0,\ldots,\boldsymbol U_N$ and $P_0,\ldots,P_{N-1}$ are smooth, the reference polynomials
\begin{equation}\label{eq:compatible-reference-polynomials}
 \bv^{\mathrm{ref}}(t)=\sum_{j=0}^N\frac{t^j}{j!}\boldsymbol U_j,
 \qquad
 p^{\mathrm{ref}}(t)=\sum_{j=0}^{N-1}\frac{t^j}{j!}P_j
\end{equation}
belong to the required spaces on $[0,\overline T]$. For $0<\tau\le\min\{1,\overline T\}$, let $C_0\ge\|\bv^{\mathrm{ref}}\|_{\mathbb V_{N,\tau}}$ be fixed by the data. For $R>2C_0$, define
\begin{equation}\label{eq:compatible-picard-set}
 \mathbb X_{\tau,R}
 :=\left\{\bv\in\mathbb V_{N,\tau}:
       \partial_t^j\bv(0)=\boldsymbol U_j\ (0\le j\le N),
       \quad\|\bv\|_{\mathbb V_{N,\tau}}\le R\right\}.
\end{equation}
The set $\mathbb X_{\tau,R}$ is nonempty since it contains $\bv^{\mathrm{ref}}$ whenever $R>C_0$.

For $\bv\in\mathbb X_{\tau,R}$, solve
\[
 \partial_t r_{\bv}+\bv\cdot\nabla r_{\bv}=0,
 \qquad r_{\bv}(0)=\rho_0.
\]
The flow is well defined, fixes the boundary and preserves volume. The transport estimate at order $s+1$ gives
\begin{equation}\label{eq:high-transport-bound}
 \|r_{\bv}(t)\|_{H^{s+1}}
 \le\|\rho_0\|_{H^{s+1}}
       \exp\!\left(C\int_0^t\|\bv(\sigma)\|_{H^{s+1}}
                                      \,\mathrm d\sigma\right)
 \le C_R,
 \qquad \rhomin\le r_{\bv}\le\rhomax.
\end{equation}
Here $\int_0^\tau\|\bv\|_{H^{s+1}}\,\mathrm dt\le\tau^{1/2}R$. Time differentiation gives the exact recursion
\[
 \partial_t^j r_{\bv}
 =-\sum_{\ell=0}^{j-1}\binom{j-1}{\ell}
       \partial_t^\ell\bv\cdot
       \nabla\partial_t^{j-1-\ell}r_{\bv},
 \quad 1\le j\le N.
\]
Starting from \eqref{eq:high-transport-bound}, Sobolev multiplication therefore yields
\begin{equation}\label{eq:nonlinear-coefficient-bounds}
\begin{gathered}
 r_{\bv}\in C H^{s+1},\qquad
 \sum_{j=1}^N\|\partial_t^j r_{\bv}\|_{C H^{s_j+2}}\le C_R,
 \qquad \theta_{\bv}=\Theta[\bv],\\
 a_{\bv}=\mu(\theta_{\bv}),\qquad
 \sum_{j=0}^N\|\partial_t^j a_{\bv}\|_{C H^{s_j+2}}\le C_R.
\end{gathered}
\end{equation}
For the viscosity bounds, use $\partial_t^j\theta_{\bv}=\Theta[\partial_t^j\bv]$ and the finite chain rule for $\mu$. The lowest order is $\partial_t^Na_{\bv}\in C H^3$, and $H^3(\Omega)$ embeds continuously into $W^{1,\infty}(\Omega)$. Thus these coefficients satisfy \eqref{eq:linear-coefficient-class}. Their initial derivatives are $R_j$ and $A_j$, independently of the choice of $\bv$ in the set.

Define $\mathscr S\bv=\boldsymbol u$ by
\begin{equation}\label{eq:effective-linearized-stokes}
 r_{\bv}\partial_t\boldsymbol u
 +\mathcal L_{a_{\bv}}\boldsymbol u+\nabla p
 =\boldsymbol h_{\bv},
 \qquad
 \boldsymbol h_{\bv}
 =r_{\bv}\bfv-\theta_{\bv}\bg
             -r_{\bv}(\bv\cdot\nabla)\bv,
\end{equation}
with $\diver\boldsymbol u=0$, homogeneous Dirichlet data, $\boldsymbol u(0)=\boldsymbol U_0$ and zero-mean pressure. The coefficient bounds and Sobolev product estimates place $\boldsymbol h_{\bv}$ in $\mathbb F_{N,\tau}$. By \eqref{eq:nonlinear-momentum-jets}, the linear compatibility recursion has precisely the specified jets $\boldsymbol U_0,\ldots,\boldsymbol U_N$. Hence \Cref{prop:linear-effective-stokes} defines $\mathscr S$ and preserves these jets.

The reference polynomials \eqref{eq:compatible-reference-polynomials} allow the high-order bound in \eqref{eq:compatible-picard-set} to be closed with a small-time factor. Set $\boldsymbol z=\mathscr S\bv-\bv^{\mathrm{ref}}$ and $\pi=p-p^{\mathrm{ref}}$. The pair $(\boldsymbol z,\pi)$ satisfies the corresponding linear equation with zero initial jets through order $N$ and forcing
\begin{equation}\label{eq:self-map-force-jets}
 \boldsymbol H_{\bv}
 :=r_{\bv}\bfv-\theta_{\bv}\bg-r_{\bv}(\bv\cdot\nabla)\bv
   -r_{\bv}\partial_t\bv^{\mathrm{ref}}
   -\mathcal L_{a_{\bv}}\bv^{\mathrm{ref}}
   -\nabla p^{\mathrm{ref}}.
\end{equation}
The jet identities give
\begin{equation}\label{eq:self-map-zero-jets}
 \partial_t^j\boldsymbol H_{\bv}(0)=\boldsymbol0,
 \qquad 0\le j<N.
\end{equation}
Differentiation of \eqref{eq:self-map-force-jets} gives the following bounds before taking a small-time factor:
\begin{equation}\label{eq:self-map-force-uniform}
 \max_{0\le j<N}
       \|\partial_t^j\boldsymbol H_{\bv}\|_{C H^{s_j-1}}
 +\|\partial_t^N\boldsymbol H_{\bv}\|_{C L^2}
 \le C_R.
\end{equation}
Indeed, the convection terms are exactly
\[
 \partial_t^j\bigl(r_{\bv}(\bv\cdot\nabla)\bv\bigr)
 =\sum_{\ell+k+n=j}\frac{j!}{\ell!k!n!}
       (\partial_t^\ell r_{\bv})
       (\partial_t^k\bv\cdot\nabla)\partial_t^n\bv.
\]
For $j<N$ the highest time index contributes one spatial derivative of an $H^{s_j}$ field, giving $H^{s_j-1}$. At $j=N$, at most one factor in each Leibniz term carries the top time derivative. If that factor is not spatially differentiated, it lies in $H^1$. When the top derivative falls on the advected velocity, its gradient lies in $L^2$. Every remaining factor is lower order and bounded. The differentiated stress is
\[
 \partial_t^j\bigl(\mathcal L_{a_{\bv}}\bv^{\mathrm{ref}}\bigr)
 =\sum_{\ell=0}^j\binom j\ell
       \mathcal L_{\partial_t^\ell a_{\bv}}
                        \partial_t^{j-\ell}\bv^{\mathrm{ref}},
\]
and has the same required bounds since the reference polynomials are smooth and \eqref{eq:nonlinear-coefficient-bounds} holds. All remaining products contain at most these derivative orders. In particular, $\partial_t^{N+1}\bv^{\mathrm{ref}}=0$, so \eqref{eq:self-map-force-uniform} requires no extra time derivative of the iterate.

For $j<N-1$, integrate $\partial_t^{j+1}\boldsymbol H_{\bv}$ in $H^{s_j-3}$ and interpolate with $H^{s_j-1}$. For $j=N-1$, use $\partial_t^N\boldsymbol H_{\bv}\in C L^2$ and interpolate with $H^2$. In both cases \eqref{eq:self-map-zero-jets} gives
\[
 \|\partial_t^j\boldsymbol H_{\bv}\|_{C H^{s_j-2}}
 \le C\|\partial_t^j\boldsymbol H_{\bv}\|_{C H^{s_j-3}}^{1/2}
       \|\partial_t^j\boldsymbol H_{\bv}\|_{C H^{s_j-1}}^{1/2}
 \le C_R\tau^{1/2}.
\]
For $0\le j<N$, the uniform $C H^{s_j-1}$ bound also gives
\[
 \|\partial_t^j\boldsymbol H_{\bv}\|_{L^2H^{s_j-1}}
 \le \tau^{1/2}C_R,
\]
while $\|\partial_t^N\boldsymbol H_{\bv}\|_{L^2L^2}\le\tau^{1/2}C_R$. These estimates account for every component of the norm in \eqref{eq:appendix-force-space}. Consequently
\begin{equation}\label{eq:self-map-high-closure}
 \|\boldsymbol H_{\bv}\|_{\mathbb F_{N,\tau}}
 \le C_R\tau^{1/2},
 \qquad
 \|\mathscr S\bv\|_{\mathbb V_{N,\tau}}
 \le C_0+C_R\tau^{1/2}.
\end{equation}
By \eqref{eq:self-map-high-closure}, choose $R=2C_0+1$ and then $\tau$ so that $C_R\tau^{1/2}\le R-C_0$. The image of $\mathbb X_{\tau,R}$ lies in the same set.

Contraction is imposed only in the lower norm $\mathbb E_\tau$, for which no higher boundary trace is required of the forcing term. Take $\bv_1,\bv_2\in\mathbb X_{\tau,R}$ and set
\[
 \boldsymbol d=\bv_1-\bv_2,
 \quad r_i=r_{\bv_i},\quad\theta_i=\Theta[\bv_i],
 \quad a_i=\mu(\theta_i),\quad
 \boldsymbol u_i=\mathscr S\bv_i,
 \quad\boldsymbol w=\boldsymbol u_1-\boldsymbol u_2.
\]
The density difference solves
\[
 \partial_t(r_1-r_2)+\bv_1\cdot\nabla(r_1-r_2)
 =-\boldsymbol d\cdot\nabla r_2,
 \qquad (r_1-r_2)(0)=0.
\]
Since $\diver\bv_1=0$ and $\bv_1|_{\partial\Omega}=0$, multiplication by $r_1-r_2$ gives
\[
 \frac12\frac{\mathrm d}{\mathrm dt}\|r_1-r_2\|_{L^2}^2
 =-\int_\Omega(\boldsymbol d\cdot\nabla r_2)(r_1-r_2)\,\mathrm d\bx
 \le \|\nabla r_2\|_{L^\infty}
       \|\boldsymbol d\|_{L^2}\|r_1-r_2\|_{L^2}.
\]
Hence, using $(r_1-r_2)(0)=0$ and the uniform bound $\|\nabla r_2\|_{L^\infty}\le C_R$,
\[
 \|r_1-r_2\|_{C([0,\tau];L^2)}
 \le C_R\int_0^\tau\|\boldsymbol d(t)\|_{L^2}\,\mathrm dt
 \le C_R\tau\|\boldsymbol d\|_{L^\infty(0,\tau;L^2)}.
\]
Combining this estimate with \eqref{eq:appendix-theta-bound} gives
\begin{equation}\label{eq:low-coefficient-differences}
\begin{gathered}
 \|r_1-r_2\|_{C L^2}
 \le C_R\tau\|\boldsymbol d\|_{L^\infty L^2},\quad
 \|\theta_1(t)-\theta_2(t)\|_{H^2}
 +\|a_1(t)-a_2(t)\|_{L^2}
 \le C_R\|\boldsymbol d(t)\|_{L^2}.
\end{gathered}
\end{equation}
With the coefficients $(r_1,a_1)$, the equation for $\boldsymbol w$ has zero initial value and force
\begin{equation}\label{eq:energy-contraction-force}
 \begin{aligned}
 \boldsymbol H_{12}={}&
 -(r_1-r_2)\partial_t\boldsymbol u_2
 -\mathcal L_{a_1-a_2}\boldsymbol u_2
 +(r_1-r_2)\bfv-(\theta_1-\theta_2)\bg\\
 &-(r_1-r_2)(\bv_1\cdot\nabla)\bv_1
 -r_2\bigl((\boldsymbol d\cdot\nabla)\bv_1
                   +(\bv_2\cdot\nabla)\boldsymbol d\bigr).
 \end{aligned}
\end{equation}
Every term in \eqref{eq:energy-contraction-force} has an $H^{-1}$ bound by $C_R(\|r_1-r_2\|_{L^2}+\|\boldsymbol d\|_{L^2})$. In particular, the stress term satisfies
\[
 \|\mathcal L_{a_1-a_2}\boldsymbol u_2\|_{H^{-1}}
 \le C_R\|a_1-a_2\|_{L^2}.
\]
For the last convection term, testing by $\bphi\in H_0^1$ gives
\[
 \begin{aligned}
 \int_\Omega r_2(\bv_2\cdot\nabla)\boldsymbol d\cdot\bphi\,\mathrm d\bx
 ={}&-\int_\Omega r_2\boldsymbol d\cdot
                        (\bv_2\cdot\nabla)\bphi\,\mathrm d\bx-\int_\Omega(\bv_2\cdot\nabla r_2)
                          \boldsymbol d\cdot\bphi\,\mathrm d\bx.
 \end{aligned}
\]
Since $s\ge7$, the three-dimensional Sobolev embedding supplies the required $L^\infty$ bounds for $\partial_t\boldsymbol u_2$, $\nabla r_2$ and the lower-order velocity derivatives. Together with $\tau\le1$, \eqref{eq:low-coefficient-differences} gives
\[
 \|\boldsymbol H_{12}\|_{L^2H^{-1}}
 \le C_R\tau^{1/2}\|\boldsymbol d\|_{\mathbb E_\tau}.
\]
The energy estimate \eqref{eq:linear-energy-difference} now yields
\begin{equation}\label{eq:effective-contraction}
 \|\mathscr S\bv_1-\mathscr S\bv_2\|_{\mathbb E_\tau}
 \le C_R\tau^{1/2}\|\bv_1-\bv_2\|_{\mathbb E_\tau}.
\end{equation}

By \eqref{eq:effective-contraction}, decrease $\tau$ further, if necessary, so that $\gamma=C_R\tau^{1/2}<1$, and start the iteration from $\bv^{(0)}=\bv^{\mathrm{ref}}$. For $n\ge0$, let $p^{(n+1)}$ denote the zero-mean pressure delivered by \Cref{prop:linear-effective-stokes} together with $\bv^{(n+1)}=\mathscr S\bv^{(n)}$, and set $p_j^{(n+1)}=\partial_t^j p^{(n+1)}$. Then
\[
 \|\bv^{(n+1)}-\bv^{(n)}\|_{\mathbb E_\tau}
 \le\gamma^n\|\bv^{(1)}-\bv^{(0)}\|_{\mathbb E_\tau}.
\]
Hence, for some $\uf$,
\[
 \bv^{(n)}\to\uf
 \quad\text{strongly in }\mathbb E_\tau.
\]
The uniform $\mathbb V_{N,\tau}$ bound also yields, after extraction,
\[
 \begin{aligned}
 \partial_t^j\bv^{(n)}&\rightharpoonup^*\partial_t^j\uf
 &&\text{in }L^\infty(0,\tau;H^{s_j}),\qquad 0\le j\le N,\\
 \partial_t^j\bv^{(n)}&\rightharpoonup\partial_t^j\uf
 &&\text{in }L^2(0,\tau;H^{s_j+1}),\qquad 0\le j\le N,\\
 \partial_t^{N+1}\bv^{(n)}&\rightharpoonup\partial_t^{N+1}\uf
 &&\text{in }L^2(0,\tau;L^2).
 \end{aligned}
\]
Distributional differentiation identifies the displayed weak limits. Moreover,
\[
 \partial_t^j\bv^{(n)}\rightharpoonup\partial_t^j\uf
 \quad\text{in }H^1(0,\tau;L^2),
 \qquad 0\le j\le N,
\]
where the case $j=N$ uses the $L^2L^2$ bound for $\partial_t^{N+1}\bv^{(n)}$. Continuity of the trace operator on $H^1(0,\tau;L^2)$ therefore gives
\[
 \partial_t^j\uf(0)=\boldsymbol U_j,
 \qquad 0\le j\le N.
\]
Weak and weak-* lower semicontinuity preserve the $L^\infty H^{s_j}$, $L^2H^{s_j+1}$ and top-order $L^2L^2$ bounds, with total norm at most $R$. These bounds, together with the strong convergence in $\mathbb E_\tau$, provide the compactness required to identify the nonlinear limit.

Interpolation of the energy convergence with the uniform $H^s$ bound gives strong convergence of the velocities in $L^\infty(0,\tau;H^{s-\eta})$ for $\eta>0$. Similarly, \eqref{eq:low-coefficient-differences} and the high transport bound give strong convergence of the densities in $L^\infty(0,\tau;H^{s+1-\eta})$. In particular,
\[
 r_{\bv^{(n)}}\to r_{\uf}\quad\text{in }L^\infty L^2,
 \qquad
 \Theta[\bv^{(n)}]\to\Theta[\uf]\quad\text{in }L^\infty H^2,
\]
where the second convergence follows directly from \eqref{eq:appendix-theta-bound} with $k=0$ and the $L^\infty L^2$ convergence of $\bv^{(n)}$. Since $H^2(\Omega)$ is an algebra in three dimensions and embeds into $L^\infty(\Omega)$, smooth composition gives
\[
 a_{\bv^{(n)}}=\mu(\Theta[\bv^{(n)}])
 \to a_{\uf}:=\mu(\Theta[\uf])
 \quad\text{strongly in }L^\infty(0,\tau;H^2(\Omega))
 \cap L^\infty((0,\tau)\times\Omega).
\]
The contraction also gives
$\cD(\bv^{(n+1)})\to\cD(\uf)$ strongly in $L^2((0,\tau)\times\Omega)$. Therefore
\[
 a_{\bv^{(n)}}\cD(\bv^{(n+1)})
 \to a_{\uf}\cD(\uf)
 \quad\text{strongly in }L^2((0,\tau)\times\Omega),
\]
and hence the variable-viscosity stress converges in $L^2(0,\tau;H^{-1})$ after taking one divergence. The strong velocity and density convergences, together with the uniform high-order bounds, likewise give
$r_{\bv^{(n)}}(\bv^{(n)}\cdot\nabla)\bv^{(n)}
\to r_{\uf}(\uf\cdot\nabla)\uf$ in distributions. By \eqref{eq:linear-pressure-hierarchy}, the pressure sequence is uniformly bounded in the corresponding spaces. After passing to a subsequence, there exist fields $p_j$ such that, for $0\le j<N$,
\[
 p_j^{(n)}\rightharpoonup^* p_j
 \quad\text{in }L^\infty(0,\tau;H^{s_j-1}(\Omega)),
 \qquad
 p_j^{(n)}\rightharpoonup p_j
 \quad\text{in }L^2(0,\tau;H^{s_j}(\Omega)),
\]
while $p_N^{(n)}\rightharpoonup p_N$ in $L^2(0,\tau;H^1(\Omega))$. Passing to the distributional identities $\partial_t p_j^{(n)}=p_{j+1}^{(n)}$ gives $p_j=\partial_t^j p_0$. Passing to the momentum equation then identifies $p_0$ with the zero-mean pressure associated with $\uf$. The zero-mean normalization is preserved under the weak convergence.

The weak limit satisfies
\[
 \partial_t^j\uf\in L^2H^{s_j+1}\cap H^1H^{s_j-1}
       \quad(j<N),
 \qquad
 \partial_t^N\uf\in L^2H^2\cap H^1L^2.
\]
The Lions--Magenes time-continuity theorem \cite[Proposition~2.1, p.~18, and Theorem~3.1, p.~19]{LionsMagenes} gives $\partial_t^j\uf\in C([0,\tau];H^{s_j})$ for $j\le N$. For these continuous representatives, the $C H^{s_j}$ norm equals the corresponding $L^\infty H^{s_j}$ norm. The lower-semicontinuous bounds and the trace identities therefore show that $\uf\in\mathbb X_{\tau,R}\subset\mathbb V_{N,\tau}$ as defined in \eqref{eq:appendix-velocity-space}. In particular, $\mathscr S\uf$ is well defined.

Since $\bv^{(n+1)}=\mathscr S\bv^{(n)}$ and both consecutive sequences have the same energy-norm limit, these convergences permit passage to the limit in every term of \eqref{eq:effective-linearized-stokes}. The limiting equation is precisely the equation defining $\mathscr S\uf$, and uniqueness for the linear problem gives $\mathscr S\uf=\uf$. Since $r_{\uf}$ satisfies $\partial_t r_{\uf}+\diver(r_{\uf}\uf)=0$, the identity
\[
 \partial_t(r_{\uf}\uf)+\diver(r_{\uf}\uf\otimes\uf)
 =r_{\uf}\bigl(\partial_t\uf+(\uf\cdot\nabla)\uf\bigr)
  +\uf\bigl(\partial_t r_{\uf}+\diver(r_{\uf}\uf)\bigr)
 =r_{\uf}\bigl(\partial_t\uf+(\uf\cdot\nabla)\uf\bigr)
\]
shows directly that the constructed momentum equation has the conservative form stated in \eqref{eq:effective-intro}. The transport equation supplies $r_{\uf}\in C H^{s+1}$, and the coefficient recursion verifies \eqref{eq:linear-coefficient-class} for the limit. Applying the linear proposition now recovers the pressure regularity in \eqref{eq:effective-high-regularity}. The characteristic flow preserves the density bounds.

We finish by proving uniqueness in the displayed class. Let
$(r_i,\boldsymbol u_i,\theta_i,p_i)$, $i=1,2$, be two solutions with the same initial data, and set
\[
 \delta r=r_1-r_2,\qquad
 \boldsymbol d=\boldsymbol u_1-\boldsymbol u_2,\qquad
 \delta\theta=\theta_1-\theta_2,\qquad
 \delta p=p_1-p_2.
\]
The transport and temperature differences satisfy
\[
 \partial_t\delta r+\boldsymbol u_1\cdot\nabla\delta r
 =-\boldsymbol d\cdot\nabla r_2,
 \qquad
 -\kf\Delta\delta\theta=\boldsymbol d\cdot\bg,
 \qquad \partial_{\bn}\delta\theta=0,\qquad
 \int_\Omega\delta\theta\,\mathrm d\bx=0.
\]
Consequently, the regularity of the two solutions gives
\begin{equation}\label{eq:effective-uniqueness-coefficients}
 \frac{\mathrm d}{\mathrm dt}\|\delta r\|_{L^2}^2
 \le C\bigl(\|\delta r\|_{L^2}^2+\|\boldsymbol d\|_{L^2}^2\bigr),
 \qquad
 \|\delta\theta\|_{H^2}
 +\|\mu(\theta_1)-\mu(\theta_2)\|_{L^\infty}
 \le C\|\boldsymbol d\|_{L^2}.
\end{equation}
Set
\[
 \boldsymbol A_2=\partial_t\boldsymbol u_2
                 +(\boldsymbol u_2\cdot\nabla)\boldsymbol u_2,
 \qquad a_i=\mu(\theta_i).
\]
Using the nonconservative form of the two regular momentum equations, their difference is
\begin{equation}\label{eq:effective-uniqueness-momentum}
\begin{aligned}
 r_1\bigl(\partial_t\boldsymbol d
          +(\boldsymbol u_1\cdot\nabla)\boldsymbol d\bigr)
 +\mathcal L_{a_1}\boldsymbol d+\nabla\delta p
 =-\delta r\,\boldsymbol A_2
     -r_1(\boldsymbol d\cdot\nabla)\boldsymbol u_2
     -\mathcal L_{a_1-a_2}\boldsymbol u_2
 +\delta r\,\bfv-\delta\theta\,\bg.
\end{aligned}
\end{equation}
Indeed,
$(\boldsymbol u_1\cdot\nabla)\boldsymbol u_1
 -(\boldsymbol u_2\cdot\nabla)\boldsymbol u_2
 =(\boldsymbol u_1\cdot\nabla)\boldsymbol d
  +(\boldsymbol d\cdot\nabla)\boldsymbol u_2$.
Since $\partial_t r_1+\diver(r_1\boldsymbol u_1)=0$, testing
\eqref{eq:effective-uniqueness-momentum} by $\boldsymbol d$ gives the exact weighted kinetic identity
\[
 \int_\Omega r_1\bigl(\partial_t\boldsymbol d
          +(\boldsymbol u_1\cdot\nabla)\boldsymbol d\bigr)
          \cdot\boldsymbol d\,\mathrm d\bx
 =\frac12\frac{\mathrm d}{\mathrm dt}
   \int_\Omega r_1|\boldsymbol d|^2\,\mathrm d\bx.
\]
The coercivity of $\mathcal L_{a_1}$ and Korn's inequality control
$\|\boldsymbol d\|_{H^1}^2$. The remaining terms satisfy, for $\eta>0$,
\[
 \begin{aligned}
 &\left|\int_\Omega
   \bigl[-\delta r\,\boldsymbol A_2
         -r_1(\boldsymbol d\cdot\nabla)\boldsymbol u_2
         +\delta r\,\bfv-\delta\theta\,\bg\bigr]
       \cdot\boldsymbol d\,\mathrm d\bx\right|
\le C\bigl(\|\delta r\|_{L^2}^2
                    +\|\boldsymbol d\|_{L^2}^2\bigr),\\
 &\left|\langle\mathcal L_{a_1-a_2}\boldsymbol u_2,
                     \boldsymbol d\rangle\right|
 \le C\|a_1-a_2\|_{L^2}\|\boldsymbol d\|_{H^1}
 \le \eta\|\boldsymbol d\|_{H^1}^2
      +C_\eta\|\boldsymbol d\|_{L^2}^2,
 \end{aligned}
\]
where \eqref{eq:effective-uniqueness-coefficients} and the displayed regularity of the two solutions were used. Choosing $\eta$ small and adding the density estimate in \eqref{eq:effective-uniqueness-coefficients} yield
\begin{equation}\label{eq:effective-uniqueness-energy}
 \frac{\mathrm d}{\mathrm dt}
 \left(\|\delta r\|_{L^2}^2
       +\int_\Omega r_1|\boldsymbol d|^2\,\mathrm d\bx\right)
 +c\|\boldsymbol d\|_{H^1}^2
 \le C\left(\|\delta r\|_{L^2}^2
       +\int_\Omega r_1|\boldsymbol d|^2\,\mathrm d\bx\right).
\end{equation}
Here $c>0$, and $C$ depends only on the bounded norms of the two solutions on $[0,T_*]$. The initial difference vanishes, so Gronwall's lemma applied to \eqref{eq:effective-uniqueness-energy} gives $\delta r=\boldsymbol d=0$ on $[0,T_*]$. The Neumann problem then gives $\delta\theta=0$, and the momentum equation together with the zero-mean pressure normalization gives $\delta p=0$. Thus the solution is unique in the asserted class.

Finally, $s-2=2N-1\ge m\ge5$, so the three-dimensional embedding $H^m\hookrightarrow W^{2,\infty}$ and the density bounds imply \Cref{ass:strong}. Taking $T_*=\tau$ completes the construction.
\end{proof}

\begin{remark}[Relation to the comparison hypothesis]\label{rem:effective-regularity}
Since $m\ge5$, the three-dimensional Sobolev embedding applied to the regularity stated in \Cref{thm:effective-local} verifies \Cref{ass:strong} on every interval $[0,T]$ with $0<T\le T_*$. The construction uses only finitely many norms of the smooth data.
\end{remark}

\begin{remark}[Compatibility and nonemptiness]\label{rem:effective-compatibility}
The first condition in \eqref{eq:nonlinear-jet-compatibility} is the usual no-slip condition for the initial acceleration. The subsequent conditions concern higher time jets determined by the data. For $m=5$, the construction uses $N=3$ and requires the full traces of $\boldsymbol U_1,\boldsymbol U_2,\boldsymbol U_3$ to vanish. These assumptions admit nontrivial time-dependent forces. For example, choose $\boldsymbol u_0=0$, any smooth positive $\rho_0$ and $\bfv(t,\bx)=t^N\boldsymbol b(\bx)$ with an independently chosen smooth field $\boldsymbol b$. Indeed, $\boldsymbol F_j=0$ for $0\le j<N$. If $\boldsymbol U_0,\ldots,\boldsymbol U_j$ vanish, then $R_\ell=0$ for $1\le\ell\le j+1$, $\Theta_\ell=0$ for $0\le\ell\le j$, and the recursion gives $\boldsymbol K_j=0$. Hence $\rho_0\boldsymbol U_{j+1}+\nabla P_j=0$, $\diver\boldsymbol U_{j+1}=0$ and $\boldsymbol U_{j+1}\cdot\bn=0$. Equivalently, $P_j$ solves the homogeneous weighted Neumann problem $\diver(\rho_0^{-1}\nabla P_j)=0$ with $\rho_0^{-1}\nabla P_j\cdot\bn=0$. The zero-mean normalization yields $P_j=0$ and therefore $\boldsymbol U_{j+1}=0$. Induction gives $\boldsymbol U_1=\cdots=\boldsymbol U_N=0$. The force is fixed in advance, and \Cref{thm:effective-local} produces the corresponding solution.
\end{remark}

\section{Global microscopic weak solutions}
\label{sec:micro-existence}

We prove \Cref{prop:microscopic-existence} here. Fix $\eps>0$ and write $D=\Oe$. Constants arising only in the Galerkin compactness argument may depend on $D$, whereas the energy and thermal estimates retain the uniformity of \Cref{lem:uniform}. We use the variable-density framework developed in \cite{Lions} and treat the Neumann transmission problem explicitly.

\begin{proof}[Proof of \Cref{prop:microscopic-existence}]
Let $H$ be the $L^2(D;\R^3)$ closure of the smooth compactly supported solenoidal fields, and put $V=H_{0,\sigma}^1(D)$. By the normal-trace characterization \cite[Theorem~III.2.3]{Galdi}, the assumptions in \eqref{eq:initial-data-assumptions} imply $\uzeroe\in H$. The argument uses the density theorem \cite[Theorem~III.4.2]{Galdi} and the de Rham pressure characterization \cite[Theorem~III.5.3]{Galdi}.

For $\bv\in H$, denote its zero extension by $\widetilde\bv$ and define $\cS_\eps\bv\in H^1(\Omega)$ variationally by
\begin{equation}\label{eq:micro-thermal-operator}
 \int_\Omega\ke\nabla\cS_\eps\bv\cdot\nabla\varphi\,\mathrm d\bx
 =\int_\Omega\widetilde\bv\cdot\bg\,\varphi\,\mathrm d\bx,
 \qquad \int_\Omega\cS_\eps\bv\,\mathrm d\bx=0,
\end{equation}
for any $\varphi\in H^1(\Omega)$. Since $\bg=\nabla F$ and the zero extension is solenoidal, the right-hand side vanishes for $\varphi\equiv1$. Lax--Milgram on the zero-mean space defines a bounded linear operator. Using Poincar\'e--Wirtinger and \Cref{lem:neumann-linf-general}, we obtain
\begin{equation}\label{eq:micro-thermal-lipschitz}
 \|\cS_\eps\bv-\cS_\eps\bw\|_{H^1(\Omega)}
 +\|\cS_\eps\bv-\cS_\eps\bw\|_{L^\infty(\Omega)}
 \le C\|\bv-\bw\|_{L^2(D)}.
\end{equation}
In addition, testing \eqref{eq:micro-thermal-operator} by $\cS_\eps\bv$ gives
\[
 \int_D(\cS_\eps\bv)\bg\cdot\bv\,\mathrm d\bx
 =\int_\Omega\ke|\nabla\cS_\eps\bv|^2\,\mathrm d\bx.
\]

By \cite[Theorem~III.4.2]{Galdi}, choose a countable subset of $C_{c,\sigma}^\infty(D;\R^3)$ that is dense in $V$ in the $H^1$ norm. After discarding linear dependencies, apply Gram--Schmidt with respect to the $H$ inner product. Since each finite span is unchanged by Gram--Schmidt, this produces an $H$-orthonormal sequence $\{\boldsymbol w_i\}_{i\ge1}\subset C_c^\infty(D;\R^3)$ of solenoidal fields whose span remains dense in $V$. Let $\mathbb P_n$ be the $H$-orthogonal projection onto $V_n=\operatorname{span}\{\boldsymbol w_1,\ldots, \boldsymbol w_n\}$. Choose smooth densities $r_n^0$ with $\rhomin\le r_n^0\le\rhomax$ and $r_n^0\to\rhozeroe$ in every finite $L^p(D)$. Set $\boldsymbol u_n(0)=\mathbb P_n\uzeroe$. For $\boldsymbol u_n(t)\in V_n$, transport $r_n$ by
\[
 \partial_t r_n+\boldsymbol u_n\cdot\nabla r_n=0,
 \qquad r_n(0)=r_n^0,
 \qquad \theta_n=\cS_\eps\boldsymbol u_n,
\]
and impose, for $1\le i\le n$,
\begin{equation}\label{eq:micro-galerkin}
\begin{aligned}
 \frac{\mathrm d}{\mathrm dt}\int_D
 r_n\boldsymbol u_n\cdot\boldsymbol w_i\,\mathrm d\bx
 ={}&\int_D r_n\boldsymbol u_n\otimes\boldsymbol u_n:
                       \nabla\boldsymbol w_i\,\mathrm d\bx-\int_D2\mu(\theta_n)\cD(\boldsymbol u_n):
                       \cD(\boldsymbol w_i)\,\mathrm d\bx\\
 &+\int_D(r_n\bfv-\theta_n\bg)\cdot\boldsymbol w_i\,\mathrm d\bx.
\end{aligned}
\end{equation}
The smooth finite-dimensional flow preserves the density bounds. The mass matrix $\mathsf M_{n,ij}(t)=\int_Dr_n\boldsymbol w_i\cdot\boldsymbol w_j \,\mathrm d\bx$ therefore satisfies $\rhomin I\le\mathsf M_n(t)\le\rhomax I$. For fixed $n$ and bounded coefficient curves, the transported smooth density depends locally Lipschitz-continuously on the curve, with
\[
 \|r[\bv]-r[\bw]\|_{C([0,t];L^\infty(D))}
 \le C_{n,R}t\|\bv-\bw\|_{C([0,t];L^2(D))}.
\]
Together with \eqref{eq:micro-thermal-lipschitz} and the Lipschitz continuity of $\mu$, this yields local solvability of the finite-dimensional integral system by Picard iteration.

Testing \eqref{eq:micro-galerkin} by $\boldsymbol u_n$ and using the transport equation gives the exact identity
\begin{equation}\label{eq:micro-galerkin-energy}
\begin{aligned}
 &\frac12\int_Dr_n(t)|\boldsymbol u_n(t)|^2\,\mathrm d\bx
 +\int_0^t\!\int_D2\mu(\theta_n)|\cD(\boldsymbol u_n)|^2
                                      \,\mathrm d\bx\,\mathrm ds+\int_0^t\!\int_\Omega\ke|\nabla\theta_n|^2\,\mathrm d\bx\,\mathrm ds\\
  &\quad=\frac12\int_Dr_n^0|\mathbb P_n\uzeroe|^2\,\mathrm d\bx +\int_0^t\!\int_Dr_n\bfv\cdot\boldsymbol u_n \,\mathrm d\bx\,\mathrm ds.
\end{aligned}
\end{equation}
Using the Poincar\'e--Korn and Young estimates from \Cref{lem:uniform}, whose constants are independent of the Galerkin dimension $n$, we obtain
\begin{equation}\label{eq:micro-galerkin-bounds}
 \|\boldsymbol u_n\|_{L^\infty(0,T;H)}
 +\|\boldsymbol u_n\|_{L^2(0,T;V)}
 +\|\theta_n\|_{L^\infty(0,T;H^1(\Omega))}
 \le C\bigl(1+\|\uzeroe\|_{L^2(D)}\bigr).
\end{equation}
For fixed $n$, these bounds control every velocity coefficient on a finite time interval. The flow estimate $\|\nabla r_n(t)\|_\infty\le\|\nabla r_n^0\|_\infty \exp(\int_0^t\|\nabla\boldsymbol u_n\|_\infty\,\mathrm ds)$ then prevents finite-time loss of the density regularity needed in the approximate system. Thus the approximation extends to $[0,T]$.

To recover kinetic-energy convergence through the momentum coordinates, we obtain strong compactness of the density first, without presupposing strong convergence of the velocity. After extraction,
\[
 r_n\rightharpoonup^\ast r\quad\text{in }L^\infty((0,T)\times D),
 \qquad
 \boldsymbol u_n\rightharpoonup\boldsymbol u
                  \quad\text{in }L^2(0,T;V).
\]
Since $V$ is a closed linear subspace of $H_0^1(D;\R^3)$, it is weakly closed. Hence, the limit remains divergence free and has zero trace. Testing the continuity equation against $H^1(D)$ functions gives
\[
 \|\partial_t r_n\|_{L^\infty(0,T;(H^1(D))^*)}\le C.
\]
The compact embedding $L^2(D)\hookrightarrow\!\hookrightarrow (H^1(D))^*$ and Arzel\`a--Ascoli imply
\begin{equation}\label{eq:micro-density-negative}
 r_n\to r\quad\text{in }C([0,T];(H^1(D))^*).
\end{equation}
For $\psi\in C_c^\infty((0,T)\times D)$, the product $\boldsymbol u_n\cdot\nabla\psi$ is bounded in $L^2(0,T;H^1(D))$. Consequently,
\[
 \left|\int_0^T\langle r_n-r,
              \boldsymbol u_n\cdot\nabla\psi\rangle\,\mathrm dt\right|
 \le C_\psi\|r_n-r\|_{C(H^1)^*}
                 \|\boldsymbol u_n\|_{L^2V}\to0.
\]
The remaining term passes by weak convergence of $\boldsymbol u_n$. Hence $\partial_t r+\diver(r\boldsymbol u)=0$ with initial value $\rhozeroe$.

Define the extensions to $\R^3$ by
\[
 \overline r(t,\bx)=
 \begin{cases}
 r(t,\bx),&\bx\in D,\\
 \rhomin,&\bx\in\R^3\setminus D,
 \end{cases}
 \qquad
 \overline{\boldsymbol u}(t,\bx)=
 \begin{cases}
 \boldsymbol u(t,\bx),&\bx\in D,\\
 \boldsymbol0,&\bx\in\R^3\setminus D.
 \end{cases}
\]
Then $\overline{\boldsymbol u}\in L^1(0,T;W^{1,2}(\R^3;\R^3))$ and $\diver\overline{\boldsymbol u}=0$ in distributions. The boundary cutoff argument used in \Cref{lem:density}, now applied on $D$, shows that
\[
 \partial_t\overline r+\diver(\overline r\,\overline{\boldsymbol u})=0
 \quad\text{in }\mathcal D'((0,T)\times\R^3).
\]
Indeed, if $d_D$ is the distance to $\partial D$ and $\chi_h$ cuts off a strip of width $h$, Hardy's inequality gives
\[
 \left|\int_0^T\!\int_Dr\boldsymbol u\cdot\nabla\chi_h\,
                  \psi\,\mathrm d\bx\,\mathrm dt\right|
 \le C h^{1/2}\rhomax\|\psi\|_\infty
                       \|\nabla\boldsymbol u\|_{L^1L^2}\to0.
\]
Extend $\overline r$ and $\overline{\boldsymbol u}$ to $t<0$ by $\overline r(t)=\overline{\rhozeroe}$ and $\overline{\boldsymbol u}(t)=\boldsymbol0$, where $\overline{\rhozeroe}=\rhozeroe$ in $D$ and $\overline{\rhozeroe}=\rhomin$ in $\R^3\setminus D$. The distributional transport equation then carries the initial value at $t=0$. Applying the renormalization theory of \cite[Section~II]{DiPernaLions} to the compactly supported fluctuation $\overline r-\rhomin$ yields $r\in C([0,T];L^p(D))$ for $1\le p<\infty$, the density bounds and
\[
 \|r(t)\|_{L^2(D)}^2=\|\rhozeroe\|_{L^2(D)}^2,
 \qquad
 \|r_n(t)\|_{L^2(D)}^2=\|r_n^0\|_{L^2(D)}^2.
\]
Together with \eqref{eq:micro-density-negative}, equality of these norms implies
\begin{equation}\label{eq:micro-density-strong}
 r_n\to r\quad\text{in }C([0,T];L^p(D)),
 \qquad 1\le p<\infty.
\end{equation}
Uniformity in time for $p=2$ follows by considering any sequence $t_n\to t$. Equation \eqref{eq:micro-density-negative} gives $r_n(t_n)\rightharpoonup r(t)$ in $L^2$, while $\|r_n(t_n)\|_{L^2(D)}\to\|r(t)\|_{L^2(D)}$. Thus $r_n(t_n)\to r(t)$ strongly. A compactness argument in $[0,T]$ and the continuity of $r$ give uniform convergence. The other finite exponents follow from the uniform $L^\infty$ bounds.

Strong velocity compactness is then recovered through momentum coordinates. For fixed $i$ and $n\ge i$, put
\[
 m_{n,i}(t)=\int_Dr_n\boldsymbol u_n\cdot\boldsymbol w_i\,\mathrm d\bx.
\]
Equation \eqref{eq:micro-galerkin} and \eqref{eq:micro-galerkin-bounds} give $\|m_{n,i}\|_{W^{1,2}(0,T)}\le C_i$. A diagonal subsequence therefore satisfies $m_{n,i}\to m_i$ in $C([0,T])$. By \eqref{eq:micro-density-strong} and weak velocity convergence,
\[
 m_i(t)=\int_Dr\boldsymbol u\cdot\boldsymbol w_i\,\mathrm d\bx
 \quad\text{for a.e. }t\in(0,T).
\]
The compact embedding $V\hookrightarrow\!\hookrightarrow H$ gives numbers $\omega_K\to0$ such that
\[
 \|(I-\mathbb P_K)\bv\|_H\le\omega_K\|\bv\|_V
 \qquad\text{for any }\bv\in V.
\]
Indeed, $\mathbb P_K\to I$ strongly on $H$ and the projections are contractions, so the convergence is uniform on the compact closure in $H$ of the unit ball of $V$. It follows that
\[
 \left|\int_0^T\!\int_Dr_n\boldsymbol u_n\cdot
                  (I-\mathbb P_K)\boldsymbol u_n\,\mathrm d\bx\,\mathrm dt\right|
 \le C\omega_K.
\]
For fixed $K$, finite-dimensional expansion gives
\[
 \int_0^T\!\int_Dr_n\boldsymbol u_n\cdot \mathbb P_K\boldsymbol u_n
                         \,\mathrm d\bx\,\mathrm dt
 =\sum_{i=1}^K\int_0^Tm_{n,i}(t)
                  (\boldsymbol u_n(t),\boldsymbol w_i)_H\,\mathrm dt
 \to\sum_{i=1}^K\int_0^Tm_i(t)
                  (\boldsymbol u(t),\boldsymbol w_i)_H\,\mathrm dt.
\]
The same projection-tail estimate holds for the weak limit:
\[
 \left|\int_0^T\!\int_D r\boldsymbol u\cdot(I-\mathbb P_K)\boldsymbol u\,\mathrm d\bx\,\mathrm dt\right|
 \le C\omega_K.
\]
Consequently, first letting $n\to\infty$ at fixed $K$ and then $K\to\infty$ proves convergence of the integrated kinetic energies. In addition, $r_n\boldsymbol u\to r\boldsymbol u$ in $L^2((0,T)\times D)$ by the almost everywhere convergence of $r_n$, their uniform $L^\infty$ bound and dominated convergence. Hence
\[
 \int_0^T\!\int_D r_n|\boldsymbol u_n-\boldsymbol u|^2
 =\int_0^T\!\int_D r_n|\boldsymbol u_n|^2
 -2\int_0^T\!\int_D r_n\boldsymbol u_n\cdot\boldsymbol u
 +\int_0^T\!\int_D r_n|\boldsymbol u|^2\to0,
\]
where the middle term converges by weak convergence of $\boldsymbol u_n$ against the strongly convergent coefficient $r_n\boldsymbol u$. Therefore
\begin{equation}\label{eq:micro-velocity-strong}
 \rhomin\|\boldsymbol u_n-\boldsymbol u\|_{L^2((0,T)\times D)}^2
 \le\int_0^T\!\int_Dr_n|\boldsymbol u_n-\boldsymbol u|^2
                         \,\mathrm d\bx\,\mathrm dt\to0.
\end{equation}
The momentum coordinates and the projection-tail estimate therefore provide the compactness needed for the convective term.

It remains to pass to the limit in the thermal coupling and the energy inequality. By \eqref{eq:micro-thermal-lipschitz} and \eqref{eq:micro-velocity-strong},
\[
 \theta_n\to\theta:=\cS_\eps\boldsymbol u
       \quad\text{in }L^2(0,T;H^1(\Omega))
                    \cap L^2(0,T;L^\infty(\Omega)).
\]
Thus $\mu(\theta_n)\to\mu(\theta)$ in $L^2L^\infty$ and a.e.\ in $(0,T)\times\Omega$, with uniform positive upper and lower bounds. In particular, for any $\bphi\in C_c^\infty((0,T)\times D;\R^3)$ with $\diver\bphi=0$,
\[
 \bigl(\mu(\theta_n)-\mu(\theta)\bigr)\cD(\bphi)
 \to0
 \quad\text{in }L^2((0,T)\times D),
\]
so the viscous term in the weak momentum equation passes to the limit against fixed tests. After passing to a subsequence converging a.e.\ in $(0,T)\times\Omega$, dominated convergence against any fixed $L^2$ test field also shows
\[
 \sqrt{\mu(\theta_n)}\,\cD(\boldsymbol u_n)
 \rightharpoonup\sqrt{\mu(\theta)}\,\cD(\boldsymbol u)
                   \quad\text{in }L^2((0,T)\times D).
\]
Moreover,
\[
 r_n\boldsymbol u_n\otimes\boldsymbol u_n
 \to r\boldsymbol u\otimes\boldsymbol u
 \quad\text{in }L^1((0,T)\times D).
\]
These convergences are sufficient to pass to the limit in every term of \eqref{eq:micro-galerkin}. To extend the limit identity from finite sums to the stated test class, note that
\[
 \boldsymbol u\otimes\boldsymbol u\in L^1(0,T;L^3(D)),
 \qquad L^2(D)\hookrightarrow L^{3/2}(D).
\]
Thus the convection functional is continuous for test fields in $C([0,T];V)$, while the momentum time term is continuous when their time derivatives converge in $L^1(0,T;H)$. More precisely, for divergence-free $\bphi\in C_c^1([0,T)\times D;\R^3)$, the density of $\operatorname{span}\{\boldsymbol w_i\}$ in $V$ and standard Bochner approximation provide finite sums
\[
 \bphi_m(t)=\sum_{i=1}^{N_m}\eta_{m,i}(t)\boldsymbol w_i,
 \qquad \eta_{m,i}\in C_c^1([0,T)),
\]
such that
\[
 \bphi_m\to\bphi\quad\text{in }C([0,T];V),
 \qquad
 \partial_t\bphi_m\to\partial_t\bphi
 \quad\text{in }L^1(0,T;H).
\]
In particular, $\bphi_m(0)\to\bphi(0)$ in $H$. The convection, viscous, forcing, momentum-time and initial terms are continuous under these convergences, so passing $m\to\infty$ yields \eqref{eq:weak-momentum} for the stated test class. Renormalization yields \eqref{eq:renormalized-continuity}, and \eqref{eq:micro-thermal-operator} gives \eqref{eq:weak-heat}. Since the extended velocity vanishes in the holes, the fixed-domain density extension equals $\rhomin$ there.

The initial energies in \eqref{eq:micro-galerkin-energy} converge since
\[
 r_n^0\to\rhozeroe
 \quad\text{in every finite }L^p(D),
 \qquad
 \mathbb P_n\uzeroe\to\uzeroe
 \quad\text{in }H.
\]
After passing to a further subsequence, the strong space--time convergence yields strong convergence of the corresponding time slices for a.e. $t\in(0,T)$. In addition,
\[
 \sqrt{\mu(\theta_n)}\,\cD(\boldsymbol u_n)
 \rightharpoonup
 \sqrt{\mu(\theta)}\,\cD(\boldsymbol u)
 \quad\text{in }L^2((0,T)\times D),
\]
and therefore weak lower semicontinuity gives, for a.e. $t\in(0,T)$,
\[
 \int_0^t\!\int_D2\mu(\theta)|\cD(\boldsymbol u)|^2
 \,\mathrm d\bx\,\mathrm ds
 \le\liminf_{n\to\infty}
 \int_0^t\!\int_D2\mu(\theta_n)|\cD(\boldsymbol u_n)|^2
 \,\mathrm d\bx\,\mathrm ds.
\]
The strong convergence $\theta_n\to\theta$ in $L^2(0,T;H^1(\Omega))$ also gives
\[
 \sqrt{\ke}\,\nabla\theta_n
 \to
 \sqrt{\ke}\,\nabla\theta
 \quad\text{strongly in }L^2((0,T)\times\Omega),
\]
and hence
\[
 \sup_{0\le t\le T}
 \left|
 \int_0^t\!\int_\Omega
 \ke\bigl(|\nabla\theta_n|^2-|\nabla\theta|^2\bigr)
 \,\mathrm d\bx\,\mathrm ds
 \right|\to0.
\]
Together with convergence of the initial energy and the work term, this yields \eqref{eq:combined-energy}.

It remains to identify a weakly continuous velocity representative. For $\boldsymbol w\in H$ and Galerkin level $n$, the density bound and \eqref{eq:micro-galerkin-bounds} give, uniformly in $n$ and for a.e. $t\in(0,T)$,
\begin{equation}\label{eq:approximate-momentum-dual-bound}
 \left|\int_D r_n(t)\boldsymbol u_n(t)\cdot\boldsymbol w\,\mathrm d\bx\right|
 \le\rhomax\norm{\boldsymbol u_n(t)}{H}\norm{\boldsymbol w}{H}
 \le C\norm{\boldsymbol w}{H}.
\end{equation}
If $\boldsymbol w$ is a finite sum of the Galerkin basis fields, the left-hand side is a continuous momentum coordinate. Hence \eqref{eq:approximate-momentum-dual-bound} extends from a.e. $t$ to every $t\in[0,T]$ for such $\boldsymbol w$. For a finite sum $\boldsymbol w=\sum_{i=1}^N a_i\boldsymbol w_i$, the uniform convergence of the momentum coordinates therefore gives
\[
 \left|\sum_{i=1}^N a_i m_i(t)\right|
 \le C\norm{\boldsymbol w}{H},
 \qquad t\in[0,T].
\]
Hence the limiting momentum coordinates define, for each $t$, a bounded functional on the dense subspace $\operatorname{span}\{\boldsymbol w_i\}$ and thus extend uniquely to a bounded functional on $H$. Since $r\in C([0,T];L^p(D))$ for $1\le p<\infty$ and $\rhomin\le r\le\rhomax$, the bilinear form
\[
 (\bv,\bw)\longmapsto\int_Dr(t)\bv\cdot\bw \,\mathrm d\bx
\]
is bounded and coercive on $H$. The Lax--Milgram theorem therefore gives a unique representative $\boldsymbol u(t)\in H$ with a bound uniform in $t$. If $t_k\to t$, this uniform $H$ bound permits extraction of a subsequence $\boldsymbol u(t_k)\rightharpoonup\boldsymbol v$ in $H$, while $r(t_k)\to r(t)$ strongly in every finite $L^p(D)$. Passing to the continuous momentum coordinates gives
\[
 \int_D r(t)\boldsymbol v\cdot\boldsymbol w_i\,\mathrm d\bx
 =m_i(t)
 =\int_D r(t)\boldsymbol u(t)\cdot\boldsymbol w_i\,\mathrm d\bx,
 \qquad1\le i\le N.
\]
Density of $\operatorname{span}\{\boldsymbol w_i\}$ in $H$ and $r(t)\ge\rhomin$ imply $\boldsymbol v=\boldsymbol u(t)$. Hence the whole sequence converges weakly, and
\[
 t\longmapsto\boldsymbol u(t)
 \quad\text{is weakly continuous from }[0,T]\text{ into }H.
\]
The same argument at $t=0$, using the initial momentum coordinates, gives $\boldsymbol u(0)=\uzeroe$. The pressure is recovered in the sense of distributions by the de Rham characterization \cite[Theorem~III.5.3]{Galdi}. Consequently, $(r,\boldsymbol u,\theta)$ satisfies all requirements of the weak-solution class in \Cref{sec:setting}, completing the proof of \Cref{prop:microscopic-existence}.
\end{proof}


\begin{thebibliography}{99}
\small
\setlength{\itemsep}{2pt}

\bibitem{AllaireI} G.~Allaire, Homogenization of the Navier--Stokes equations in open sets perforated with tiny holes. I. Abstract framework, a volume distribution of holes, \emph{Arch. Ration. Mech. Anal.} \textbf{113} (1991), no.~3, 209--259, \doi{10.1007/BF00375065}.

\bibitem{AllaireII} G.~Allaire, Homogenization of the Navier--Stokes equations in open sets perforated with tiny holes. II.\@ Noncritical sizes of the holes for a volume distribution and a surface distribution of holes, \emph{Arch. Ration. Mech. Anal.} \textbf{113} (1991), no.~3, 261--298, \doi{10.1007/BF00375066}.

\bibitem{BO2022} P.~Bella and F.~Oschmann, Homogenization and low Mach number limit of compressible Navier--Stokes equations in critically perforated domains, \emph{J. Math. Fluid Mech.} \textbf{24} (2022), no.~3, Art.~79, \doi{10.1007/s00021-022-00707-1}.

\bibitem{BOP} D.~Basari\'c, F.~Oschmann and J.~Pan, Qualitative derivation of a density-dependent incompressible Darcy law, \emph{Sci. China Math.} \textbf{69} (2026), 1--24, \doi{10.1007/s11425-025-2540-2}.

\bibitem{CBSV} T.~Crin-Barat, S.~\v Skondri\'c and A.~Violini, Relative energy method for weak--strong uniqueness of the inhomogeneous Navier--Stokes equations far from vacuum, \emph{J. Evol. Equ.} \textbf{25} (2025), Art.~6, \doi{10.1007/s00028-024-01036-8}.

\bibitem{ChoeKim} H.~J.~Choe and H.~Kim, Strong solutions of the Navier--Stokes equations for nonhomogeneous incompressible fluids, \emph{Comm. Partial Differential Equations} \textbf{28} (2003), no.~5--6, 1183--1201, \doi{10.1081/PDE-120021191}.

\bibitem{ChoKim} Y.~Cho and H.~Kim, Unique solvability for the density-dependent Navier--Stokes equations, \emph{Nonlinear Anal.} \textbf{59} (2004), no.~4, 465--489, \doi{10.1016/j.na.2004.07.020}.

\bibitem{DFL} L.~Diening, E.~Feireisl and Y.~Lu, The inverse of the divergence operator on perforated domains with applications to homogenization problems for the compressible Navier--Stokes system, \emph{ESAIM Control Optim. Calc. Var.} \textbf{23} (2017), no.~3, 851--868, \doi{10.1051/cocv/2016016}.

\bibitem{DiPernaLions} R.~J.~DiPerna and P.-L.~Lions, Ordinary differential equations, transport theory and Sobolev spaces, \emph{Invent. Math.} \textbf{98} (1989), 511--547, \doi{10.1007/BF01393835}.

\bibitem{FLS} E.~Feireisl, Y.~Lu and Y.~Sun, Homogenization of a non-homogeneous heat conducting fluid, \emph{Asymptot. Anal.} \textbf{125} (2021), no.~3--4, 327--346, \doi{10.3233/ASY-201658}.

\bibitem{Galdi} G.~P.~Galdi, \emph{An Introduction to the Mathematical Theory of the Navier--Stokes Equations: Steady-State Problems}, 2nd ed., Springer, New York, 2011, \doi{10.1007/978-0-387-09620-9}.

\bibitem{HJ2024} R.~M.~H\"ofer and J.~Jansen, Convergence rates and fluctuations for the Stokes--Brinkman equations as homogenization limit in perforated domains, \emph{Arch. Ration. Mech. Anal.} \textbf{248} (2024), Art.~50, \doi{10.1007/s00205-024-01993-x}.

\bibitem{Hoefer2023} R.~M.~H\"ofer, Homogenization of the Navier--Stokes equations in perforated domains in the inviscid limit, \emph{Nonlinearity} \textbf{36} (2023), no.~11, 6020--6047, \doi{10.1088/1361-6544/acfe56}.

\bibitem{JLP} W.~Jing, Y.~Lu and C.~Prange, Unified quantitative analysis of the Stokes equations in dilute perforated domains via layer potentials, \emph{Multiscale Model. Simul.} \textbf{23} (2025), no.~3, 1145--1182, \doi{10.1137/24M1696925}.

\bibitem{Lions} P.-L.~Lions, \emph{Mathematical Topics in Fluid Mechanics. Vol.~1: Incompressible Models}, Oxford Lecture Series in Mathematics and its Applications, vol.~3, Clarendon Press, Oxford, 1996.

\bibitem{LionsMagenes} J.-L.~Lions and E.~Magenes, \emph{Non-Homogeneous Boundary Value Problems and Applications. Vol.~I}, Grundlehren der mathematischen Wissenschaften, vol.~181, Springer-Verlag, Berlin--Heidelberg--New York, 1972, \doi{10.1007/978-3-642-65161-8}.

\bibitem{Lu2020} Y.~Lu, Homogenization of Stokes equations in perforated domains: a unified approach, \emph{J. Math. Fluid Mech.} \textbf{22} (2020), Art.~44, \doi{10.1007/s00021-020-00506-6}.

\bibitem{LPW2026} Y.~Lu, J.~Pan and L.~Wang, Homogenization of a non-homogeneous incompressible heat-conducting fluid in perforated domains, \emph{arXiv preprint} arXiv:2607.11036 [math.AP] (2026), \doi{10.48550/arXiv.2607.11036}.

\bibitem{Pan2025} J.~Pan, Homogenization of non-homogeneous incompressible Navier--Stokes system in critically perforated domains, \emph{J. Math. Fluid Mech.} \textbf{27} (2025), no.~2, Art.~24, \doi{10.1007/s00021-025-00931-5}.

\bibitem{SolonnikovStokes} V.~A.~Solonnikov, On estimates of solutions of the non-stationary Stokes problem in anisotropic Sobolev spaces and on estimates for the resolvent of the Stokes operator, \emph{Russian Math. Surveys} \textbf{58} (2003), no.~2, 331--365, \doi{10.1070/RM2003v058n02ABEH000613}.

\bibitem{Stampacchia} G.~Stampacchia, Le probl\`eme de Dirichlet pour les \'equations elliptiques du second ordre \`a coefficients discontinus, \emph{Ann. Inst. Fourier (Grenoble)} \textbf{15} (1965), no.~1, 189--257, \doi{10.5802/aif.204}.

\end{thebibliography}
\end{document}